\documentclass[12pt]{article}
\usepackage{cite,amsfonts,amsmath,amsbsy,amssymb,amsthm,setspace,mathtools}
\usepackage{graphicx}  
\usepackage{pb-diagram}
\newtheoremstyle{case}{}{}{}{}{}{:}{ }{}
\theoremstyle{case}

\newtheorem{axiom}{Axiom}

\newtheorem{theorem}{Theorem}
\newtheorem{observation}{Observation}
\newtheorem{definition}{Definition}
\newtheorem{lemma}{Lemma}
\newtheorem{corollary}{Corollary}

\newtheorem*{modified axiom 6}{Modified Axiom 6}
\newtheorem*{modified def 4}{Modified Definition 4}
\newtheorem*{modified axiom 31}{Modified Axiom 31}
\newtheorem*{modified thm 41}{Modified Theorem 41}
\newtheorem*{modified thm 42}{Modified Theorem 42}
\newtheorem*{observation def 18}{Observation about Definition 18}

\newcommand{\R}{\mathbb{R}}

\begin{document}  	
	
\title{A Finite, Feasible, Quantifier-Free Foundation for Constructive Geometry}
\author{John Burke}

\maketitle

\begin{abstract}
	
	In this paper we will develop an axiomatic foundation for the geometric study of straight edge, protractor, and compass constructions, which while being related to previous foundations such as \cite{EuclidElements}, \cite{HilbertFoundation}, \cite{TarskiSystem}, \cite{Beeson}, and \cite{PaArticle}, will be the first to have all axioms written and all proofs conducted in quantifier-free first order logic. All constructions within the system will be justified to be feasible by basic human faculties. No statement in the system will refer to infinitely many objects and one can posit an interpretation of the system which is in accordance to our free, creative process of geometric constructions. We are also able to capture analogous results to Euclid's work on non-planar geometry in Book XI of \textit{The Elements}. 
	
	This paper primarily builds on the work of Suppes in \cite{SuppesAxioms} and \cite{SuppesFinite} and draws from Beeson's work in \cite{Beeson}. By further developing Suppes' work on parallel line segments, we are able to develop analogs to most theorems about parallel lines without assuming an equivalent to Euclid's Fifth Postulate which we deem as introducing non-feasible constructions. In \cite{Beeson} Besson defines the characteristics such a geometric foundation should have to called constructive. This work satisfies these characteristics. Additionally this work would be considered constructive as Suppes defined it (see \cite{SuppesConstructive} and \cite{PaArticle}). 
\end{abstract}

\tableofcontents

\section{Introduction}

In \cite{SuppesAxioms} Suppes introduced a set of quantifier-free axioms for constructive affine planar geometry. The main features of this axiomatic system are a relation for when points are colinear and two constructions which had interpretations relating to the doubling of a line segment and finding the midpoint of a line segment. From these primitive concepts Suppes was able to define when two line segments are parallel. The great advantage of this axiomatic system is that from a finite set of points, to start with, one can prove theorems about the possible finite configurations one can create using the basic geometric tools of halving and doubling a line segment. Furthermore, this axiomatic system allows one to prove theorems about the parallel relation between line segments created from these constructions. It is compelling that from such basic tools one would be able to mathematically discuss concepts such as parallel line segments. Additionally, In \cite{SuppesFinite} Suppes was able to give a finite constructive model for the axioms in which points are given rational number coordinates. It is natural to ask, as Van Bendegem did in his article \textit{Finistism in Geometry} \cite{VBFinitism}, if this affine theory can be expanded all the way into a full-fledged geometrical theory. The objective of this work is to expand Suppes' theory to create a finite, feasible, constructive axiomatic theory which is an analog to Hilbert's axiomatic system presented in his work \textit{Grundlagen der Geometrie} (tr. \textit{The Foundations of Geometry}) \cite{HilbertFoundation}.

\subsection{Foundations of Geometry}

Geometry is one of the oldest form of mathematics. During the reign of Ptolemy I (323–283 BCE), an Alexandrian Greek name Euclid wrote a text called \textit{The Elements} which was one of the first well presented (and enduring) texts collecting the vast amount of geometric knowledge up to that time in his part of the world. The most important feature of the work was that all geometric statements were proved from a small list of assumptions. While almost all of the assumptions were `indisputable' truths about points, lines and angles in space and magnitudes of these objects, one of the assumptions called Euclid's fifth postulate, or the parallel postulate, was not as obviously true. The postulate states that if two lines are drawn which are intersected by a third line, called a transversal, in such a way that the sum of the measures of the two inner angles on the same side of the the transversal is less than two right angles, then the two lines inevitably must intersect each other on that side if extended far enough (see Figure \ref{euclidfifth1}). Once one has made relatively intuitive interpretations for the meaning of the geometric objects such as points, line and angles, one realizes that this postulate asserts a claim that is not realistically testable or observable in all instances. One can construct examples where the length of the line segment from $a$ to $b$, see Figure \ref{euclidfifth1}, would be larger than any magnitude one could construct or traverse. This is done by simply making the interior angles deviate from being supplemental by extraordinarily small amounts (see Figure \ref{euclidfifth2}). In fact this issue was pointed out by Proclus in his \textit{Commentary on Euclid's Elements} around 700 hundred years after Euclid lived. Proclus writes

\begin{quotation}
	 ``This [fifth postulate] ought even to be struck out of the Postulates altogether; for it is a theorem involving many difficulties which Ptolemy, in a certain book, set himself to solve, and it requires for the demonstration of it a number of definitions as well as theorems. And the converse of it is actually proved by Euclid himself as a theorem. It may be that some would be deceived and would think it proper to place even the assumption in question among the postulates ... . So in this case the fact that, when the right angles are lessened, the straight lines converge is true and necessary; but the statement that, since they converge more and more as they are produced, they will sometime meet is plausible but not necessary, in the absence of some argument showing that this is true in the case of straight lines. For the fact that some lines exist which approach indefinitely, but yet remain non-secant, although it seems improbable and paradoxical, is nevertheless true and fully ascertained with regard to other species of lines [for example curves like the hyperbola that has asymptotes]. May not then the same thing be possible in the case of straight lines that happens in the case of the lines referred to? Indeed, until the statement in the Postulate is clinched by proof, the facts shown in the case of other lines may direct our imagination the opposite way. And, though the controversial arguments against the meeting of the straight lines should contain much that is surprising, is there not all the more reason why we should expel from our body of doctrine this merely plausible and unreasoned (hypothesis)?''
 
\end{quotation}

It is evident that Euclid also had some amount of misgivings about this postulate given the fact that he choose to hold off using this assumption, when it would have been useful, early on in the text.

\begin{figure}[h!]
	\begin{picture}(216,120)
	\put(75,0){\includegraphics[scale=.9]{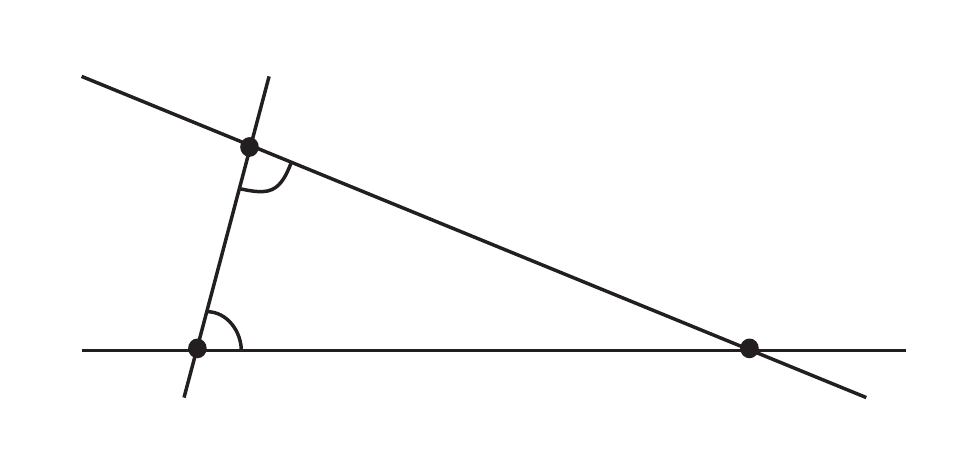}}
	\put(127,20){$a$}
	\put(269,32){$b$}
	\end{picture}
	\caption{}
	\label{euclidfifth1}
\end{figure}

\begin{figure}[h!]
	\begin{picture}(216,120)
	\put(75,0){\includegraphics[scale=.9]{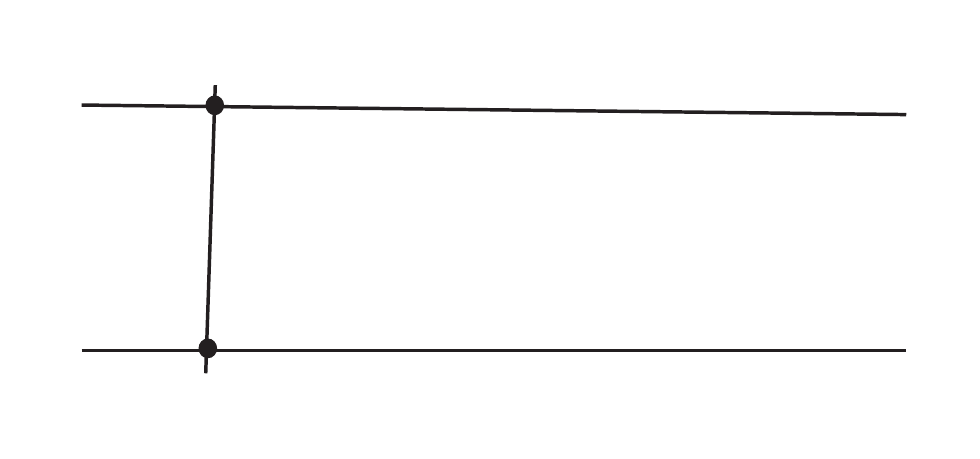}}

	\end{picture}
	\caption{Two non-parallel lines whose point of intersection will be very far to the right of the transversal.}
	\label{euclidfifth2}
\end{figure}

In 1890's Hilbert developed a set of axioms for Euclidean geometry. Unlike Euclid, Hilbert started with undefined terms such as points, lines and angles and defined relations that these terms can satisfy between each other. He then listed axioms which these terms and relations must satisfy. Hilbert also used a more modern and less constructive axiom about parallel lines. His axiom simply stated that given a line, $\ell$, and a point, $p$, not on that line, there is one and only one line, $\ell'$, through $p$ that does not intersect the original line. Using classical logic one can prove that this statement is equivalent to Euclid's fifth postulate. It is important to note that Hilbert also added an axiom of completeness to his set of assumptions. This axiom dealt with the inability to make extensions point-wise of any line. From this Hilbert was able to show that the only set-theoretic model of his axiomatic system was the Real Cartesian plain. 

Later in the 1920's Tarski developed a set of axioms for what he called elementary geometry. Important characteristics of his axiomatic set was that the only undefined term was `point'. All linear structure was dealt with via the colinear relation among points. Angle congruence was supplanted with only speaking about a triple of points as a triangle and using a clever 5-segment axiom instead of the more familiar Side-Angle-Side triangle congruence. 

In \cite{Beeson}, Beeson developed constructive versions of Tarski's axiomatic system some of these versions had quantifier-free axioms while others did not. The two main works that this current work is influenced by are \cite{Beeson} by Beeson and \cite{SuppesAxioms} by Suppes. They both developed axioms for constructive geometry. It is important to point out here that they are using the term constructive geometry in two very different ways. Suppes would define constructive axioms for geometry as a set of quantifier-free axioms that have operations which have intuitive interpretations as geometric constructions such as finding the intersection of two non parallel lines or finding the midpoint of a line segment. In fact, Suppes along with Moler in \cite{SuppesConstructive} were the first to development such a set of geometric axioms. Many others have done related work and Pambuccian summed up these works in \cite{PaArticle}. Meanwhile Beeson would define a constructive theory (which potentially could be axiomatized using constructions/functions instead of quantifiers) as a theory where the method of proof using this axioms involves using only intuitionistic logic. 

\subsection{Models and Geometric Spaces}

When considering the viable models to interpret these axiomatic systems one can create set-theoretic models in which the points are interpreted as tuples of numbers called coordinates. In \cite{SuppesFinite} Suppes defines a constructive model in which one starts with a finite number of coordinates and then is able to give coordinates to the points which are produced by applying constructions to elements of the original collection. This model is finite because one is only allowed to produce one new point at a time. Additionally, each formula in the quantifier-free language only discusses finitely many terms at a time. Some of Beeson's axiomatic systems are quantifier-free while others are not. Note that a geometric system containing an axiom with an existential quantifier which asserts the existence of a point between any two points would only have set theoretic models which contain infinitely many points. Beeson in fact only considers infinite set-theoretic models for all of his axiomatic systems. It is important to point out that although some of Beeson's axiomatic system are quantifier free he never stipulates that the language of the theory is quantifier free. Thus he freely uses definitions that involve existential quantifiers.

In his text \textit{New Foundation for Physical Geometry: The Theory of Linear Structures} \cite{MaudlinBook}, Maudlin make a distinction between metaphorical and geometrical spaces. In the category of metaphorical spaces Maudlin would place topological groups such as the integers, the real number line, $\R^2$ and $\R^3$. In the category of geometrical spaces Maudlin would place 2 and 3 dimensional Euclidean space and Minkowski space as well as Riemannian geometries. This may at first seem like a nonsensical distinction to some mathematicians and physicists given that many would identify 2-dimensional Euclidean space with $\R^2$, but in fact $\R^2$ is a set-theoretic model of the theory of Euclidean space (using Hilbert's axioms for instance). $\R^2$ has more structure than Euclidean space. The point $(0,0)$ called the origin has an elevated importance among all of the points in the space $\R^2$. Given a point/element of $\R^2$ you can discuss properties of it such as if one of its coordinates is negative or positive or rational or irrational. One can perform arithmetical operations to two given points/elements of $\R^2$. The points of 2-dimensional Euclidean space have none of this structure. All points are intrinsically identical. The space is homogeneous and isotropic. Most physicists would agree that we live in a three dimensional euclidean space (note the three spacial dimensions of spacetime are not curved, to the best of our knowledge, hence space is euclidean), but most would not agree that we live in the set of all ordered triples of real numbers called $\R^3$. In geometrical spaces the structure between points is intrinsic, but for metaphorical spaces the structure between points is emergent from other arithmetic structure between elements. Maudlin says the the following: 
\begin{quotation}
	``Through most of the history of mathematics, mathematicians would never have thought of even comparing a fundamentally arithmetic (and set-theoretic) object such as $\R^2$ with a fundamental geometical object such as Euclidean space. So confusion between the two, which is now so prevalent, is relatively recent in historical vintage.''
\end{quotation}

He goes on to say:
\begin{quotation}
	``Even if it is possible to use a metaphorical space such a the space of ordered triples of real numbers to represent a geometrical space, this imposes a screen of mathematical representation between us and the object in which we are interested.''
\end{quotation}
Interestingly Hilbert also felt the need to circumvent the use of real numbers coordinates in his \textit{Grundlagen der Geometrie}. He wrote that ``The present geometrical investigation seeks to uncover which axioms, hypotheses or aids are necessary for the proof of a fact in elementary geometry ...'' Hilbert referred to this as ``purity in the methods of proof''. 

The ancient Greeks made this distinction between the rigorous study of the actual physical space we live in, Geometry, and the rigorous study of numbers, Artihmetic. They referred to geometrical magnitudes as \textit{megethos} and numbers as \textit{arithmos}. There is a fascinating history dating from the ancient Greeks up to Dedekind in the nineteenth century where the supremacy of geometry was supplanted by arithmetic culminating in the modern concept of a set and the foundational theory of sets. The works of \textit{Greek Mathematical Thought and the Origins of Algebra} \cite{KlineBookGreek} by Jacob Klien and \textit{Mathematical Thoughts from Ancient to Modern Times} \cite{KlineBook} by Morris Kline are wonderful resources to better understanding the ancient Greek philosophy of mathematics and how from the seventeenth through the nineteenth century the arithmetizing of space and all of mathematics took hold. 

As for this work, we will focus on the study of geometric space (not metaphorical space). We will claim that there is a significant and important distinction between geometry and arithmetic. We do not simply want to have a finite, feasible foundation for arithmetic and then attempt to arithmetize geometry with it. In fact the need to have arithmetic accomplish the feat of being a foundation to geometry might be seen as a reason for introducing extraordinarily nonfinitistic means to arithmetic (for instance the need for irrational numbers as distances between coordinatized points). Therefore we will not present any arithmetical models of our theory. The only model to be considered will be the real-world physical phenomenon of using tools to construct points and intuitive relations between these points.


\subsection{Feasible Constructions}

We hold that all construction must be feasible. By feasible we simply mean that the intuitive interpretation of the construction is executable (by simple tools) for all circumstances where it is applicable. We introduce seven undefined constructions in our system. We introduce a midpoint construction (like Suppes did) which produces a point that is understood to be equidistant and colinear with two given points. This construction is executable in the real-world by folding a string or straight edge whose ends are first held down at the original two points. We also introduce a segment extension construction which extends a given line segment in one of two directions by a length of some other given segment. This is executable using a marked straight edge. Once we have define `sides' of a line there is an angle transport (on same side of segment) construction. Given triples $abc$ and $def$ this construction produces a point $x$ such that angle $dex$ is congruent to $abc$, segment $ex$ is congruent to $bc$, and $x$ and $f$ are on the same side of $de$. (See Figure \ref{atofigure}) This is executable using a marked protractor and marked straight edge. In Section \ref{CompassConstructions}, we introduce a circle-circle intersection construction which produces a point of intersection between two circles if one circle contains a point inside and a point outside of the other circle. This construction is executable with a compass. We also introduce a construction called the crossbow construction. Given a line segment $bd$ and two other points $a$ and $c$ which are on opposite sides of $bd$ The construction produces a point between $a$ and $c$ which is also colinear with $b$ and $d$. (See Figure \ref{crossbowfigure}) This construction is also executable with a straight edge. Since the length of segment from $b$ to the constructed point is not known prior to the construction, one might worry that the length of the straight edge needed to construct the point might be many magnitudes greater than any constructed up to that point especially as $a$ and $c$ get farther and farther away from $b$. This concern is a serious one for a feasible foundation. Fortunately within our system it is possible to prove that the line segment from $b$ to this constructed point is `less than' $ba$ or $bc$ where `less than' is defined as it would be in Hilbert's system. The proof of this fact is sketched out in Section \ref{Parallelsection} Thus there is no worry of not being able to produce a straight edge with a suitable length to create the desired point of intersection with segment $ac$. Lastly, we introduce a construction we call the orthogonal construction which given three points $a$, $b$ and $c$, where $a \not= b$ and $c \not= b$, produces a point $o$ such that $oba$ and $obc$ are right angles  [see Figure \ref{}]. To demonstrate the feasibility of such a construction we designed a tool called the orthogonator which is sketched and discussed in section \ref{SoildGeo}.

\subsection{Undefined Constructions}

As in most axiomatic system where constructions take the place of existential quantifiers, there are circumstances where one could theoretically apply a construction to a collection of points in which the intuitive real-world construction would not apply. Two examples in our system would be to apply the crossbow construction where the points $a$ and $c$ are on the same side of $bd$ or to apply the angle transport construction where $f$ is colinear with $d$ and $e$ or where $a$, $b$, and $c$ are colinear (and thus do not form an angle in our system). Formally, in our theory, there are no statements that can be proved about constructions applied to these non-intuitive settings. Thus one could easily ignore these applications as being outside the scope of the theory and allow the constructions to be undefined in these instances. One could also choose to define the construction in some trivial unimportant way such as letting the crossbow construction in the application described above produce the point $b$. This later choice need only be made if one finds it philosophically important that all constructions are formally defined `everywhere'. Mathematically either choice will not affect the work contended in future sections. Having said that, our stance philosophically is to leave the construction undefined.

\subsection{Constructive Geometry}

We claim that our system is constructive in two ways. As has been mentioned already our axiomatic system and the theory developed from these axioms is completely in the formal language of quantifier-free first-order logic. In the place of existential quantifiers we will introduce functions which will be referred to as constructions since they have interpretations of physical constructions an individual can carry out in the physical-world. This implies that our system is constructive as stipulated by Suppes.

In \cite{Beeson} Beeson defined the properties a constructive theory of geometry should satisfy. He claimed that by including certain stability axioms the only significant difference between a classical theory of geometry and a constructive theory was that in a constructive theorem one is not able to make case distinctions when constructing the object in an existential claim. Thus one should always have uniform (case-free) constructions. We in fact only need uniform constructions to achieve our goals. Because of this, we claim, without stipulating the stability axioms, that our theory is constructive as Beeson defines it.

\subsection{Sides of a Line Segment and Angle Orientation}




Hilbert's system satisfy what is called the plane separation property. One defines two points not on a given line as being on the same side of that line if the segment between those two points does not intersect the given line. One then defines these points as being on opposite sides of that line if the segment between them does intersect the line. In a modern set-theoretic framework the property claims that all of the points of the plane not on a line can be separated into two equivalence classes called sides. One will note that our crossbow construction from earlier is related to these concepts. One will also notice that it would not be possible to smoothly transition the totality of these concepts over to a quantifier-free language. Having said that Tarski and others were able to define a relation of four points for when $a$ and $b$ are on the same or opposite side(s) of a segment $cd$  \cite{Beeson}. Both of these relations were defined using only the between relation and an existential quantifier. It is impressive that Tarski and others were able to avoid using universal quantifiers when defining sides of a line segment. Having said that we will not follow Tarski and others on this issue. Although the truth values of same/opposite side(s) relations can be determined by simply verifying the order relations between points, we have two issues with incorporating these relations. First we want our entire theory to be quantifier-free. Secondly incorporating these relations will take us too far away form a Hilbert-style theory. Hilbert was able to prove many theorems from only his order, incidence, and congruence axioms without relying on a parallel postulate. This part of the theory that excludes the parallel postulate is referred to as a Hilbert Plane. Tarski theory, on the other hand uses an axiom equivalent to the Euclid's fifth postulate and a line-circle continuity axiom (related to compass constructions) in order to obtain similar results. Given that we have major issues with traditional axioms about parallel lines and wish to hold off on involving compass construction, we have chosen to go a different rout. 

One way that we could mimic a Hilbert-style axiomatic system would be to simply introduce an undefined relation for two points being on the same side of a line segment as was done by Greenberg in \cite{GreenbergBook}. One could then include several axioms to insure that this relation had the necessary properties. We have philosophical issue with this method. Having an undefined `same side' relation might lead one to think we are relying on an intuition about how lines partition the plane into two disjoint regions. This intuition is not acceptable for our goals. We will therefore aim to develop the concept of sides of a line from a foundation more in line with our goals. We can all agree that an individual is able to determine if two points $a$ and $b$ are on the 'same side' of line segment $cd$. So how does one determine this. We conclude that the method actually boils down to some intuitive grasp of orientation and not the infinite plane being partitioned by the line through $c$ and $d$. Anyone who has studied elementary trigonometry has been exposed to concept that angles in the plane can be positively or negatively oriented. We thus introduce an undefined same orientation relation. Given two non-colinear ordered triples $abc$ and $def$ the relation has a truth value of true when both ordered angles are either positively oriented or both negatively oriented. By introducing axioms pertaining to this relation we are able to fully develop a relation of `same side' from which we are able to prove many analogous results to Hilbert. Additionally all of these assumptions only pertain to the orientation of angles in triangles or pertain to points involved in feasible constructions. We find this a more philosophically appealing foundation to build from. This is also of interest in a purely mathematical sense. We have not found any previous works where angle orientation is treated in synthetic geometry on its own or as it relates to sides of a line. 

\subsection{Distinction for Beeson's Work}

Suppes' approach formally codifies the everyday reality that humans only manipulate finite objects in finite ways. We claim that it may be possible to interpret Beeson's axiomatic systems which use skolem functions as being constructively finite in a similar way. Given this claim, one might suppose that a path to a finite, feasible, constructive foundation for synthetic geometry would be best forged by conducing a deeper study of Beeson's work. We hold that although some of Beeson's systems have the potential to be constructive and finite, in the sense that they only speak about finitely many objects and finitely many construction there are features of his systems that are not feasible. His use of axioms which are equivalent to Euclid's fifth postulate being a prime example. 

The current work is distinct from Beeson's work is three important ways. First we claim that constructions producing a point of intersection for two non-parallel line (segments) are not feasible. Thus we will build on the work of Suppes to codify the concept of parallel line segments without such a construction. Secondly we will be developing our system as an analog to Hilbert's axiomatic system. Thirdly, while some of Beeson's axiomatic systems have quantifier-free axioms, Beeson still uses definition which involve existential quantifiers. This means that the over all theory is not quantifier-free. Our system and the theory developed from it is completely in quantifier-free first-order logic as was Suppes'.   

\subsection{Our Parallel Axiom}

As the long history of failed attempts to prove Euclid's fifth postulate and the discovery of non-euclidean geometry at the end of the nineteenth century would convince you of, one must assume a somewhat non-trivial feature of space that is not derived from simpler concepts of order, incidence, congruence of segments and angles, and basic constructions using straightedge, compass, or protractor. We have chosen to follow the work of Suppes and assume that the midpoint construction satisfies the property of bisymmetry. Formally this axiom states that for any four points $a$, $b$, $c$, and $d$ we have 	$mid(mid(a,b),mid(c,d))=mid(mid(a,c),mid(bd))$ where $mid(-,-)$ is the midpoint construction. We would like to explain why we believe this axiom does not violate our desire for a finite, feasible, constructive foundation.  

It is easy to see why this axiom is finite and constructive given that it only discusses four points and a construction. So we must ask ourselves if it is feasible. As was stated early our only requirement for the quantifier-free theory to be feasible is for the constructions to have standard interpretation which can be performed by human faculties. Obviously the midpoint construction is feasible. Additionally we strive for the concepts and objects of our system be accessible in terms of constructions that can be performed. Unlike Euclid's fifth postulate which makes claims that can not be verified by human experience, for any four points one can actually verify if the midpoint construction satisfies the bisymmetry condition in that instance. Thus our theory makes a assumption about space which could actually be experimentally tested leading to probabilistic confidence in the theory while also being potentially falsifiable.


\subsection{Number of Axioms}

In total our system has 36 axioms. It is natural for quantifier-free axiomatic system to have more axioms than their first-order logic counterparts [See \cite{SuppesInfinitesimalAnalysis} for an example]. Traditionally Hilbert's axiom for the Hilbert plane are listed thirteen in total, but this is misleading since many of them would be presented as separate axioms in a more modern treatment. For example Hilbert's first congruence axiom discusses not only the existence of points on either sides of a particular point on given line, but also states that the line segment congruence relation is reflexive. In a modern system this would be two separate axioms. By our counting Hilbert's axiomatic system for the Hilbert plane totals 19 axioms. Furthermore if one chooses to have a plane separation axiom instead of a Pasch's axiom, as Greenberg has in \cite{GreenbergBook}, the total would go up to 20. Our system would take 29 axioms to codify an appropriate analog to the Hilbert plane.  This is surely more, but there are instances where having a system with the only objects being points is actually more efficient. For instance we have no need for an analog to Hilbert's first incidence axiom which states that any two points determine a line. Fortunately, our system only needs one additional axiom to extend our analog of the Hilbert plane to a analog of Hilbert's treatment of the Euclidean plane sans continuity axioms. (Note that Hilbert's axioms make no reference to compass constructions or circle continuity.) In the end we are pleased that our system is still able to be presented with such economy given that it is presented in such a minimalist and restrictive language.

\subsection{Outline of Paper}

Section 2 covers Hilbert's order axioms except for Pasch's axiom, Section 3 discusses Hilbert's incidence axioms, and Section 5 discusses his congruence axioms. In Section 5 we introduce the extension and angle transport constructions. In Section 4 we focus on Pasch's axiom and plane separation properties. In this section we introduce our crossbow construction and angle orientation relation. We define and prove properties of sides of a line and also prove an analog to Hilbert's line separation property. In Section 6 we point out results from Hilbert's theory which directly translate over to our system. In Section 7 we point out a few results which do not translate over as effectively. In particular we introduce a midpoint construction and uniform constructions for erecting a perpendicular segment and dropping a perpendicular segment. By the end of Section 7 we have fully develop an analog to Hilbert's theory of a Hilbert plane. In section 8 we introduce concept of parallel line segment developed by Suppes. We then expand on Suppes' affine plane to develop a theory for a flat geometric plane. It is in this section that we introduce our parallel axiom to take the place of traditional parallel postulates. Given that we do not have the full might of a traditional parallel postulate, there are results of Hilbert's pertaining to parallel lines for which we do not have analogs. Having said that, there are many results pertaining to parallel lines that we do have analogs of. For example we have analogs of the Alternate Interior Angle Theorem and its converse. We are also able to prove that parallel line segments are equidistant and that extensions of non-parallel lines either come closer together or move farther apart. Using these properties of parallel lines we are able to define and prove properties about convex quadrilaterals and also define a uniform construction for trisecting a segment. Lastly in this section we sketch out a proof of our earlier claim that our crossbow construction is feasible. In Section 9 we introduce a circle-circle intersection construction, we assume a version of circle-circle continuity, and prove uniform constructions for versions of line-circle continuity and segment-circle continuity. In Section 10 we define a relation for when four points are coplanar. We then develop analogs to many solid geometry results of Hilbert's and Euclid's. We also discuss the modifications needed to fully incorporate the planar results of the earlier sections into this new context. Lastly, we prove many results pertaining to sides of a plane.


\section{Axioms of Order}

Hilbert introduced an undefined \textit{between} relation for three points, said `$b$ is between $a$ and $c$', which has the interpretation of stating that $b$ is on the interior of the line segment with $a$ and $c$ as endpoints. We will introduce the undefined relation $B(a,b,c,)$. 


\subsection{Hilbert's Order Axiom 1}

Hilbert's first order axiom states that for three distinct points, $a$, $b$, and $c$, on a line, if $b$ is between $a$ and $c$, then $b$ is between $c$ and $a$. 

We thus introduce the following two axioms: 

\begin{axiom}\label{betweenuniquepoint}
	$B(a,b,c) \rightarrow a \not= b$ and $a \not= c$
\end{axiom}

\begin{axiom}\label{betweenreverse}
	$B(a,b,c) \rightarrow B(c,b,a)$
\end{axiom}

Note that from these two axioms one can also prove that if $b$ is between $a$ and $c$, then $b \not= c$. 

\subsection{Hilbert's Order Axiom 2} \label{sectionorderaxiom2}

Hilbert's second order axiom states that given two points $a$ and $b$ there exists a point $c$  such that $b$ is between $a$ and $c$ 

Our axiomatic system will include a construction for segment extension. This construction, denoted $ext(ab,cd)$, will be understood to be a point constructed so that $B(a,b,ext(ab,cd))$ and the line segment from $b$ to $ext(ab,cd)$ is congruent to the line segment from $c$ to $d$. (Line segment congruence will be defined in the next section.) There we will introduce an axiom which states these properties of this construction. 

For the time being, we can satisfy Hilbert's axiom by noting that $ext(ab,ab)$ will function as the needed point $c$. 

\subsection{Hilbert's Order Axiom 3}

Hilbert's third order axiom states that given three distinct points on a line, one and only one lies between the other two. We thus introduce the following axiom. 

\begin{axiom}\label{onlyonebetween}
	$\lnot(B(a,b,c) \wedge B(a,c,b))$
\end{axiom}

Note if $B(a,b,c)$, then $B(c,b,a)$. Therefore $\lnot B(a,c,b)$, $\lnot B(b,c,a)$, $\lnot B(c,a,b)$, and $\lnot B(b,a,c)$. 

\subsection{Hilbert's Order Axiom 4}

Hilbert's fourth order axiom is known as Pasch's axiom. We will be devoting all of section \ref{Pasch} to this topic as well as the closely related topic of plane separation. For now we will turn our attention to Hilbert's axioms of incidence. 
	
\section{Axioms of Incidence}

Hilbert's axiomatic system for planar geometry refers to two classes of objects (points and lines). His system would maybe best be formalized today by using the language of set theory and second order logic. To define the relations among these two types of objects Hilbert had two incidence axioms describing how points and lines can be 'on' and 'contain' each other. 

Our axiomatic system in the tradition of Tarski will only have one type of object (points) all aspects of linearity will have to be dealt with by a relation of colinearity among three points. 

We now give a definition for the colinearity relation, defined $L(-,-,-)$, of three points:

\begin{definition}
	$L(a,b,c) \equiv B(a,b,c) \lor B(a,c,b) \lor B(b,a,c)$
\end{definition}

Note by Axiom \ref{betweenreverse} , $a$, $b$, and $c$ are also colinear when $B(c,b,a)$, $B(b,c,a)$, and $B(c,a,b)$. By Axiom \ref{betweenuniquepoint}, if $L(a,b,c)$, then $a$, $b$, and $c$ are all distinct points. Lastly one can prove that if $L(a,b,c)$, then any permutation of $a$, $b$, and $c$ is still colinear. 

\subsection{Hilbert's Incidence Axiom 1}

Hilbert's first incidence axiom states for any two points there is a line that `contains' them. Again since we do not have the goal of formalizing lines as objects, we will not be introducing an analogous axiom. Our system will only have an intuitive interpretation of line segments with two endpoints.
	
\subsection{Hilbert's Incidence Axiom 2}

Hilbert's second incidence axiom states that for every two points there exists no more than one line that `contains' them. We will need a modified version of this axiom. In particular we will not want that points $a$, $b$, and $c$ are co-linear and points $a$, $b$, and $d$ are co-linear while $a$, $c$, and $d$ are not co-linear. 

Thus we have the following Axiom: 

\begin{axiom}\label{uniqueline}
$L(a,b,c) \wedge L(a,b,d) \rightarrow L(a,c,d)$ 
\end{axiom}

From this axiom one can prove that any distinct triple of $a$, $b$, $c$, or $d$ is colinear. Thus if two linear triples share two points then any (distinct) triple formed from those four points is colinear. 

\subsection{Hilbert's Incidence Axiom 3}

Hilbert's third incidence axiom states that every line contains at least two points and that there are at least three points not all on the same line. 

Note that co-linearity is defined among three distinct points. In section \ref{congruence} we will define congruence of lines segments (defined by two endpoints). Because of this, our system will not allow for the discussion of any linear structures for less than two points. 

In order to codify the second point we will be follow Beeson and Suppes by having an axiom that states that there are three distinct constants which are not colinear. First, though, we will defined a new relation which states three points are distinct and non-colinear.  

\begin{definition}
	$T(a,b,c) \equiv \lnot L(a,b,c) \wedge a \not= b \wedge b \not= c \wedge a \not= c$
\end{definition}

This definition simply states that $a$, $b$, and $c$ form a triangle.

Next we introduce the following axiom. 

\begin{axiom} \label{3noncolinearpoints}
	$ T(\alpha, \beta, \gamma)$
\end{axiom}






\section{Pasch's Axiom, Plane Separation, and Angle Orientation}\label{Pasch}

Hilbert's fourth order axiom is known as Pasch's axiom. It states that if $a$, $b$, and $c$ are three distinct non-colinear points and $\ell$ is a line that does not pass through any of the three points, but does intersect the line segment connecting $a$ and $b$, then $\ell$ contains a point between $b$ and $c$ or $a$ and $c$. In other words, if a line intersects the interior of one side of a triangle (and does not intersect any of the vertices), then it must interest the interior of exactly one of the two remaining sides. Figure \ref{paschaxiom} illustrates this concepts. Pasch's axiom assures that the axiomatic system is one of planar geometry and not a higher dimensional geometry.

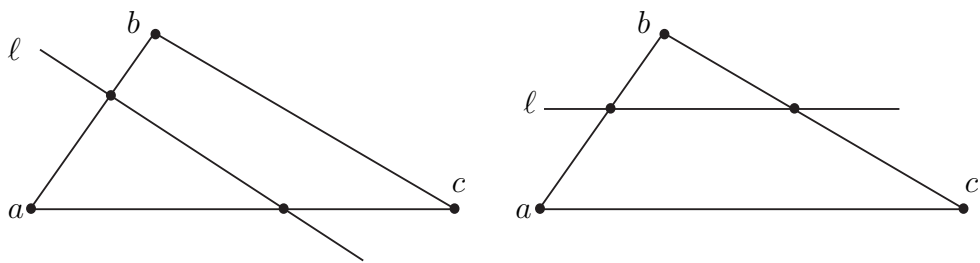
\begin{figure}[h!]
	\begin{picture}(216,130)
	\put(15,0){\includegraphics[scale=.7]{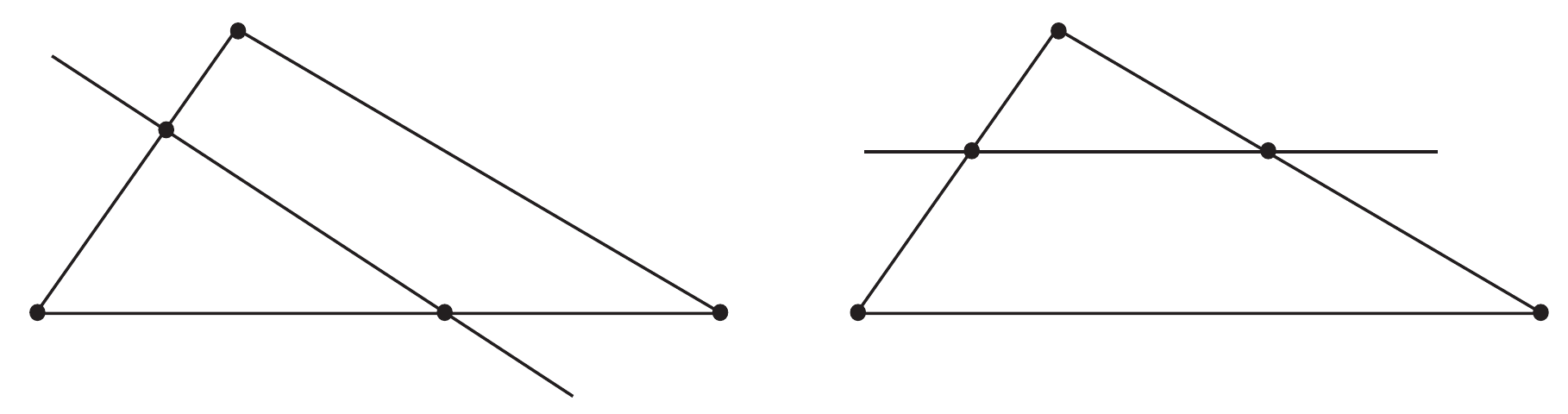}}
	\put(15,80){$\ell$}
	\put(210,60){$\ell$}
	\put(207,20){$a$}
	\put(15,20){$a$}
	\put(60,90){$b$}
	\put(253,90){$b$}
	\put(377,30){$c$}
	\put(183,30){$c$}
	\end{picture}
	\caption{Pasch's axiom}
	\label{paschaxiom}
\end{figure}

As the reader will note, that it is not a simple task to capture this axiom in the a formal language that does not allow for the mentioning of lines. Tarski's did this very thing by defining two subversions of Pasch's axiom called inner Pasch and outer Pasch. (The original version of Tarski's axioms assumed both inner and outer Pasch. It was later proved by Gupta in \cite{Gupta} that outer Pasch could be proved from inner Pasch.)

\subsection{Plane Separation}

There is a property of the Euclidean plane that states that every line divided the points on the plane not on the line into two sets called sides having the following properties: If $a$ and $b$ are on different sides of the line, then the line segment between $a$ and $b$ contains a point on the line and if $a$ and $b$ are on the same side of the line then the line segment between $a$ and $b$ contains no points of the line.

In some presentation of Hilbert's system (in \cite{GreenbergBook} for example) a plane separation axiom was assumed and Pasch's `axiom' was shown to be a result of this assumption. As has been stated, Hilbert assumed Pasch's axiom. In Hartshorne's text \cite{HartshorneBook} one can find a proof of the plane separation property based off of Hilbert's original system. The proof relies heavily on Pasch's axiom. The proof shows that the relation `same side' is an equivalence relation for the points not on a given line. He then shows that if points $a$ and $b$ are not on the same side of some line and $b$ and $c$ are also not on the same side then $a$ and $c$ are on the same side. This shows that there are only two equivalence classes.

Since Tarski's system does not allow for the discussion of sets or lines for that matter, some work must be done to translate this ideas over in an analogous ways. Since we will not be following Tarski et al in their methods, we will try to avoid going too deep into this topic. We will simply point out a few important details about their work. They define `opposite side' as a 4-ary relation. They then define `same side' by having a witness point on the other side of the line. One is then able to prove a plane separation theorem which states that if $a$ and $b$ are on the same side and $a$ and $c$ are on opposite sides then $b$ and $c$ are on opposite sides. See \cite{Beeson} for details.

We will be taking a very different track from Hilbert and Tarski et al. First, we will not be assuming any version of Pasch's axiom. We will instead introduce a relation whose interpretation is understood as two angles having the same orientation as well as one new construction. From axioms about this relation and construction we will be be able to prove analogous statements to Hilbert's plane separation axiom and Pasch's axiom. 


\subsection{Angle Orientation and the Crossbow Construction} \label{angleorientationsec}


We desire to formally capture humans' intuitive understanding of orientation of the plane without resorting to concepts such as an infinite line dividing the whole plane into two disjoint regions. We thus choose to have an undefined relation for when two non-colinear triples of points (two angles) have the same orientation. Since we are not able to discuss angles as Hilbert would, we will have an undefined 6-ary relation for two triples of points.

By $SameOrientation(a,b,c,d,e,f)$ we will interpret that the rotation from segment $ab$ to segment $bc$ has the same orientation as the rotation from segment $de$ to segment $ef$. See Figure \ref{sameorientation}. For ease of reading we will use the notation $SO(abc,def)$ instead of $SameOrientation(a,b,c,d,e,f)$. 

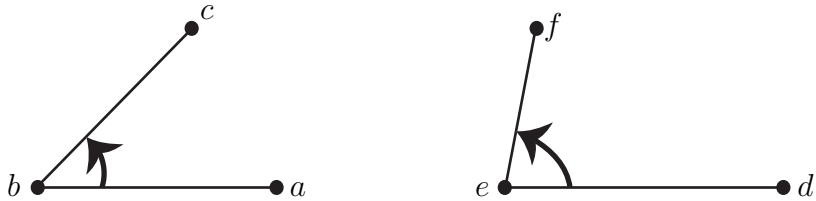
\begin{figure}[h!] 
		\begin{picture}(200,130)
		\put(35,0){\includegraphics[scale=1]{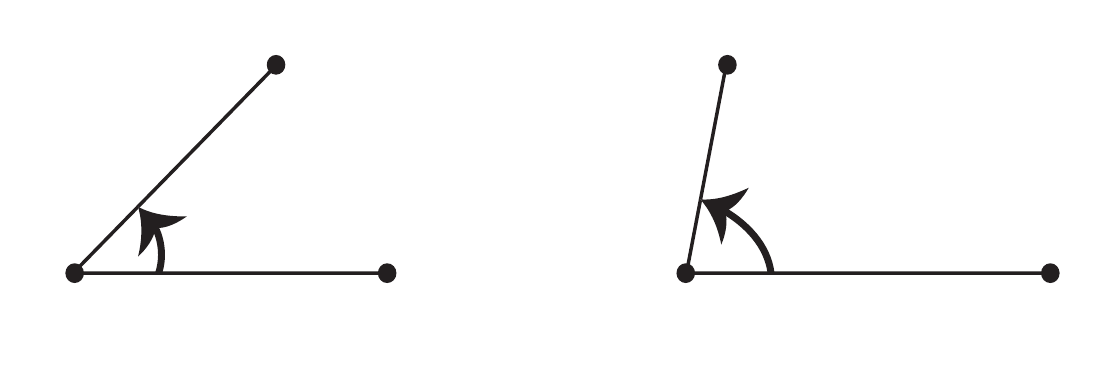}}
		\put(45,25){$b$}
	\put(152,25){$a$}
	
	\put(222,25){$e$}
	\put(344,25){$d$}
	\put(118,92){$c$}
	\put(248,86){$f$}
		\end{picture}
		\caption{Diagrammatic representation of angles $abc$ and $def$ having same orientation}
\label{sameorientation}
\end{figure}

We have quite a few axioms pertaining to this relation. nearly all have interpretations that are elementary about the relationship between line segments and angle orientation. 

\begin{axiom}\label{SOtriangle}
	$SO(abc,def) \rightarrow T(a,b,c) \wedge T(d,e,f)$
\end{axiom}

This axiom simply states that if two angles have the same orientation then they are both distinct non-colinear triples. This allows us to not have to deal with non-distinct or colinear triples. Much like Hilbert, we will only be discussing angles less than a straight angle. 

\begin{axiom}\label{SOreflexive}
	$SO(abc,abc)$
\end{axiom}

\begin{axiom}\label{SOtrans}
	$SO(abc,def) \wedge SO(abc,ghi) \rightarrow SO(abc,ghi)$
\end{axiom}

These two axioms assure that the relation is reflexive, symmetric, and transitive. 

We now define a relation, $SameDirection(a,b,c)$, shortened $SD(a,b,c)$, for when two points of a colinear triple are on the same side of the line in relation to the third point. This is analogous to Hilbert's concept of a ray.  

\begin{definition} \hspace{.1in}
	
	$SD(a,b,c) \equiv B(a,b,c) \lor B(a,c,b) \lor (a \not= b \wedge c=b) \lor (a=b \wedge c \not= b)$
\end{definition}

The following axioms states that extending or shortening a side of an angle does not change it's orientation. This takes the place of Hilbert's use of rays when defining angles. 

\begin{axiom}\label{SOray}
	$T(a,b,c) \wedge SD(b,a,e)  \rightarrow SO(abc, ebc)$
\end{axiom}

We now define two angles having opposite orientation to mean that they do not have the same orientation and are both distinct non-colinear triples. 

\begin{definition} \label{OOdef}
$OO(abc,def) \equiv \lnot SO(abc,def) \wedge T(a,b,c) \wedge T(d,e,f)$ [See Figure \ref{oppositeorientation}.]
\end{definition}

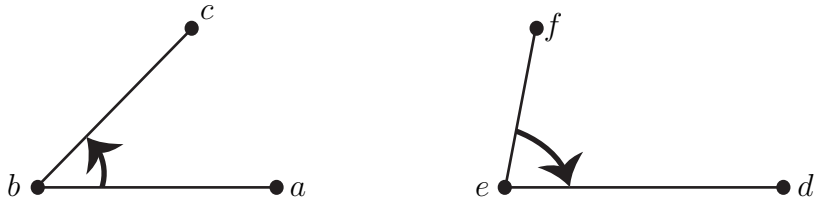
\begin{figure}[h!]
	\begin{picture}(216,130)
	\put(35,0){\includegraphics[scale=1]{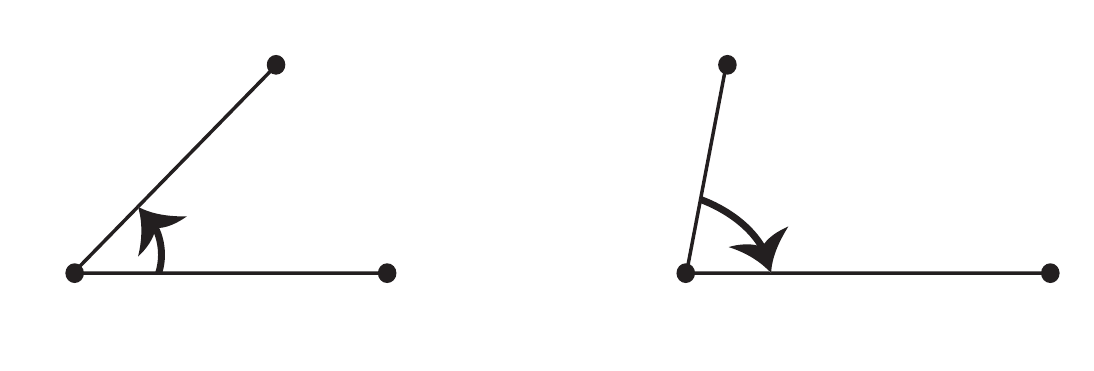}}
	\put(45,25){$b$}
\put(152,25){$a$}

\put(222,25){$e$}
\put(344,25){$d$}
\put(118,92){$c$}
\put(248,86){$f$}
	\end{picture}
	\caption{Diagrammatic representation of angles $abc$ and $def$ having opposite orientation}
	\label{oppositeorientation}
\end{figure}

Given the transitivity and symmetry properties of the \textit{SameOrientation} relation we can prove the following theorem. 

\begin{theorem}\label{SOOOOO}
	$SO(abc,def) \wedge OO(def,ghi) \rightarrow OO(abc,ghi)$
\end{theorem}

The following axiom states that switching the initial and terminal side of an angle changes its orientation. 

\begin{axiom}\label{OOreverse}
	$OO(abc,cba)$
\end{axiom}

The next axiom states that there is parity among angles orientation. This axiom is a critical property to build toward plane separation. 

\begin{axiom}\label{OOOOSO}
	$OO(abc,def) \wedge OO(def,ghi) \rightarrow SO(abc,ghi)$
\end{axiom}

One can prove the following lemma which states given two angles with same (opposite) orientations if one changes the initial and terminal sides of both angles the two resulting angles will have the same (opposite) orientation. 

\begin{lemma}\label{SOOOreverseboth}
	$SO(abc,def) \rightarrow SO(cba,fed)$ and 
	
	$OO(abc,def) \rightarrow OO(cba,fed)$
\end{lemma}

We now introduce some axioms dealing with the interplay between angle orientation and order. 

\begin{axiom}\label{intersectOO}
	$B(a,b,c) \wedge T(a,c,d) \rightarrow OO(dba,dbc)$ [See Figure \ref{intersectionoppositeorientation}.]
\end{axiom}

\begin{figure}[h!] 
	\begin{picture}(216,80)
	\put(105,0){\includegraphics[scale=.8]{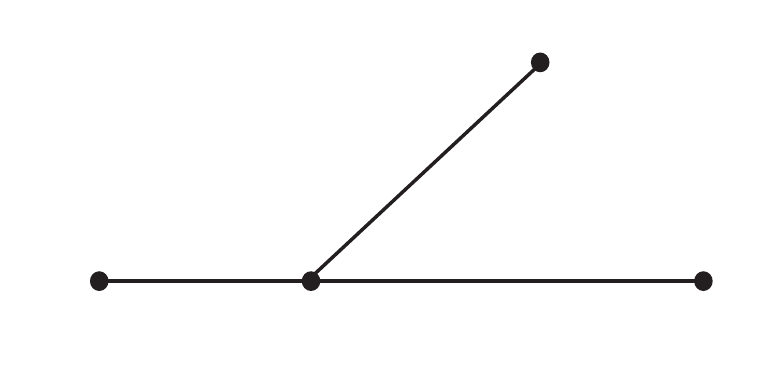}}
	\put(115,15){$a$}
	\put(276,15){$c$}
	\put(180,12){$b$}
	\put(230,80){$d$}
	\end{picture}
	\caption{}
\label{intersectionoppositeorientation}
\end{figure}

Given a line segment $db$, $OO(dba,dbc)$ is our proto-version of stating $a$ and $c$ are on `opposite sides' of $db$. Thus the previous axiom states something analogous to the idea that if a line segment contains one endpoint between $a$ and $c$ then $a$ and $c$ are on opposite sides of the line segment. (A fully flushed out definition of opposite sides of a line segment will be discussed in subsection \ref{sidesofline}.)

The next theorem shows that there is some relationship between colinearality and angle orientation.



\begin{theorem}\label{OOline}
	$L(a,b,c) \wedge OO(abd,abe) \rightarrow OO(cbd,cbe)$ [See Figure \ref{intersectopposite2}]
\end{theorem}

\begin{figure}[h!] 
	\begin{picture}(216,110)
	\put(105,0){\includegraphics[scale=.8]{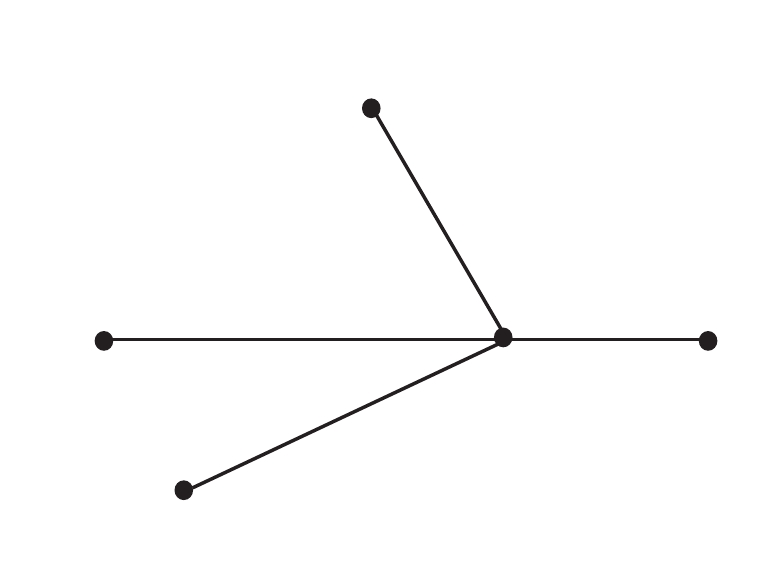}}
	\put(115,50){$a$}
	\put(276,50){$c$}
	\put(220,42){$b$}
	\put(200,105){$d$}
	\put(135,15){$e$}
	\end{picture}
	\caption{}
\label{intersectopposite2}
\end{figure}

\begin{proof}
	Since $L(a,b,c)$, either $B(a,b,c)$ or $SD(b,a,c)$. If $SD(b,a,c)$, then by Axiom \ref{SOray} and Theorem \ref{SOOOOO} we have $OO(cbd,cbe)$. Since $L(a,b,c)$, $a\not= c$. Since $OO(abd,abe)$, $T(a,b,d)$ and thus $a \not= d$ and $d \not= c$. Since $a$, $c$, and $d$ are all distinct, $T(a,c,d)$ otherwise $L(a,c,d)$. This would imply $L(a,b,d)$ by Axiom \ref{uniqueline}, but $T(a,b,d)$. Similarly $T(a,c,e)$. Thus if $B(a,b,c)$ then $OO(abd,cbd)$ and $OO(abe,cbe)$ by Axiom \ref{intersectOO}. Since $OO(abd,cbd)$ and $OO(abd,abe)$, we can infer $SO(cbd,abe)$ by Axiom \ref{OOOOSO}. And lastly, since $OO(abe,cbe)$, $OO(cbd,cbe)$ by Theorem \ref{SOOOOO}.
\end{proof}

We now define the relation that states when a point is in the interior of an angle. 

\begin{definition}
	$Int(d, abc) \equiv SO(cbd,cba) \wedge SO(abd,abc)$ [See Figure \ref{intdef}.]
\end{definition}

\begin{figure}[h!]
	\begin{picture}(216,80)
	\put(110,0){\includegraphics[scale=1]{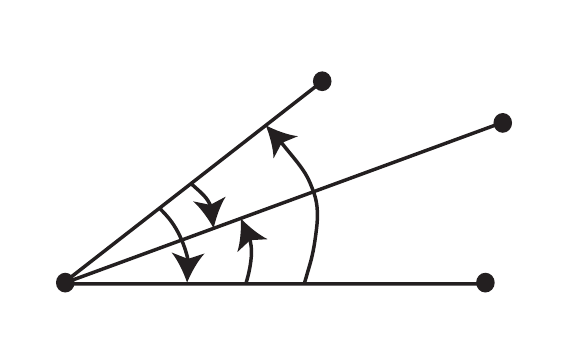}}
	\put(257,50){$d$}
	\put(210,70){$c$}
	\put(258,12){$a$}
	\put(117,12){$b$}
	\end{picture}
	\caption{}
	\label{intdef}
\end{figure}

From the definition one can prove the following theorem. 

\begin{theorem}\label{intOO}
	$Int(d,abc) \rightarrow OO(dba,dbc)$
\end{theorem}

\begin{proof}
	By Axiom \ref{OOreverse} we have $OO(abc,cba)$. Since we have $SO(abd,abc)$, by Theorem \ref{SOOOOO} we obtain $OO(abd,cba)$. Since we also have $SO(cba,cbd)$, by the same theorem we now have $OO(abd,cbd)$. By applying Lemma \ref{SOOOreverseboth} above we can infer $OO(dba,dbc)$.
	
\end{proof}

We now wish to introduce our first construction. First we will define the non-strict colinear relation which will make the first axiom about this construction significantly easier to state. 

\begin{definition} \label{nonstrictcolinear}
	$\widetilde{L}(a,b,c) \equiv L(a,b,c) \lor a=b \lor a=b \lor a=c$
\end{definition}

Do note that $\widetilde{L}(a,b,b)$ and $\widetilde{L}(a,a,a)$ are always true. 

Our first construction is called the crossbar construction. It is denoted $crossbow(d,a,b,c)$ which will be write as $cb(d,abc)$ for ease of reading. There is only one axiom pertaining to this construction. This axiom states that if $a$ and $c$ are on `opposite sides' of $bd$, then the constructed point $cb(d,abc)$ is non-strict colinear with $b$ and $d$ and is between $a$ and $c$. Note that the assumption of the axiom guaranties that $b \not= d$. See Figure \ref{crossbowfigure}.

\begin{axiom}\label{crossbow}
	$OO(dba,dbc) \rightarrow \widetilde{L}(b, cb(d,abc), d) \wedge B(a,cb(d,abc), c)$
\end{axiom}

\begin{figure}[h!] 
	\begin{picture}(216,110)
	\put(20,0){\includegraphics[scale=.7]{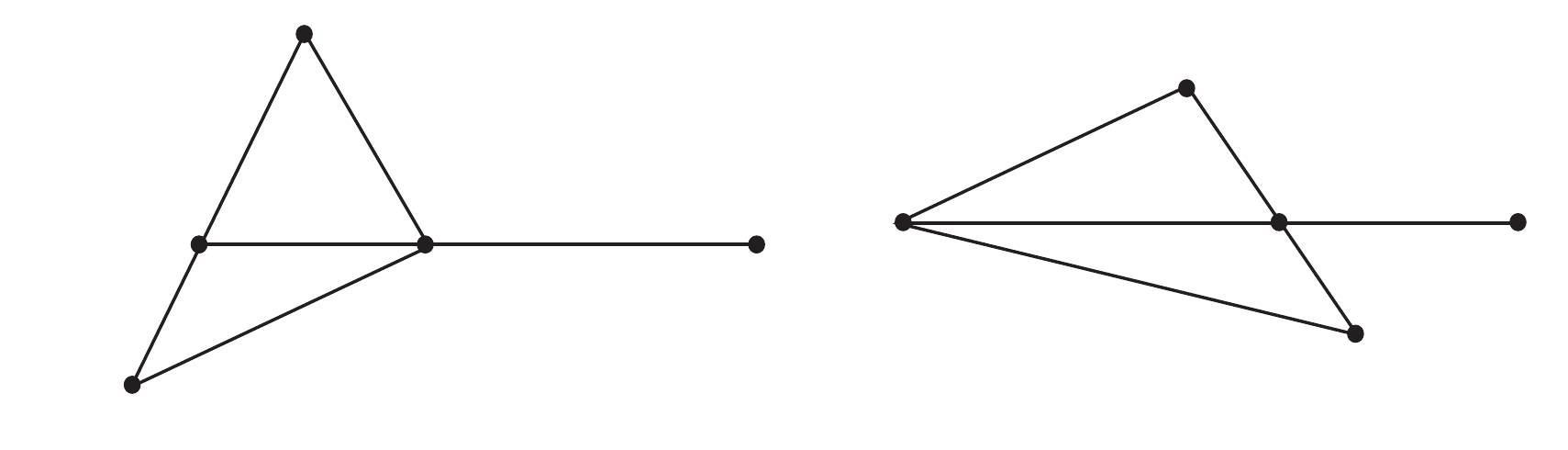}}
	\put(15,50){$cb(d,abc)$}
	\put(80,95){$a$}
	\put(275,85){$a$}
	\put(302,55){$cb(d,abc)$}
	\put(115,50){$b$}
	\put(211,50){$b$}
	\put(185,50){$d$}
	\put(360,50){$d$}
	\put(325,20){$c$}
	\put(37,10){$c$}
	\end{picture}
	\caption{}
\label{crossbowfigure}
\end{figure}

This axiom gives us an analog of the concept that if $a$ and $c$ are on opposite sides of $bd$ then the line $bd$ intersects the interior of $ac$.

The next theorem states that given a triangle $abc$ with a point $d$ between $a$ and $c$ we have that $d$ is in the interior of angle $abc$. 

\begin{theorem}\label{betweensideint}
	$T(a,b,c) \wedge B(a,d,c) \rightarrow Int(d,abc)$ 
\end{theorem}

\begin{proof}
	Note that given the assumptions all four point $a$, $b$, $c$, and $d$ are distinct. In particular if $b=d$ then $\lnot T(a,b,c)$. Suppose $\lnot T(a,b,d)$, then $L(a,b,d)$. But since $L(a,d,c)$ also, we obtain $L(a,b,c)$. This is a contradiction. Thus $T(a,b,d)$ and by similar arguments $T(c,b,d)$.
	
	Suppose $\lnot Int(d,abc)$, then $\lnot SO(cbd,cba)$ or $\lnot SO(abd,abc)$. Given the final sentence of the previous paragraph, if $\lnot SO(cbd,cba)$ then $OO(cbd,cba)$ Thus, $B(a, cb(c,abd), d)$. Since $L(a,d,c)$, we have $L(a, cb(c,abd), c)$. Because of this we know that $cb(c,abd) \not= c$. Since we also have $L(b, cb(c,abd), c)$  or $cb(c,abd) = b$, we can conclude $L(a,b,c)$. This is a contradiction. Similarly, supposing $\lnot SO(abd,abc)$ leads to a contradiction. Thus $Int(d,abc)$.  
\end{proof}

We are now able to prove our analog to Hilbert's Crossbar Theorem. The following theorem states that if $d$ is in the interior of angle $abc$, then $cb(d,abc)$ is the same direction from $d$ as $b$. 

\begin{theorem}\label{crossbar}
	$Int(d,abc) \rightarrow SD(b, cb(d,abc), d)$ [See Figure \ref{crossbartheorem}]
\end{theorem}

\begin{proof}
Let $cb(d,abc)=x$. If $x=d$ we are done. Suppose $x \not= d$. Since $Int(d,abc)$ we have $d \not= b$. By Axiom \ref{crossbow} we know $\widetilde{L}(b,x,d)$. We can infer $L(b,x,d)$. This implies that either $SD(b,x, d)$ or $B(x,b, d)$ (See Theorem \ref{lineseparation1}.) Suppose $B(b,x, d)$. One can show $T(x,d,a)$. Thus by Axiom \ref{intersectOO} we obtain $OO(abx, abd)$. Since $Int(d,abc)$, we have $SO(abd,abc)$. By Theorem \ref{betweensideint}, we know $Int(x,abc)$ and thus $SO(abx,abc)$. Using the symmetry and transitivity of the same orientation relation we can infer $SO(abx, abd)$. This is a contradiction. Thus $SD(b,x, d)$.
\end{proof}

\begin{figure}[h!]
	\begin{picture}(216,80)
	\put(110,0){\includegraphics[scale=.8]{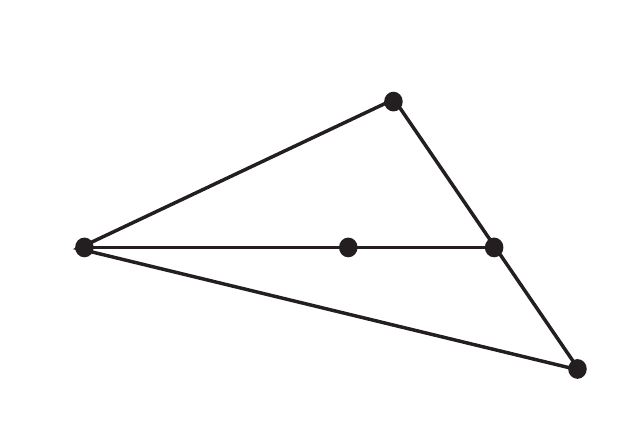}}
	\put(230,40){$cb(d,abc)=x$}
	\put(120,40){$b$}
	\put(190,50){$d$}
	\put(205,80){$a$}
	\put(250,10){$c$}
	\end{picture}
	\caption{}
	\label{crossbartheorem}
\end{figure}

The next theorem states that if angle $cbd$ and angle $cba$ have the same orientation and $a$ and $d$ are not on the same 'ray', then either $d$ is interior to angle $abc$ or $a$ is interior to angle $dbc$. This theorem will be used when angle inequalities are defined.

\begin{theorem}\label{SOint}
	$SO(cbd,cba) \wedge T(a,d,b) \rightarrow Int(d,abc) \lor Int(a,dbc)$
\end{theorem}

\begin{proof}
	If $SO(abd,abc)$, then $Int(d,abc)$ and we are done. Suppose $\lnot SO(abd, abc)$. Since $SO(cbd,cba)$, we have $T(c,b,d)$ and $T(c,b,a)$. Since $T(c,b,a)$, we have $T(a,b,c)$. We have also assumed $T(a,b,d)$, thus $OO(abd,abc)$. Thus $SO(abd,cba)$. Since we assume $SO(cbd,cba)$, we now have $SO(abd,cbd)$. So $SO(dba,dbc)$. Therefore $Int(a, dbc)$. 
\end{proof}

We now introduce our last axiom pertaining to angle orientation. 

\begin{axiom} \label{triangleOO}
	$T(a,b,c) \rightarrow OO(abc,bac)$ [See Figure \ref{triangleorientation}.]
\end{axiom}

\begin{figure}[h!] 
	\begin{picture}(216,115)
	\put(117,0){\includegraphics[scale=.8]{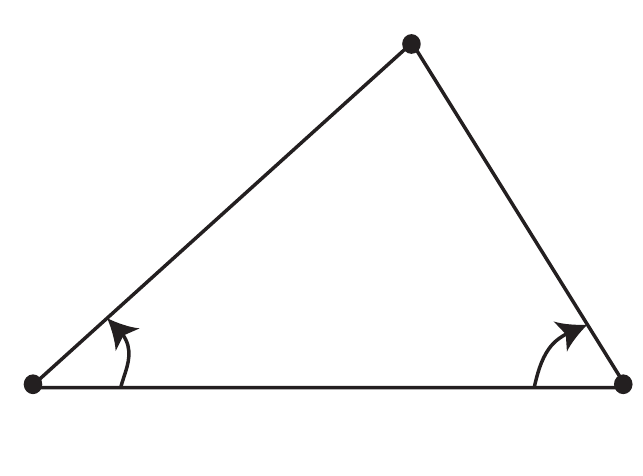}}
	
	\put(117,13){$a$}
	\put(264,13){$b$}
	\put(210,103){$c$}
	
	\end{picture}
	\caption{}
\label{triangleorientation}
\end{figure}

This axiom  is the final assumption needed to obtain our analogs to plane separation. The following theorem shows that if two points are on the `same side' of a line segment, then no point can be co-linear with the line segment and be between the two points. 

\begin{theorem}\label{SOnointersection}
	$SO(cbd,cba) \wedge B(a,x,d) \rightarrow \lnot \widetilde{L}(b,x,c)$ [See Figure \ref{sameorientationnointersect} ]
\end{theorem}

\begin{figure}[h!] 
	\begin{picture}(216,130)
	\put(30,0){\includegraphics[scale=.8]{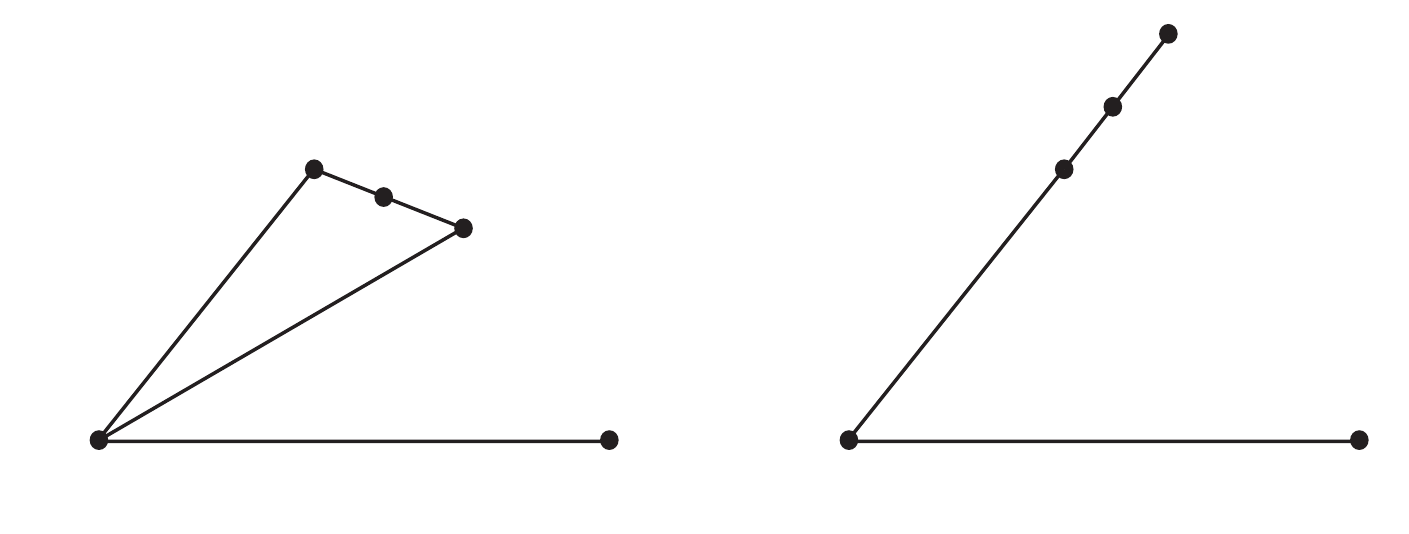}}
	\put(175,12){$c$}
	\put(45,12){$b$}
	\put(94,86){$a$}
	\put(122,79){$x$}
	\put(141,68){$d$}
	
	\put(348,12){$c$}
	\put(218,12){$b$}
	\put(266,84){$a$}
	\put(278,98){$x$}
	\put(290,115){$d$}

	\end{picture}
	\caption{}
\label{sameorientationnointersect}
\end{figure}

\begin{proof}
	Start by noting that $T(c,b,d)$ and $T(c,b,a)$. We will first show that $x \not= b$ and $x \not= c$. Suppose $x =b$. Then $B(a,b,d)$. If $L(a,d,c)$, then $L(b,c,d)$ which contradicts $T(c,b,d)$. Thus we can infer $T(a,d,c)$. By Axiom \ref{intersectOO} we have $OO(cbd,cba)$. This is a contradiction.
	
	Now suppose $x = c$. Then $B(a,c,d)$. If $L(a,d,b)$, then $L(b,c,d)$ which contradicts $T(c,b,d)$. Thus we can infer $T(a,d,b)$. By Axiom \ref{intersectOO} we have $OO(bcd,bca)$ and by an application of Axiom \ref{triangleOO} we have $OO(cbd,cba)$ This is a contradiction. Thus $x \not= b$ and $x \not= c$. If we suppose $\widetilde{L}(b,x,c)$ we can then infer $L(b,x,c)$. We will show that this assumption leads to a contradiction. There are two cases to consider. 
	
	Case 1: Suppose $T(a,b,d)$. By Theorem \ref{betweensideint}, we have $Int(x, abd)$ and thus $OO(xba,xbd)$. Since we are supposing $L(b,x,c)$, By Theorem \ref{OOline}, we have $OO(cba,cbd)$. This is a contradiction.

	Case 2: Suppose $\lnot T(a,b,d)$. Since $T(c,b,d)$, $T(c,b,a)$ and $a \not= d$, we have that $a$, $b$, and $d$ are all distinct. Thus $L(a,b,d)$. Since $L(a,x,d)$ also, we have $L(b,a,x)$. And since $L(b,x,c)$, we have $L(a,b,c)$. This is a contradiction. 
	
\end{proof} 

We now turn to proving our analog to Pasch's axiom. But first we prove a lemma.

\begin{lemma}\label{Paschlemma}
	$T(a,b,c) \wedge B(a,d,b) \wedge T(a,b,e) \wedge T(e,d,c) \rightarrow OO(edb,edc) \lor OO(edc,eda)$ [See Figure \ref{paschlemmafigure}.]
\end{lemma}

\begin{figure}[h!]
	\begin{picture}(216,170)
	\put(0,-10){\includegraphics[scale=.8]{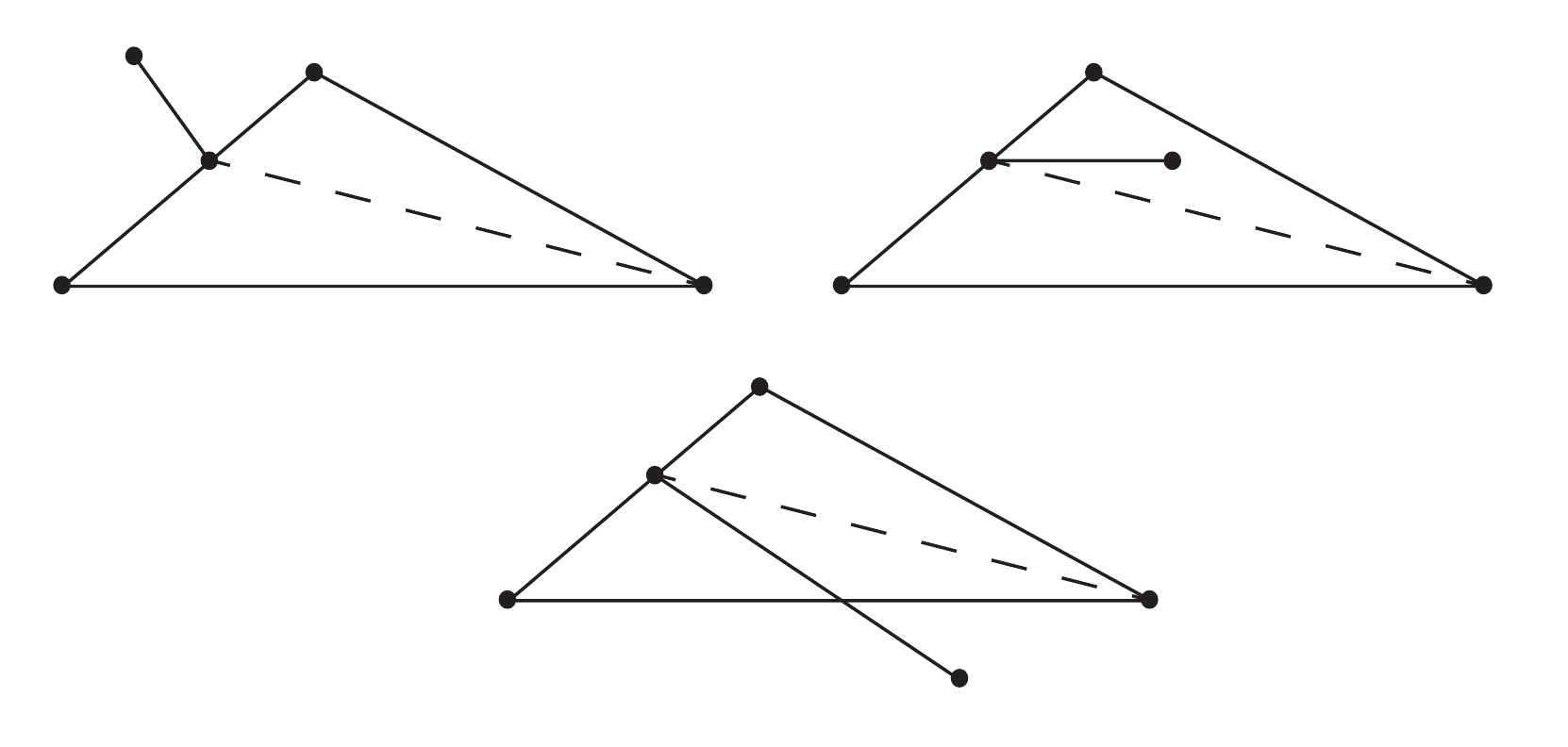}}
	\put(120,15){$a$}
	\put(173,75){$b$}
	\put(280,15){$c$}
	\put(150,55){$d$}
	\put(225,2){$e$}
	\put(10,92){$a$}
	\put(63,152){$b$}
	\put(170,92){$c$}
	\put(37,132){$d$}
	\put(22,150){$e$}
	\put(200,92){$a$}
	\put(253,152){$b$}
	\put(360,92){$c$}
	\put(227,132){$d$}
	\put(280,135){$e$}
	\end{picture}
	\caption{}
	\label{paschlemmafigure}
\end{figure}

\begin{proof}
		
	We first want to show $T(e,d,b)$. One can show that $e$, $d$, and $b$ are all distinct. Thus if $\lnot T(e,d,b)$, then $L(e,d,b)$. But since $L(a,d,b)$, we would have $L(a,e,b)$. This is a contradiction. Thus $T(e,d,b)$. 
	
	Now if $OO(edb,edc)$ we are done. So suppose $\lnot OO(edb,edc)$. Since $T(e,d,b)$ and $T(e,d,c)$, and $\lnot OO(edb,edc)$, we have $SO(edb,edc)$. By Axiom \ref{intersectOO}, we have $OO(ebd,eda)$. Therefore by Theorem  \ref{SOOOOO}, we obtain $OO(edc,eda)$. 
\end{proof}



	
	
	

This next lemma states something analogous to the idea that a line can't intersect the interior of three sides of a triangle. It is used in the proof of a theorem in Section \ref{SoildGeo}. 

\begin{lemma}\label{orientlemma3}
	$T(a,b,c) \wedge B(a,d,b) \wedge B(b,e,c) \wedge B(a,f,c) \rightarrow T(d,e,f)$
\end{lemma}

\begin{proof}
	First we will show that $d$, $e$, and $f$ are distinct. If $d=e$ then $B(a,e,b)$. Since $L(b,e,c)$ and $L(b,e,a)$, we have $L(a,b,c)$. This is a contradiction. Thus $e \not= d$. In similar fashion we can show the $d \not = f$ and $e \not= f$. Furthermore we can show that all points $a$, $b$, $c$, $d$, $e$, and $f$ are all distinct. 
	We can show $T(a,f,b)$. If $\lnot T(a,f,b)$, then $L(a,f,b)$. This would imply $L(a,b,c)$ which is a contradiction. Likewise we can show that $T(c,f,b)$, $T(c,d,e)$ and $T(a,d,c)$.
	
	Suppose $\lnot T(d,e,f)$. Since $d$, $e$, and $f$ are distinct, we have $L(d,e,f)$.  Since $L(d,e,f)$ without loss of generality we can say $B(d,f,e)$. 
	Since $T(a,f,b)$ and $B(a,d,b)$ by Theorem \ref{betweensideint} we have $Int(d, afb)$ and thus $SO(afd,afb)$. Likewise we have $Int(e,cfb)$ and thus $SO(cfe, cfb)$. By axiom \ref{intersectOO}, we have $OO(afb,cfb)$. Since $SO(afd,afb)$ and $OO(afb,cfb)$ by Theorem \ref{SOOOOO} we have $OO(afd, cfb)$. Since $SO(cfe, cfb)$ and $OO(afb,cfb)$ by Theorem \ref{SOOOOO} we have $OO(cfe, afb)$. Since $B(d,f,e)$ we have $OO(cfe,cfd)$ via an application of Axiom \ref{intersectOO}. Thus by Axiom \ref{OOOOSO} we have $SO(afd, cfd)$. But by Axiom \ref{intersectOO} we know $OO(afd,cfd)$. This is a contradiction. Thus $T(d,e,f)$. 
\end{proof}

We now state our version of Pasch's axiom as a theorem:

\begin{theorem}\label{Paschthm}
	Let $A \equiv T(a,b,c) \wedge B(a,d,b) \wedge T(a,b,e) \wedge T(e,d,c)$, 
	
	$P \equiv \widetilde{L}(d, cb(e,adc),e) \wedge B(a, cb(e,adc),c)$, 
	
	$Q \equiv \widetilde{L}(d, cb(e,bdc),e) \wedge B(b, cb(e,bdc),c)$,
	
	$R \equiv   B(b,x,c) \rightarrow \lnot \widetilde{L}(e,d,x)$, and
	
	$S \equiv B(a,x,c) \rightarrow \lnot \widetilde{L}(e,d,x)$, then

 $A \rightarrow (P \wedge R) \lor (Q \wedge S) $, 
	
	

\end{theorem}

\begin{proof}
	Assuming $A$, Lemma \ref{Paschlemma} tells us that $OO(edb,edc)$ or $OO(edc,eda)$. If $OO(edb,edc)$, then by Axiom \ref{crossbow} we have $P$. As was pointed out in the proof of the previous lemma, by Axiom \ref{intersectOO} we know that $OO(eda,edb)$. Thus by Axiom \ref{SOOOOO} and the previous lemma we have $SO(eda,edc)$. Therefore by Theorem \ref{SOnointersection} we have $R$. If $OO(edc,eda)$ by similar reasoning we have $Q$ and $S$


	
	
\end{proof}

\subsection{Sides of a Line Segment} \label{sidesofline}

Up to this point we have using the relation $SO(abc,abd)$ to be our proto-definition for $c$ and $d$ being on the `same side' of segment $ab$. It is quite obvious that more work still needs to be done. Although we were able to prove quite a few results about the crossbow construction and the interior of angles, we have not explicitly shown that if $c$ and $d$ are on the `same side' of $ab$, then $c$ and $d$ are on the `same side' of $ba$. Furthermore we would desire that if points $a$, $b$, $c$, and $d$ were colinear then $e$ and $f$ being on the `same side' of $ab$ would imply that $e$ and $f$ were on the `same side' of $cd$. We can accomplish this. 

First we will define a same side relation, $SameSide(c,d,a,b)$, shortened to $SS(c,d,ab)$.

\begin{definition}
	$SS(c,d,ab) \equiv SO(abc,abd)$
\end{definition}

The first issue pointed in the previous paragraph is addressed by part three of the following theorem. From the reflexivity of the angle orientation relation we can prove part 1 of the theorem below. By the symmetry of angle orientation we can prove part 2 and by using the Axiom \ref{triangleOO} we can prove part 3. Part 4 can be proved from the transitivity property of the same orientation relation

\begin{theorem} \label{sideslinethm1}\hspace{1in}
	\begin{enumerate}
		\item $T(a,b,c) \rightarrow SS(c,c,ab)$
		\item $SS(c,d,ab) \rightarrow SS(d,c,ab)$
		\item $SS(c,d,ab) \rightarrow SS(c,d,ba)$
		\item $SS(c,d,ad) \wedge SS(d,e,ab) \rightarrow SS(c,e,ab)$
	\end{enumerate}
\end{theorem}

To address the second concern pointed out in the first paragraph of this section we  have the following theorem.
\begin{theorem} \label{sideslinethm2}
	Given $a \not= b$, $c \not= d$, $\widetilde{L}(a,b,c)$, and $\widetilde{L}(b,c,d)$, if $SS(e,f,ab)$ then $SS(e,f,cd)$. 
\end{theorem}

\begin{proof}

Case 1)	Let $a=c$ and $b=d$. Then we are done.

Case 2) Let $a=d$ and $b=c$. This can be proved by an application of part 3 of the previous theorem.

Case 3) Let $a=c$ and $b \not= d$. If $B(b,a,d)$ (see left side of Figure \ref{sideslinethm2figure}), then $SS(e,f,ab)$ implies $SO(abe,abf)$ which in turn implies $SO(dbe,dbf)$ by Axiom \ref{SOray}. Using Axiom \ref{triangleOO} we can infer $SO(bde,bdf)$ and this implies $SO(cde,cdf)$. Thus by definition we have $SS(e,f,cd)$. If $\lnot B(b,a,d)$, we have $SD(a,b,d)$ (see theorem \ref{lineseparation1}). (See right side of figure \ref{sideslinethm2figure}.) If $SS(e,f,ab)$, then we have $SO(abe,abf)$. This implies $SO(bae,baf)$ by axiom \ref{triangleOO}. From this we can obtain $SO(bce,bcf)$ and then $SO(dce,dcf)$ by Axiom \ref{SOray}. Thus we have $SS(e,f,cd)$. 
	
\begin{figure}[h!]
	\begin{picture}(216,95)
	\put(45,0){\includegraphics[scale=.8]{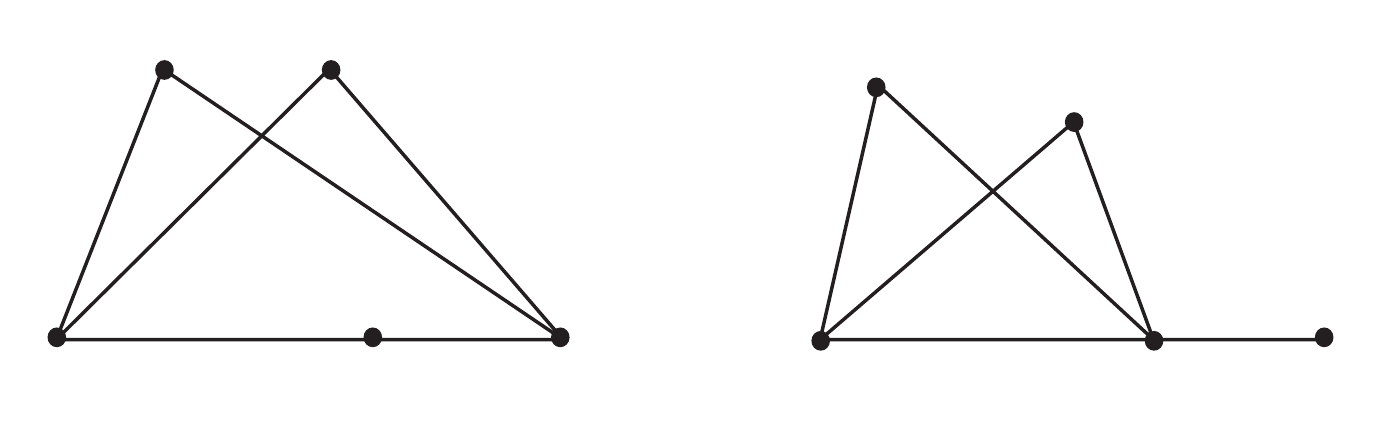}}
	\put(118,13){$a=c$}
	\put(50,13){$b$}
	\put(175,13){$d$}
	\put(74,85){$e$}
	\put(112,85){$f$}
	
	\put(220,13){$a=c$}
	\put(310,9){$b$}
	\put(352,13){$d$}
	\put(238,81){$e$}
	\put(285,73){$f$}
\end{picture}
	\caption{}
	\label{sideslinethm2figure}
\end{figure}

	The next three cases follow from similar methods.

Case 4) Let $a=d$ and $b \not= d$.

Case 5) Let $b=c$ and $a \not= d$.

Case 6) Let $b=d$ and $a \not= c$.

Case 7) Let $a$, $b$, $c$, and $d$ all be distinct. Suppose $SD(b,a,c)$. Since $SS(e,f,ab)$ we then have $SO(abe,abf)$ which will imply $SO(cbe,cbf)$ by Axiom \ref{SOray}. By Axiom \ref{triangleOO} we obtain $SO(bce,bcf)$. 
	
	Now suppose $\lnot SD(b,a,c)$. We can infer that $B(a,b,c)$ (see Theorem \ref{lineseparation1}). Since $SS(e,f,ab)$, we have $SO(abe,abf)$. By Theorem \ref{OOline} we $SO(cbe,cbf)$ and by Axiom \ref{triangleOO} we have $SO(bce,bcf)$ also.
	
	Since $SO(bce,bcf)$, if $SD(c,b,d)$ also then we have $SO(dce,dcf)$. If $\lnot SD(c,b,d)$, then $B(d,c,b)$. Since we have $SO(bce,bcf)$, by theorem \ref{OOline} we obtain $SO(dce,dcf)$ also. We can then can infer $SS(e,f,dc)$ and by applying Theorem \ref{sideslinethm1} we have $SS(e,f,cd)$.

\end{proof}


	

We can define an opposite sides relation in a similar fashion. This relation can be denoted $OppositeSides(c,d,a,b)$, but as with other relations will be shortened to $OS(c,d,ab)$. 

\begin{definition}
	$OS(c,d,ab) \equiv OO(abc,abd)$ 
\end{definition}

One can then prove the following theorem. 

\begin{theorem}
	$T(a,b,c) \wedge T(a,b,d) \wedge \lnot SS(c,d,ab) \rightarrow OS(c,d,ab)$
\end{theorem}

There are analogous versions of the two main theorems of this subsection where the same side relation is replaced with the opposite sides relation. We omit the proofs. 

\begin{theorem} \label{oppositesidethm} \hspace{1in}
	\begin{enumerate}
		\item $OS(c,d,ab) \rightarrow OS(d,c,ab)$
		\item $OS(c,d,ab) \rightarrow OS(c,d,ba)$
		\item $OS(c,d,ad) \wedge OS(d,e,ab) \rightarrow SS(c,e,ab)$
	\end{enumerate}
\end{theorem}
 
\begin{theorem} \label{oppositesidethm2}
	Given $a \not= b$, $c \not= d$, $\widetilde{L}(a,b,c)$, and $\widetilde{L}(b,c,d)$, if $OS(e,f,ab)$ then $OS(e,f,cd)$. 
\end{theorem}

\subsubsection{Reasons for Using Angle Orientation}

It is possible that the reader may be asking themselves why we did not just introduce a 'same side' relation to start with and avoid a discussion of angle orientation. The reader may also wonder if the number of axioms assumed would be minimized and the road to proving Pasch's axiom would simpler if we could just assume a small set of assumptions about the same side relations to obtain the plain separation property and then prove Pasch's axiom as a theorem in the style of Greenberg in \cite{GreenbergBook}.  They may think that this would lead to not having need for a construction such as the crossbar construction. One may wonder if such a strategy would be more economical. We would reply that in fact one is not able to greatly reduce the number of axioms assumed and constructions used by introducing a `same side' relation to start with. 


If one were to try to obtain the plane separation properties and Pasch's axiom simply from an undefined `same same' relation they would in fact have to assume quite a few axioms. They would have to assume that the relation had the three properties stated in Theorem \ref{sideslinethm1}. They would have to define an `opposite side' relation. They would have to assume that if $a$ and $b$ were on the same side and $b$ and $c$ were on opposite sides then $a$ and $c$ would be on opposite sides. They would have to assume an axiom analogous to Theorem \ref{SOnointersection} which would interpret the property that lines do not intersect a line segment between two points on the same side of it. They would have to assume an axiom like Axiom \ref{intersectOO} which states that if a line intersects a segment then the endpoints of the segment are on opposite sides. They would have to introduce a crossbow-like construction and assume an axiom like Axiom \ref{crossbow} to produce the point for where a line intersects the interior of a segment with points on opposite sides of it. They would have to assume an axiom like Theorem \ref{sideslinethm2} in order to preserve the sides of two line segments when they are colinear. They would then need to prove Pasch's axiom and then also prove the Crossbar Theorem. Because of this we see no formal reason why having an undefined 'same side' relation would be a great improvement.  

\subsection{Line Separation} \label{lineseparation}

An important result in any presentation of Hilbert's system is proving a line separation property. In \cite{HartshorneBook} Hartshorne defines line separation as the property that the set of points on a line $\ell$ not equal to a particular point $a$ on the line can be divided into two sets called sides such that $b$ and $c$ are on the same side of $a$ if and only if $a$ is not in the segment $bc$ and $b$ and $c$ are on different sides of $a$ if and only if $a$ is in the segment $bc$. 

We already have relations which codify the idea of $b$ and $c$ being on the same or opposite sides of $a$. The relation $B(b,a,c)$ will take the place of $b$ and $c$ being on `opposite sides' of $a$ and $SD(a,b,c)$ will take the place of $b$ and $c$ being on the `same side' of $a$.

It can be shown that if $L(a,b,c)$ then $\lnot SD(a,b,c)$ implies $B(b,a,c)$ and $\lnot B(b,a,c)$ implies $SD(a,b,c)$. 

\begin{theorem} \label{lineseparation1} \hspace{1in}
	\begin{enumerate}
		\item $L(a,b,c)) \wedge \lnot SD(a,b,c) \rightarrow B(b,a,c)$
		\item 	$L(a,b,c) \wedge \lnot  B(b,a,c) \rightarrow SD(a,b,c)$
	\end{enumerate}
\end{theorem}

\begin{proof}
	To prove part 1 note that $\lnot SD(a,b,c)$ implies that $\lnot B(a,b,c)$ and $\lnot B(a,c,b)$. Since we have $L(a,b,c)$ and $\lnot B(a,c,b)$, $\lnot B(a,b,c)$ we can conclude $B(b,a,c)$. 
	
	Part 2 is proved via a straight forward proof based on definitions. 
\end{proof}

Hilbert's line separation property is captured by the following theorem. 

\begin{theorem} \label{lineseparationthm} \hspace{1in}
	\begin{enumerate}
		\item $SD(a,b,b)$ 
		\item $SD(a,b,c) \rightarrow SD(a,c,b)$
		\item $SD(a,b,c) \wedge SD(a,c,d) \rightarrow SD(a,b,d)$
		\item $SD(a,b,c) \wedge B(d,a,c) \rightarrow B(d,a,b)$
		\item $B(b,a,c) \wedge B(d,a,c) \rightarrow SD(a,b,d)$
	\end{enumerate}
\end{theorem}

\begin{figure}[h!]
	\begin{picture}(216,90)
	\put(0,0){\includegraphics[scale=.6]{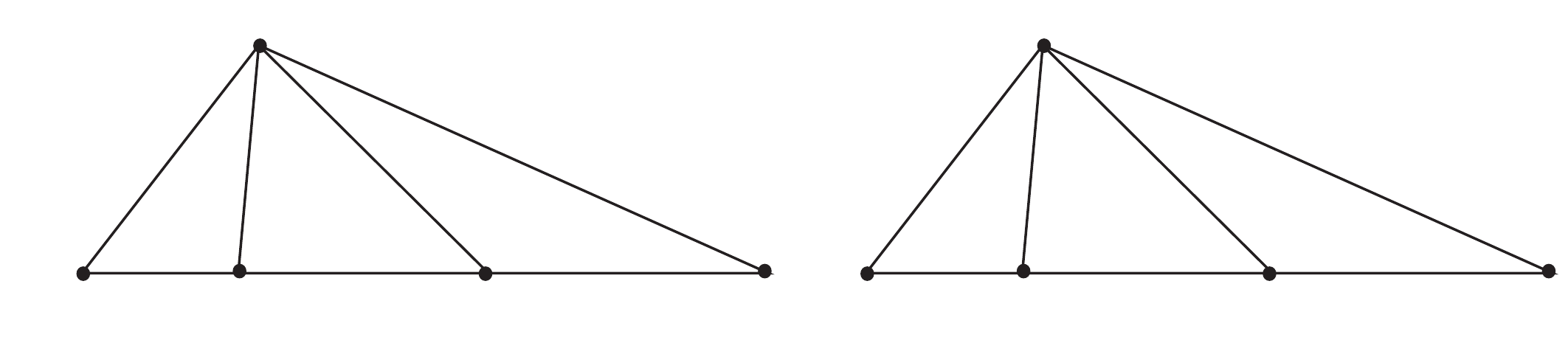}}
	\put(112,11){$c$}
	\put(52,9){$b$}
	\put(15,11){$a$}
	\put(180,11){$d$}
	\put(52,73){$\gamma$}
	
	\put(200,9){$d$}
	\put(237,9){$b$}
	\put(297,11){$a$}
	\put(365,11){$c$}
	\put(237,73){$\gamma$}

	\end{picture}
	\caption{}
	\label{lineseparationthm2figure}
\end{figure}

Parts 1 and 2 follow directly from the definition of the same direction relation. The first three parts show that the two points being of the `same side' of $a$ is a reflexive, symmetric, and transitive relation. One can use the symmetric and transitive properties along with the previous theorem to prove part 4. Thus we will only prove parts 3 and part 5. 

\begin{proof}
	
	By Axiom \ref{3noncolinearpoints} we know that at least one of $\alpha$, $\beta$, or $\gamma$ are not colinear with $a$ and $b$. Without loss of generality let $T(a,b,\gamma)$. The assumptions of both part 3 and part 5 imply $L(a,b,c)$ and $L(a,c,d)$. From this it can be shown that $\gamma$ is not colinear with any pair of points among $a$, $b$, $c$, or $d$. 
	
	To prove part 3 (see right side of Figure \ref{lineseparationthm2figure})note that $SD(a,b,c)$ implies $\lnot B(b,a,c)$ by the previous theorem. We now wish to show that $\lnot B(b,a,c)$ implies $SO(a \gamma b, a \gamma c)$. Suppose $\lnot SO(a \gamma b, a \gamma c)$, then by the last sentence in the previous paragraph we can infer $OO(a \gamma b, a \gamma c)$. We will show that this leads to a contradiction. Given $OO(a \gamma b, a \gamma c)$ we can infer certain properties about $cb(a, b \gamma c)$. By axiom \ref{crossbow} we know $B(b, cb(a, b \gamma c), c)$. Now if $cb(a, b \gamma c) = \gamma$, then $L(b, \gamma, c)$. This is a contradiction. Thus $cb(a, b \gamma c) \not= \gamma$. If $cb(a, b \gamma c) \not= a$, then $L(\gamma, cb(a, b \gamma c), a)$. We already know $L(b, cb(a, b \gamma c), c)$ and $L(a,b,c)$. By applying axiom \ref{uniqueline} repeatedly we can show $L(b,a,cb(a, b \gamma c))$, followed by $L(\gamma, b, cb(a, b \gamma c))$, then followed by $L(\gamma, b, c)$. This is a contradiction. Thus we can infer that $SO(a \gamma b, a \gamma c)$. By similar methods we can show that $SD(a,c,d)$ implies $SO(a \gamma c, a \gamma d)$. Thus by axiom \ref{SOtrans} we have $SO(a \gamma b, a \gamma d)$. We want to show that this implies $\lnot B(b,a,d)$. Suppose $B(b,a,d)$, then by axiom \ref{intersectOO} we could infer $OO(a \gamma b, a \gamma d)$. This is a contradiction. Thus $\lnot B(b,a,d)$. By the previous theorem this implies $SD(a,b,d)$. 
	
	To prove part 5 (see left side of Figure \ref{lineseparationthm2figure}) note that $B(b,a,c)$ implies $OO(a \gamma b, a \gamma c)$ and $B(d,a,c)$ implies $OO(a \gamma d, a \gamma c)$. By Axiom \ref{OOOOSO} we can infer $SO(a \gamma b, a \gamma d)$. This in turn implies $\lnot B(b,a,d)$ which implies $SD(a,b,d)$ by the previous theorem. 
\end{proof}

\section{Axioms of Congruence} \label{congruence}

Hilbert introduced two congruence relations. One relation was the congruence of line segments. The interpretation was that two line segments were congruent if they had the same length. The other relation was for two angles being congruent. Angles for Hilbert were formed by two rays that shared an endpoint. All angles were less than a straight angle. For us an angle is an ordered triple of three non-colinear points. Tarski was able to avoid discussing angles all together by a use of ``five-segment'' axiom that took the place of the Side-Angle-Side triangle congruence axiom used by Hilbert. In \cite{Beeson} it was shown that a Hilbert style angle congruence can be derived from Tarski's axioms. 

\subsection{Hilbert's Congruence Axiom 1}

In his first congruence axiom Hilbert states two important properties. First is that every line segment is congruent to itself. 

We introduce a relation for line segment congruence $C(a,b,c,d)$ which may be shortened as $C(ab,cd)$. Its interpretation will be that the line segment $ab$ is the same length as the line segment $cd$. The following axiom captures Hilbert's reflexive properties as well a stating that segment congruence is unaffected by reversing endpoints. 

\begin{axiom} \label{segcongreflexiveswitchendpoints}
	$C(ab,ab)$ and $C(ab,ba)$
\end{axiom}

It is important to point out that we will be allowing for the congruence of null segments such as $aa \cong bb$. In Hilbert's system segment congruence is only referred to when discussing distinct pairs of points. In \cite{HartshorneBook} Hartshorne follows this convention. Because of this, our definition of segment inequality will differ slightly for that of Hilbert and in turn Hartshorne. There are pros and cons to either convention, but we have chosen our because it will aid in defining and proving results about compass constructions later on. (It appears that Hartshornes section of circle continuity is in fact flawed in that it's defining of circle-circle continuity and this definition's application in the proof of line-circle continuity because the two different version of segment inequality are conflated.)

The second property of Hilbert's first congruence axiom is that given a ray with endpoint $a$ and a line segment $cd$, there is a unique point $x$ on the ray such that $ax$ is congruent to $cd$. 

Back in subsection \ref{sectionorderaxiom2} we stated that we would be introducing an extension construction $ext(a,b,c,d)$ (which may be shortened to $ext(ab,cd)$). The constructed point $ext(ab,cd)$ will be interpreted as being a point that extends the line segment $ab$ by the length of the line segment $cd$ in the direction from $a$ to $b$. These properties are formally codified in the following axiom [see Figure \ref{extfigure}]. 

\begin{axiom} \label{extcong}
	$a \not= b \rightarrow C(b,ext(ab,cd),c,d)$
\end{axiom}

\begin{axiom} \label{extbetween}
	$a \not= b \wedge c \not= d \rightarrow B(a, b, ext(ab,cd))$
\end{axiom}

\begin{figure}[h!]
	\begin{picture}(216,80)
	\put(102,0){\includegraphics[scale=.9]{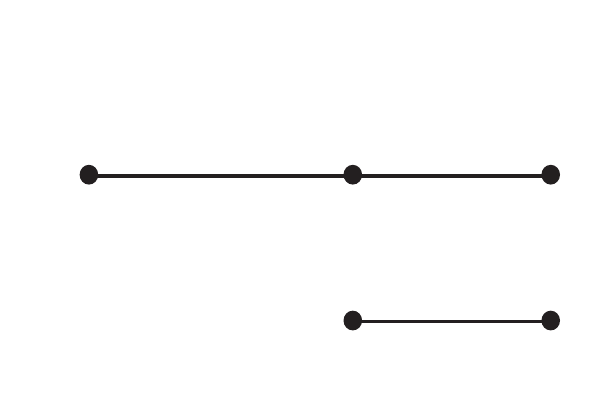}}
	\put(190,10){$c$}	
	\put(245,10){$d$}

	\put(120,45){$a$}	
	\put(190,45){$b$}	
	\put(245,45){$ext(ab,cd)$}
	
	\end{picture}
	\caption{}
	\label{extfigure}
\end{figure}

One will note that the condition $a \not= b$ is necessary since we conclude that $B(a, b, ext(ab,cd))$. Also one would not be able to make sense of a direction from $a$ to $b$ when $a =b$. We will be allowing extensions of non-null segments by null segments.  

We will now define a new construction from the extension. This construction will have a central importance in later sections about parallel line segments. 

\begin{definition} \label{doubdef}
	If $a \not= b$, then $doub(a,b) \equiv ext(ab,ab)$. Also $doub(a,a) = a$.
\end{definition}

We will call this construction the doubling construction. Its interpretation is a point that extends the line segment $ab$, in the direction from $a$ to $b$, by the length of $ab$ if $a \not= b$. In the case $a = b$ the construction is simply the point $a$. 

We now define a construction which will help satisfy the second property of Hilbert's first congruence axiom. 

\begin{definition} \label{layoffdef}
	$lf(ab,cd) \equiv ext(doub(b,a),a,c,d)$ [See Figure \ref{layofffigure}]
\end{definition}

\begin{figure}[h!]
	\begin{picture}(216,80)
	\put(70,0){\includegraphics[scale=.9]{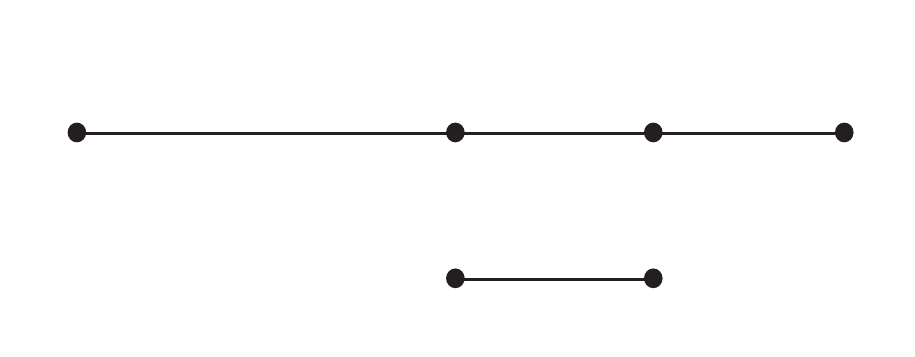}}
	\put(185,10){$c$}	
	\put(240,10){$d$}
	
	\put(80,45){$doub(b,a)$}	
	\put(185,45){$a$}	
	\put(220,64){$ext(doub(ba),a,c,d)$}
	\put(290,45){$b$}
	
	\end{picture}
	\caption{}
	\label{layofffigure}
\end{figure}

This construction will be called the layoff construction. One can easily verify $C(a,lf(ab,cd),c,d)$. One can also verify $SD(a,lf(ab,cd), b)$ by using results about line separation in subsection \ref{lineseparation}.

We want the points constructed by the extension construction to be unique in terms of order and congruence properties so we introduce the following axiom. 

\begin{axiom} \label{extunique}
	$(B(a,b,x) \lor b=x) \wedge C(bx,cd) \rightarrow x = ext(ab,cd)$
\end{axiom} 

This last definition and last axiom give us a layoff construction with analogous properties to Hilbert's first congruence axiom. 

We can now define an inequality relation between two line segments. 

\begin{definition} \label{segmentinequalitydef}
	$ab < cd \equiv B(c,lf(cd,ab),d) \lor (a=b \wedge c \not= d)$ 
	
	

\end{definition}

 It should be pointed out that by the uniqueness of the extension construction, if $B(a,b,c)$, then $ab < ac$. 

In the future we will have need of the following theorem. We omit the proof.  

\begin{theorem}
	$ab< cd \wedge cd< ef \rightarrow ab<ef$
\end{theorem}



\subsection{Hilbert's Congruence Axiom 2}

The following axiom is identical to Hilbert's second congruence axiom. 

\begin{axiom} \label{segcongtrans}
	$C(ab,cd) \wedge C(ab,ef) \rightarrow C(cd,ef)$
\end{axiom}

From Axiom \ref{segcongreflexiveswitchendpoints} and Axiom \ref{segcongtrans} we see that line segment congruence is reflexive, symmetric, and transitive. 

\subsection{Hilbert's Congruence Axiom 3}

Hilbert's third congruence axiom states that if $ab$ and $bc$ are two segments on the same line that have no points in common other than $b$ and $de$ and $ef$ are two segments one the same line that have no points in common other than $e$, then if $ab$ is congruent to $de$ and $bc$ is congruent to $ef$, we have $ac$ is congruent to $df$. This is know as the addition of segments axiom.  We capture this axiom as follows. 

\begin{axiom}
	$B(a,b,c) \wedge B(d,e,f) \wedge C(ab,de) \wedge C(bc,ef) \rightarrow C(ac,df)$
\end{axiom}

\subsection{Hilbert's Congruence Axiom 4}

Hilbert's fourth congruence axiom deals with the construction of congruent angles. First we will introduce an angle congruence relation. The relation will be notated $AC(a,b,c,d,e,f)$ (shortened to $AC(abc,def)$ for readability). In order to avoid having to speak about the congruence of colinear or non distinct triples will have the following axiom.

\begin{axiom} \label{anglecongnoncolinear}
	$AC(abc,def) \rightarrow T(a,b,c) \wedge T(d,e,f)$
\end{axiom}

Hilbert's fourth axiom states that given an angle $abc$ and a line $de$, then on a given side of the line $de$ there is a unique ray $dx$ such that angle $edx$ is congruent to angle $abc$. [Hilbert's original axiom also included properties about the interior of angles. We will be following the convention of Hartshorne and not included this. We have done the necessary work to codify the interior of angles in Section \ref{angleorientationsec}] We capture this axiom by introducing a construction. Our angle transport (same side) construction will be denoted $ats(a,b,c,d,e,f)$ (shortened to $ats(abc,def)$). We interpret this constructed point to have two important properties. We want that the angle formed from $d$ to $e$ to $ats(abc,def)$ to be congruent with the angle $abc$. We also want that the constructed point is on the same side of $de$ as $f$. The following axiom captures these properties as well as requiring that the segment $bc$ is congruent to the segment from $e$ to $ats(abc,def)$. 

\begin{figure}[h!]
	\begin{picture}(216,80)
	\put(70,0){\includegraphics[scale=.9]{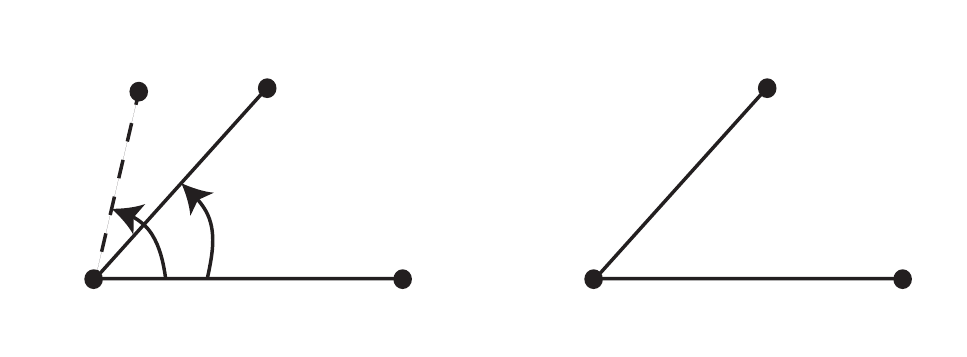}}
	\put(177,10){$d$}	
	\put(85,10){$e$}	
	\put(130,75){$ats(abc,def)$}	
	\put(110,65){$f$}	
	
	\put(305,10){$a$}	
	\put(217,10){$b$}	
	\put(275,68){$c$}
	
	\end{picture}
	\caption{}
	\label{}
\end{figure}

\begin{axiom} \label{atsproperties}
		$T(a,b,c) \wedge T(d,e,f) \rightarrow AC(a,b,c,d,e,ats(abc,def)) \wedge$
	
		 $SS(ats(abc,def),f,de) \wedge C(b,c,e,ats(abc,def))$.
\end{axiom}

	

In order to fully capture Hilbert's axiom we will need to define an angle transport opposite side construction . 

\begin{definition} \label{atodef}
	$ato(abc,def) \equiv ats(a,b,c,d,e,doub(f,e))$  [See Figure \ref{atofigure}]
\end{definition}

\begin{figure}[h!] 
	\begin{picture}(216,120)
	\put(70,0){\includegraphics[scale=.9]{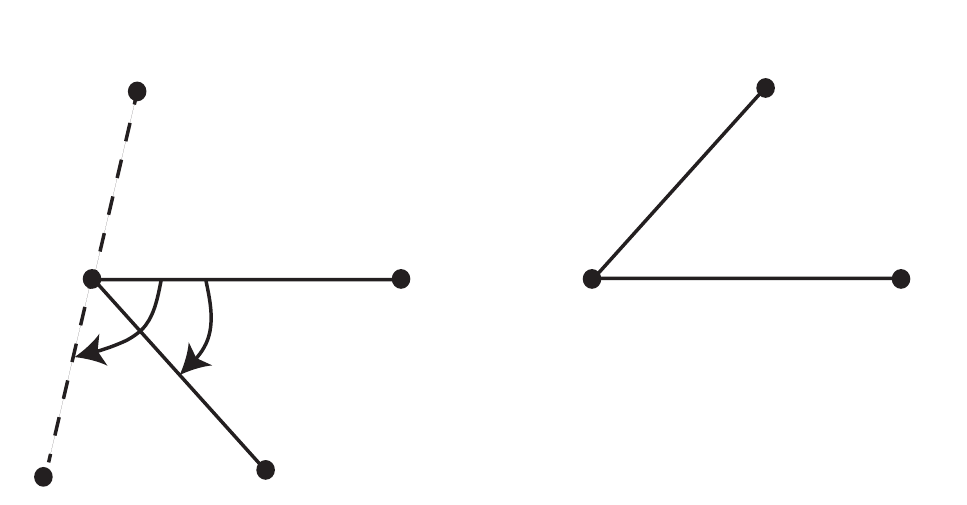}}
	\put(177,48){$d$}	
	\put(83,50){$e$}	
	\put(145,5){$ats(abc,def)$}	
	\put(85,3){$doub(f,e)$}	
	\put(110,105){$f$}	
	
	\put(305,48){$a$}	
	\put(217,48){$b$}	
	\put(275,106){$c$}
	
	\end{picture}
	\caption{}
\label{atofigure}
\end{figure}

By axiom \ref{intersectOO}, if $T(d,e,f)$, then $OS(doub(f,e),f,de)$

Lastly we need to guaranty that angle construction are unique up to segment congruence, angle congruence, and angle orientation. The following axiom does this. 

\begin{axiom} \label{atsunique}
	 $AC(abc,dex) \wedge SS(f,x,de) \wedge C(bc,ex) \rightarrow x = ats(abc,def)$
\end{axiom}

From these axioms one can prove the following theorem. This theorem can be interpreted as saying that two congruent angles that have a shared initial side and both of their terminal sides on the same side of the shared initial side have their terminal sides lying on the same ray.

 \begin{theorem} \label{conganglessamesidesameray}
	$AC(abc,abd) \wedge SS(c,d,ab) \rightarrow SD(b,c,d)$
\end{theorem}

We can fully capture Hilbert's properties of angle transportation and angle congruence by additionally assuming the following axiom. This axiom states that an angle constructed by shortening or extending the terminal side of an angle is congruent to the original angle.

\begin{axiom} \label{samerayanglecongruence}
	$SD(b,c,d) \wedge T(a,b,c) \rightarrow AC(abc,abd)$
\end{axiom}

We will now take a quick aside in order to define inequality relations for angles. 

\begin{definition}
	$abc < def \equiv Int(ats(abc,def), def)$ 
	
	
\end{definition}

Note if  $Int(d, abc)$, then $abd < abc$. 

For the proof of the feasibility of the crossbow construction we have need of the following Theorem.

\begin{theorem} \label{angleinequalitytransthm}
	If $abc < def$ and $def < ghi$, then $abc < ghi$. 
\end{theorem}

We omit the proof of this lemma given that one only has to check that the the correct line segment is interior to the correct angle. This simply boils down to using methods of proof about angle orientation that were covered extensively in Section \ref{angleorientationsec}.


In Hilbert's fourth congruence axiom Hilbert also states that every angle is congruent to itself. We in addition will also need to state that angle congruence is unaffected by reversing the order of triples. 

\begin{axiom}
	$AC(abc,abc) \wedge AC(abc,cba)$
\end{axiom}

\subsection{Hilbert's Congruence Axiom 5}
  
Hilbert's fifth congruence axiom is nearly identical to the following axiom. 
  
\begin{axiom}
$AC(abc,def) \wedge AC(abc,ghi) \rightarrow AC(def,ghi)$
\end{axiom} 
      
The last two axioms show that the angle congruence relation is reflexive, symmetric, and transitive. 
  
\subsection{Hilbert's Congruence Axiom 6}

Hilbert's sixth and final congruence axiom is a weak form of the side-angle-side congruence for triangles. The following axiom is identical to Hilbert's sixth congruence axiom. 


\begin{axiom}
	$T(a,b,c) \wedge T(d,e,f) \wedge C(ab,de) \wedge C(ac,df) \wedge AC(bac,edf) \rightarrow  AC(abc,def)$
	
	
\end{axiom}

\section{Sums and Differences, Supplemental Angles, Right Angles, and Triangle Congruences}

In this section we will be referring heavily to Hartshorne's text \cite{HartshorneBook}, in particular sections 8, 9 and 10. At this time we have captured (without the notion of a line) what is referred to as a Hilbert Plane which is a Euclidean plane without the notion of parallel lines and circles. In Section 8 of \cite{HartshorneBook} Hartshorne develops the concept of the difference of two segments. Note the sum of two segments was discussed above [See \cite{HartshorneBook} Propositions 8.2 and 8.3.] Fortunately he does not present this topic by defining operations on congruence classes of segments. His methods carry over to our system Thus his proofs that the differences of congruent segments are congruent can be incorporated into our system. In section 9 Hartshorne develops the sum and difference of angles. [See \cite{HartshorneBook} Proposition 9.4 and Exercise 9.1.] In a similar fashion his concepts and results that the sum and difference of congruent angle are congruent carry over to our system with simple modifications. In particular one important modification needed is using the SameDirection relation instead of the concept of rays. In this work we have not needed to invoke use of these results thus we not be explicitly developing them here. 

In the same section Hartshorne also defines and then develops the properties of supplemental angles, vertical angles, right angles, and in Section 10 he proves the various triangle congruence theorems such as Side-Angle-Side, Angle-Side-Angle and Side-Side-Side. We will define appropriate versions of supplemental angles, vertical angles, and right angles. From these definitions one can directly carry over Hartshorne's theorems about these topics as well as the triangle congruence theorems into corresponding proofs for our system. Thus we will not be developing the basic properties of supplemental, vertical and right angles or the basic triangle congruence theorems. We will however provide a uniform construction of a right angle, the uniform construction of erecting a perpendicular, and the uniform construction of dropping a perpendicular in the next section.

\begin{definition}
	If $B(a,b,c)$ and $T(a,c,d)$ we say angle $dba$ and angle $dbc$ are supplementary. 
\end{definition}

\begin{figure}[h!]
	\begin{picture}(216,30)
	\put(130,0){\includegraphics[scale=.9]{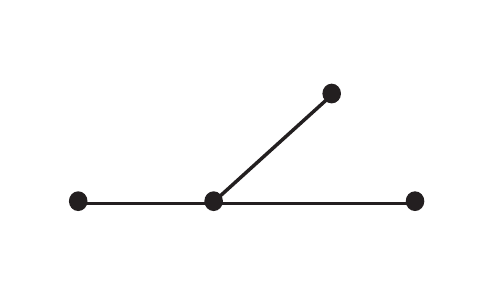}}
	\put(220,48){$d$}	
	\put(240,18){$c$}	
	\put(142,18){$a$}	
	\put(180,14){$b$}	
	\end{picture}
	\caption{}
	\label{}
\end{figure}

\begin{definition}
	If $B(a,c,b)$, $B(d,c,e)$ and $T(a,c,d)$ we say angle $acd$ is vertical to angle $bce$. 
\end{definition}

\begin{figure}[h!]
	\begin{picture}(216,100)
	\put(115,0){\includegraphics[scale=.9]{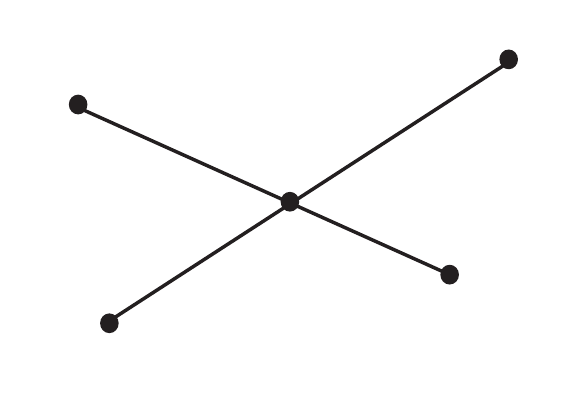}}
	\put(188,56){$c$}	
	\put(235,24){$b$}	
	\put(140,77){$a$}	
	\put(150,14){$d$}	
	\put(238,90){$e$}	
	\end{picture}
	\caption{}
	\label{}
\end{figure}

\begin{definition}
		Angle $abc$ is a right angle if it is congruent to one of its supplements.
	
\end{definition}


In \cite{HartshorneBook} it is proved that supplements of congruent angles are congruent, vertical angles are congruent, and all right angles are congruent.

\section{Midpoints, Perpendiculars, and Angle Bisectors}

From Hilbert's axioms there is a proof stating that given a line segment $ab$ one can construct a point $c$ such that triangle $acb$ is isosceles with $ac$ congruent to $bc$ (Theorem 10.2 in \cite{HartshorneBook}). This construction is a vital ingredient for proving the existence of midpoints and angle bisectors. The construction does have a case distinction depending on the angle congruence relationship between angles $dab$ and $dba$ for some point $d$ which is non-colinear with $a$ and $b$. The three case distinctions are $dab > dba$, $dab < dba$, or $dab \cong dba$. This means that there is a different procedure for constructing the point $c$ depending on this relationship. Because of this, stating such a theorem would be complicated in our system. Additionally, this will violate our goal of have only case free constructions. T





First we state a theorem pertaining to uniformly constructing a point which is our analog to dropping a perpendicular from a point off a segment to that given segment. 

\begin{theorem}\label{dropperpthm}
	Given $T(a,b,c)$ there is a uniform construction of a point $p$ such that $\widetilde{L}(a,p,b)$ and $apc$ or $bpc$ is right. 
\end{theorem}

\begin{figure}[h!]
	\begin{picture}(216,110)
	\put(70,0){\includegraphics[scale=.9]{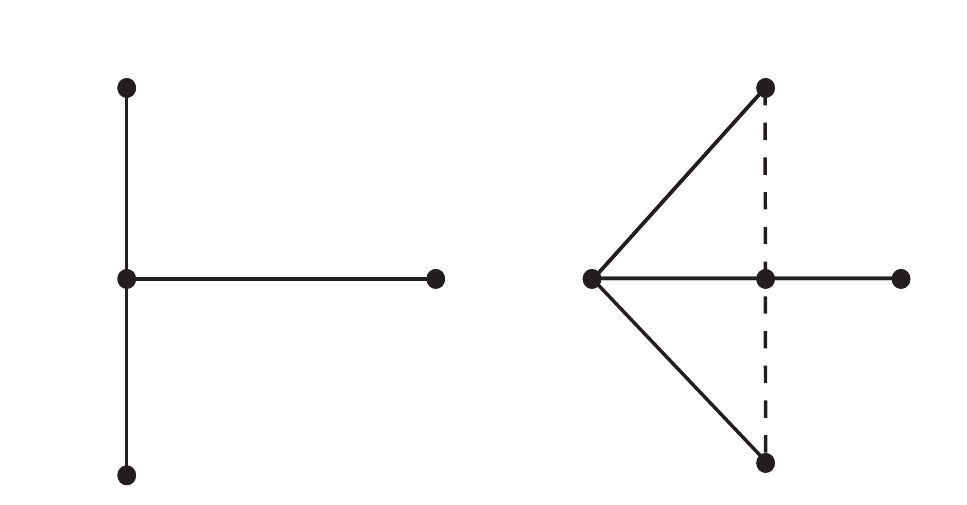}}
	
	\put(70,55){$p=a$}
	\put(90,0){$c'$}
	\put(90,105){$c$}	
	\put(180,45){$b$}	

	\put(212,55){$a$}
	\put(273,65){$p$}
	\put(273,105){$c$}
	\put(273,0){$c'$}
	\put(310,55){$b$}

	\end{picture}
	\caption{}
	\label{droppepfig}
\end{figure}

\begin{proof}
	(Reference Figure \ref{droppepfig}.) Let $T(a,b,c)$. Let $c' = ato(abc,abc)$. (Thus $c'$ is the point resulting from reflecting $c$ over $ab$.) Let $p=(b,cac')$. (Note that there is no requirement for $cac'$ to be an angle.) If $p=a$, then $bpc \cong bpc'$ and $L(c,p,c')$. Thus we can conclude $bpc$ is right. If $p \not= a$, then we can conclude $T(c,a,p)$ since $\widetilde{L}(a,p,b)$. Similarly we have $T(c',a,p)$. Given how $c'$ was constructed, we can conclude that triangle $cap$ is congruent to triangle $c'ap$ by Side-Angle-Side. Since $B(c,p,c')$ we have $apc$ is right. 
\end{proof}

In order to provide uniform constructions of erecting a perpendicular line segment, constructing a isosceles triangle with a given base, bisecting an angle, and others we will need to add only one more undefined construction. We introduce the midpoint construction denoted $mid(a,b)$. Its interpretation is simply a point $c$ that is colinear with $a$ and $b$ where $ac$ is congruent to $bc$ when $a$ and $b$ are distinct. The following axiom captures these properties. Note that we could have replace the conclusion $B(a, mid(a,b), b)$ with  $L(a, mid(a,b), b)$ in the first axiom. 

\begin{axiom}
	$a \not= b \rightarrow B(a, mid(a,b), b) \wedge C(ac,bc)$
\end{axiom}

\begin{axiom}
	$mid(a,a)=a$
\end{axiom}

We can now prove a theorem that states that one can create a right angle. 

\begin{theorem} \label{rightangleconstructionthm}
	Angle $\alpha d \beta$ where $d=   mid(\alpha, lf(\beta \gamma, \alpha \beta))$ is a right angle.
\end{theorem}

\begin{figure}[h!]
	\begin{picture}(216,120)
	\put(85,0){\includegraphics[scale=.9]{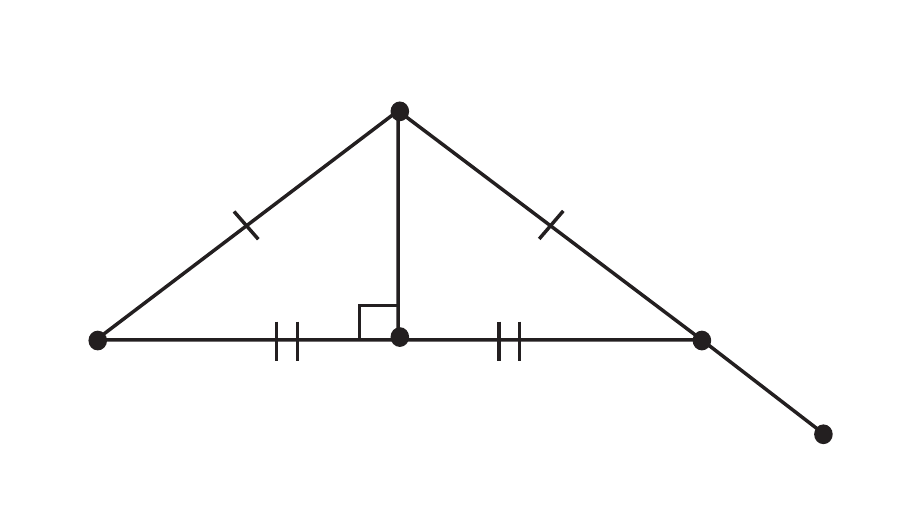}}
	\put(185,30){$d$}	
	\put(270,45){$c$}	
	\put(297,8){$\gamma$}	
	\put(188,108){$\beta$}	
	\put(100,35){$\alpha$}	
	\end{picture}
	\caption{}
	\label{rightangleconstruction}
\end{figure}

\begin{proof}
	Recall that $\alpha$, $\beta$, and $\gamma$ are three constants that by Axiom \ref{3noncolinearpoints} form a triangle. To construct our right angle we will first construct the point $lf(\beta \gamma, \alpha \beta)=c$. We then construct the point $m(\alpha, lf(\beta \gamma, \alpha \beta))=d$. See Figure \ref{rightangleconstruction}. Note that angle $\alpha d \beta$ and angle $cd\beta$ are supplemental. Also triangles $\alpha d \beta$ and angle $cd\beta$ are congruent and by the side-side-side triangle congruence theorem we have that the two supplemental angles which are congruent. Thus angle $\alpha d \beta$ is a right angle. 
\end{proof}

Given this theorem we can erect a perpendicular segment at a given point on a line segment on the same side as some point off of the line segment using our angle transport construction. We can also construct a perpendicular bisecting segment on either side of a given segment by using the midpoint construction and the angle transport construction. Furthermore we can prove that given a line segment $ab$ there is a point $c$ such that $abc$ is a isosceles triangle with base $ab$. This is accomplished by constructing a perpendicular bisector segment of $ab$ and letting $c$ be the endpoint of the bisector that is not the midpoint of $ab$. We can also construct a segment that bisects a given angle by first using the lay off construction to create congruent sides of the angle and then constructing the midpoint of the segment between the two non-shared endpoint of these sides. The segment from the vertex of the original angle to this midpoint can be proved to bisect the original angle [see figure \ref{anglebisector}].

\begin{figure}[h!]
	\begin{picture}(216,80)
	\put(115,0){\includegraphics[scale=.9]{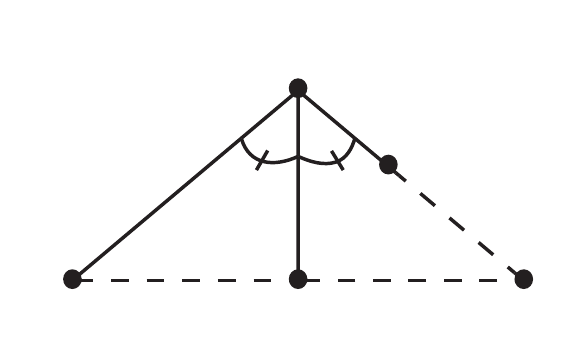}}
	\put(173,5){$m(d,c)$}	
	\put(221,49){$a$}	
	\put(257,10){$lf(ba,bc)$}	
	\put(190,75){$b$}	
	\put(122,15){$c$}	
	\end{picture}
	\caption{}
	\label{anglebisector}
\end{figure}

Lastly it should be noted that in the system devised by Tarski the construction of angle bisetors, midpoints and perpendiculars relied on the constructions involving circles (compass constructions) and a circle circumscription construction which is logically equivalent to Eulcid's parallel postulate. Euclid also invokes compass constructing for some of this constructions. 

\section{Parallelograms and Parallel Segments}\label{Parallelsection}

This section, although modified and extended, is based off of Suppes' work in \cite{SuppesAxioms} which was heavily influenced by the work of Szmielew in \cite{SzmielewBook}. In \cite{SuppesAxioms} Suppes defined what he called a constructive affine plane. The system was quantifier-free and the only objects were points. The primitive constructions were that of finding the midpoint between two points and the (directed) doubling of a segment. From these constructions Suppes was able to define when two line segments were parallel. We will follow Suppes in his definition of parallel lines segments, but it should be noted that that we will be extending may of his results since our system contains the concepts of angle orientation and angle congruence. 

For his constructive affine plane, Suppes assumed axioms pertaining to the relation of three points being colinear and properties of the midpoint and doubling constructions. Suppes used an undefined non-strict linear relation. We defined this relation from a strict linear relation in definition \ref{nonstrictcolinear}. 

All but one of Suppes' axioms about the midpoint and doubling constructions are provable in our system. We list these results in the following theorem. 

\begin{theorem} \hspace{1in}
	
	\begin{enumerate}
		\item $mid(a,a)=a$
		\item $mid(a,b)=mid(b,a)$
		\item $mid(a,b)=mid(a,b')\rightarrow b=b'$
		\item $doub(a,b)=doub(b,a) \rightarrow a=b$
		\item $doub(a,b)=doub(a,b') \rightarrow b=b'$
		\item $doub(a,b)=doub(a',b) \rightarrow a=a'$
		\item $mid(a,doub(a,b))=b$
		\item $\widetilde{L}(a,b,mid(a,b))$
	\end{enumerate}
\end{theorem}

As we said above, there is one crucial axiom of Suppes that we can not prove in our system at this point. At this time we have captured many of the results pertaining to a Hilbert Plane and have thus made no reference to parallel lines. If we wish to fully capture the nature of the plane and in particular the properties of parallel line segments we will need another assumption. Given the objective of this paper, we cannot use an assumption like Euclid's fifth or Playfair's axiom. Based off of the work of Suppes and Szmielew, we will show that by assuming the midpoint construction has the property of bisymmetry (which was called bicommutativity by Suppes and Szmielew) we can capture all the necessary features of parallel line segments. 

We assume the following axiom. 

\begin{axiom}

	$mid(mid(a,b),mid(c,d))=mid(mid(a,c),mid(b,d))$
\end{axiom}


We now define a parallelogram relation for four ordered points. 

\begin{definition}
	Points $a$, $b$, $c$, and $d$ form a parallelogram, denoted $P(a,b,c,d)$, if and only if $mid(a,c) = mid(b,d)$. [See Figure \ref{parallelogramfigure}]
\end{definition}

\begin{figure}[h!]
	\begin{picture}(216,150)
	\put(73,0){\includegraphics[scale=.9]{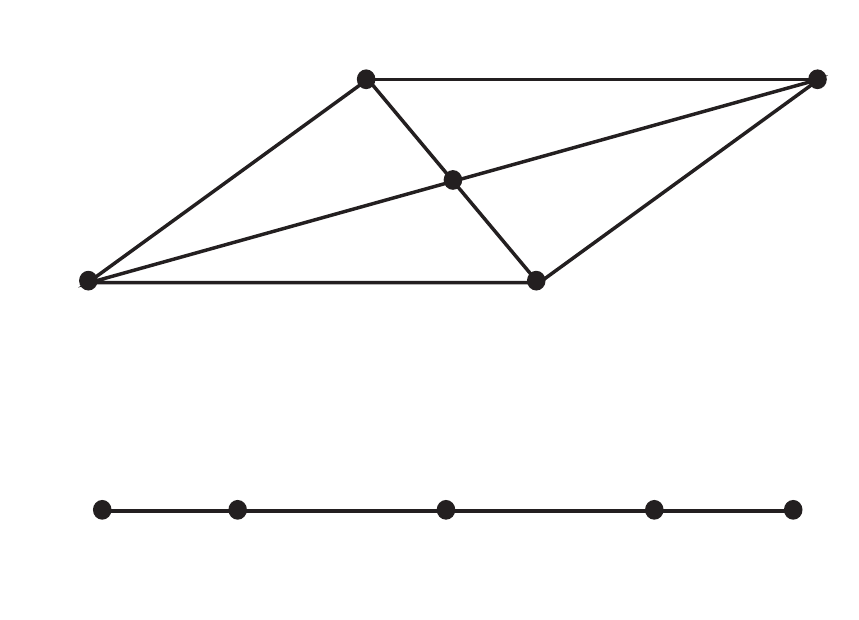}}

	\put(212,77){$a$}	
	\put(90,75){$b$}	
	\put(160,145){$c$}	
	\put(290,140){$d$}	
	\put(200,107){\footnotesize{$m(a,c)$}}	
	\put(155,114){\footnotesize{$m(b,d)$}}

	\put(162,11){\footnotesize{$m(a,c)=m(b,d)$}}	
	\put(95,17){$d$}
	\put(280,15){$b$}		
	\put(133,33){$a$}	
	\put(240,33){$c$}	
	\end{picture}
	\caption{}
	\label{parallelogramfigure}
\end{figure}

Note that we are allowing for flat parallelograms where $a$, $b$, $c$, and $d$ are all pairwise colinear with the midpoint $mid(a,c) = mid(b,d)$. (Which was not allowed by Suppes, but was allowed by Szmielew.) From the above definition one can prove the following results. We reference Suppes' analogous theorems, but do note that he proofs, while containing the necessary ingredients for proving these theorem, would need slight modifications since our definition of the parallelogram relation is not identical to Suppes'.

\begin{theorem} 
$P(a,b,c,d) \rightarrow P(c,b,a,d)$
\end{theorem}

\begin{theorem} (\cite{SuppesAxioms} Theorem 5)
$b \not= c \wedge P(a,b,c,d) \rightarrow \lnot P(a,c,b,d)$
\end{theorem}

\begin{theorem} (\cite{SuppesAxioms} Theorem 9 and 10)
	$P(a,b,c,d) \rightarrow P(b,c,d,a) \wedge P(c,d,a,b) \wedge P(d,a,b,c) \wedge P(c,b,a,d) \wedge P(b,a,d,c) \wedge P(a,d,c,b) \wedge P(d,c,b,a)$
\end{theorem}

The following two theorems show that given three distinct non colinear points one can construct a fourth point to form a parallelogram and that this point is unique. 

\begin{theorem} (\cite{SuppesAxioms} Theorem 7)
	$P(a,b,c,doub(b,m(a,c)))$ [See Figure \ref{parallelogramconstruction}.]
\end{theorem}

\begin{figure}[h!]
	\begin{picture}(216,80)
	\put(73,0){\includegraphics[scale=.9]{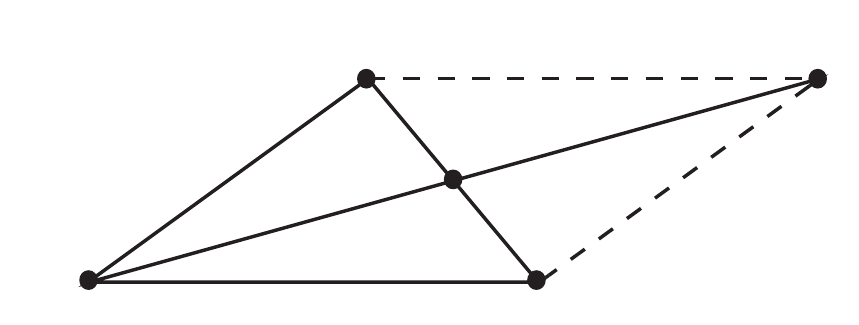}}
	
	\put(212,2){$a$}	
	\put(90,0){$b$}	
	\put(160,70){$c$}	
	\put(290,65){$d(b,m(a,c))$}	
	\put(200,32){\footnotesize{$m(a,c)$}}

	\end{picture}
	\caption{}
	\label{parallelogramconstruction}
\end{figure}

\begin{theorem} (\cite{SuppesAxioms} Theorem 6)
	$P(a,b,c,d) \wedge P(a,b,c,d') \rightarrow d=d'$ 
\end{theorem}

The next theorem will be used to show that the parallel relation between line segments is transitive, but for now it will be stated in terms of parallelograms. 

\begin{theorem} (\cite{SuppesAxioms} Theorem 12)
	$P(a,b,c,d) \wedge P(c,d,e,f) \rightarrow P(a,b,f,e)$ [See Figure \ref{parallelogramtrans}]
\end{theorem}

\begin{figure}[h!]
	\begin{picture}(216,90)
	\put(73,0){\includegraphics[scale=.9]{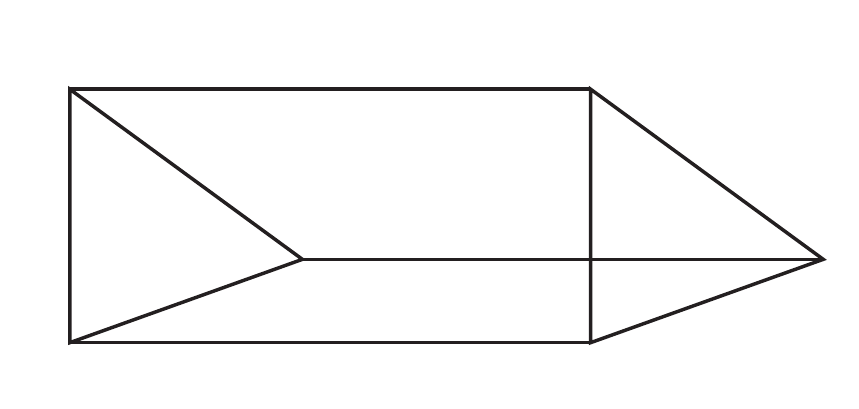}}
	
	\put(225,5){$b$}	
	\put(229,85){$f$}	
	\put(83,12){$a$}	
	\put(83,83){$e$}	
	\put(152,42){$d$}	
	\put(285,42){$c$}

	\end{picture}
	\caption{}
	\label{parallelogramtrans}
\end{figure}

We now define a relation for when two line segments are parallel. 

\begin{definition} (\cite{SuppesAxioms} Definition 3)
	$ ab \parallel cd \equiv T(a,b,c) \wedge c \not= d  \wedge \widetilde{L}(c,doub(b,mid(a,c)),d) $
	
	
\end{definition}

\begin{figure}[h!]
	\begin{picture}(216,80)
	\put(73,0){\includegraphics[scale=.9]{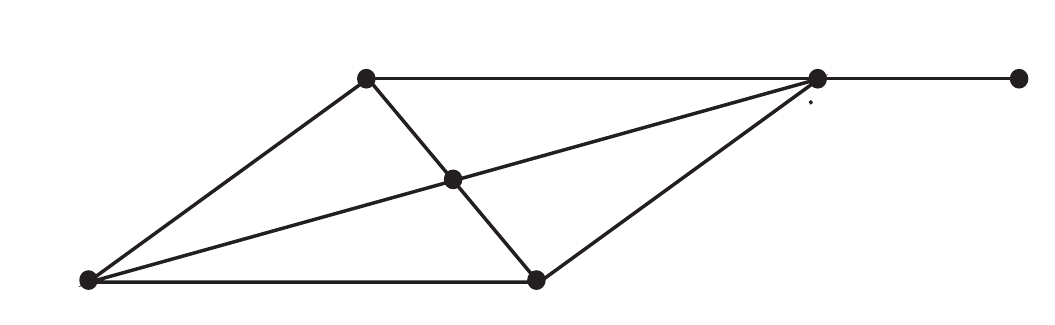}}
	
	\put(212,2){$a$}	
	\put(90,0){$b$}	
	\put(160,70){$c$}	
	\put(260,70){$d(b,m(a,c))$}	
	\put(200,32){\footnotesize{$m(a,c)$}}
	\put(342,66){$d$}

	\end{picture}
	\caption{}
	\label{}
\end{figure}

Note the definition does not allow for $ab \parallel cd$ if $\widetilde{L}(a,b,c)$. 

In Theorem 16 of \cite{SuppesAxioms}, Suppes lists a set of properties of parallel line segments without providing proofs. We have unfortunately not been able to produce proofs for all of these properties from his system. Fortunately are able to prove these properties in our system with the aid of angle congruence. 

\begin{theorem} \label{parallelsegthm} \hspace{1in}
	\begin{enumerate}
		\item $T(a,b,c) \wedge P(a,b,c,d) \rightarrow ab \parallel cd \wedge bc \parallel ad$
		\item If $ab \parallel cd$, then $a$, $b$, $c$, and $d$ are all distinct.
		\item $ab \parallel cd \rightarrow cd \parallel ab$
		\item $ab \parallel cd \rightarrow ba \parallel cd$
		\item $ab \parallel pq \wedge ab \parallel rs \wedge T(p,q,r) \rightarrow pq \parallel rs$
	\end{enumerate}
\end{theorem}

\textbf{\textit{Change of Notation:}} Before proving this theorem we will make a change of notation for the ease of the reader. We will use $ab \cong cd$ in place of $C(ab,cd)$ and we will use $abc \cong def$ in place of $AC(abc,def)$. 

\begin{proof}
	To prove part 1, one simply has to verify the non-linearity of various triples of points. To prove part 2, one uses the non-colinearity of points to prove that they are distinct. 
	
	To prove part 3 (see Figure \ref{paralletheorempart3} ), let $doub(b,mid(c,a))=x$. Since $ab \parallel cd$, we have $\widetilde{L}(c,x,d)$ and $c \not= d$. Now let $doub(d,mid(c,a))=y$. We need to show that $\widetilde{L}(a,b,y)$ and $T(c,d,a)$ (we already know that $a \not= b$).
	
	 Note that $ac$, $bx$, and $yd$ have a shared midpoint. Let this midpoint be called $z$. Since $T(a,b,c)$, using Theorem \ref{betweensideint} and Axiom \ref{SOray} we can show that $z$ and $x$ are in the interior of $abc$. Thus $T(a,b,z)$. By Theorem \ref{betweensideint} we can also infer that $z$ is in the interior of $cxa$. Thus $T(c,x,a)$ and $T(c,x,z)$. From $T(a,b,z)$ we can infer that $T(a,y,z)$. From $T(c,x,z)$ we can infer that $T(c,d,z)$ and from $T(c,x,a)$ we can infer that $T(c,d,a)$. Further more, if $d \not= x$ and $b \not= y$, then we can show that $T(d,x,z)$ and $T(y,b,z)$. 
	
	Since we know $T(c,d,a)$, if $y=b$, then $\widetilde{L}(a,b,y)$ and we can done. Thus assume $y \not= b$. If $x = d$ by properties of midpoints and extension constructions we can show that $b=y$ which would be a contradiction. Thus we also have $x \not = d$. 
	
	So we have $L(c,x,d)$. Therefore one of these three points must be between the other two. Without loss of generality let $B(c,x,d)$. Note that angles $zxc$ and $zxd$ are supplemental. Using vertical angles, segment congruence, and side-angle-side, one can show that $zby \cong zxd$ and $zba \cong zxc$. Note that since $c$ and $d$ are on opposite sides of $bx$, by axiom \ref{intOO} and results about the same side relation (theorems \ref{sideslinethm1} and \ref{sideslinethm2}), we have that $a$ and $y$ are on opposite sides of $bx$. Thus $OO(zby,zba)$. Note that $zba$ and $zb(doub(a,b))$ are supplemental. Since $zba \cong zxc$ we have $zxd \cong zb(doud(a,b))$. Thus $zby \cong zb(doub(a,b))$ and $y$ and $doub(a,b)$ are on the same side of $bz$. Therefore, by the uniqueness of angle constructions, we have $SD(b,y,doub(a,b))$. Thus $L(a,b,y)$. 
	
	\begin{figure}[h!]
		\begin{picture}(216,125)
		\put(60,0){\includegraphics[scale=.9]{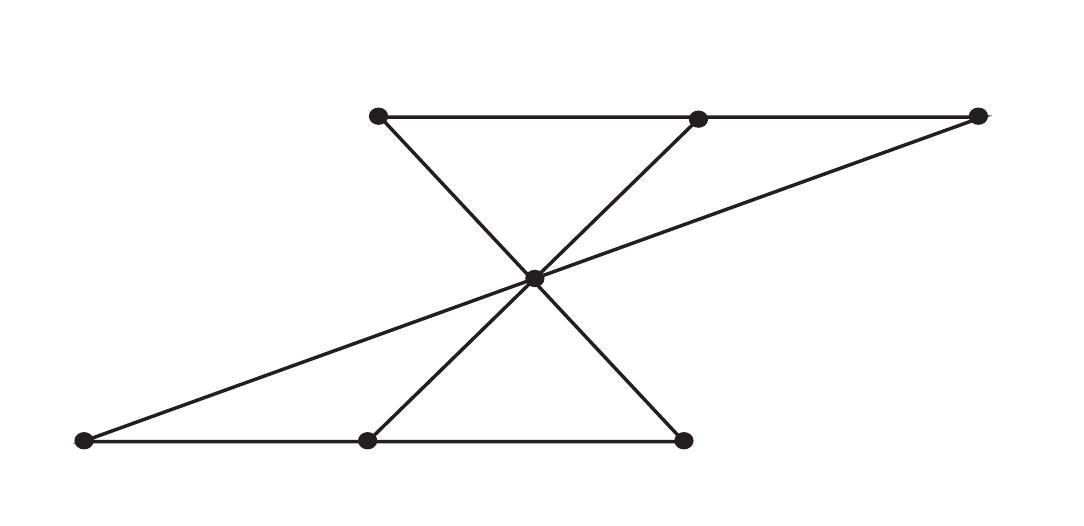}}
		
		\put(240,12){$a$}	
		\put(75,12){$y$}
		\put(150,10){$b$}
		\put(150,107){$c$}	
		\put(243,110){$x$}	
		\put(205,59){\footnotesize{$m(a,c)=z$}}
		\put(320,105){$d$}

		\end{picture}
		\caption{}
		\label{paralletheorempart3}
	\end{figure}
	
	To prove part 4 (See Figure \ref{paralletheorempart4}), let $doub(b,mid(a,c))=x$ and let $doub(a,mid(b,c))=y$. By the assumption $\widetilde{L}(c,x,d)$, $c \not= d$, and $T(a,b,c)$. Note $T(b,a,c)$. We wish to show that $\widetilde{L}(y,c,d)$. By theorem \ref{parallelogramconstruction} $P(c,x,a,b)$ and $P(a,b,y,c)$ and by theorem \ref{parallelogramtrans} we have $P(c,x,c,y)$. Thus $mid(x,y)=mid(c,c)=c$. Thus we can infer that $L(x,c,y)$ and thus $\widetilde{L}(y,c,d)$. 
	
		\begin{figure}[h!]
		\begin{picture}(216,125)
		\put(100,0){\includegraphics[scale=.9]{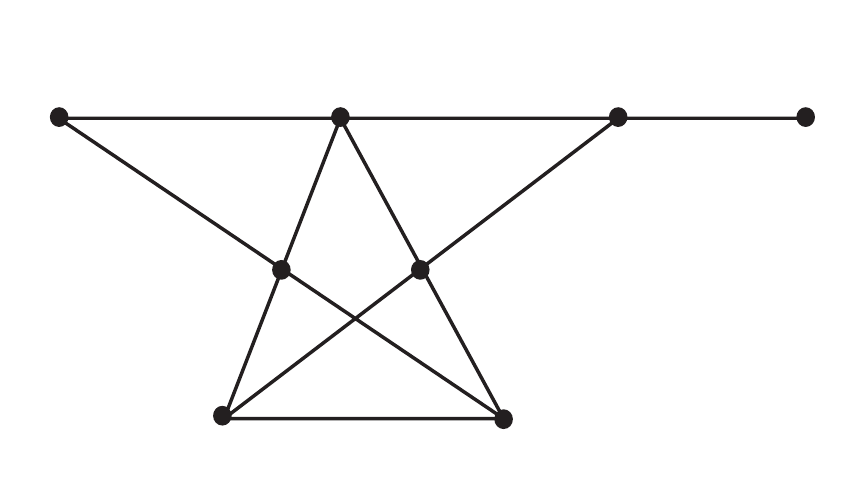}}
		
		\put(235,12){$a$}	
		\put(105,100){$y$}
		\put(150,10){$b$}
		\put(185,100){$c$}	
		\put(255,100){$x$}	
		\put(217,55){\footnotesize{$m(a,c)$}}
		\put(137,53){\footnotesize{$m(b,c)$}}
		\put(310,100){$d$}

		\end{picture}
		\caption{}
		\label{paralletheorempart4}
	\end{figure}
	
	To prove part 5 (see Figure \ref{paralletheorempart5}), let $doub(b,mid(a,p))=x$ and let $doub(b,mid(a,r))=y$. We know that $\widetilde{L}(p,x,q)$ and $\widetilde{L}(r,y,s)$. We are also assuming $T(p,q,r)$. From this last assumption we can infer that $T(x,p,r)$. Using part 2 of this theorem we can also infer that $r \not= s$ and $p \not= q$. 
	
	Since $P(p,x,a,b)$ and $P(a,b,r,y)$ by Theorem \ref{parallelogramconstruction} we have $P(p,x,y,r)$ by theorem \ref{parallelogramtrans}. Since $P(p,x,y,r)$, $T(x,p,r)$, and $r \not= s$  we have $px \parallel rs$. By part 3 we have $rs \parallel px$. Thus $T(s,r,p)$. From this and the fact that $\widetilde{L}(p,x,q)$ and $p \not= q$ we have $rs \parallel pq$ and by part 3 again we have $pq \parallel rs$.

		\begin{figure}[h!]
		\begin{picture}(216,140)
	\put(125,0){\includegraphics[scale=.9]{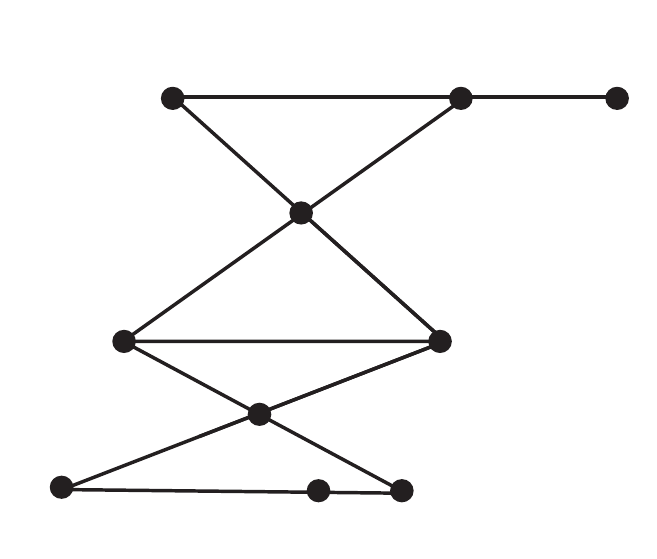}}
	
	\put(235,12){$y$}	
	\put(205,8){$s$}	
	\put(243,50){$a$}	
	\put(147,50){$b$}	
	\put(243,125){$x$}	
	\put(290,125){$q$}		
	\put(165,125){$p$}	
	\put(132,12){$r$}
	\put(213,85){\footnotesize{$m(a,p)$}}
	\put(205,35){\footnotesize{$m(a,r)$}}

	\end{picture}
	\caption{}
	\label{paralletheorempart5}
\end{figure}

\end{proof}

The following theorem shows that given two parallel line segments one can create new pairs of parallel line segments by simply extending or truncating one or both of the original lines. 

\begin{theorem} \label{parallelsegthm2}\hspace{1in}
	
	\begin{enumerate}
		\item 	$ab \parallel cd \wedge \widetilde{L}(c,d,x) \rightarrow ab \parallel cx$
		\item 	$ab \parallel cd \wedge \widetilde{L}(a,b,y) \rightarrow ay \parallel cd$
	\end{enumerate}
\end{theorem}

\begin{proof}
	Part 1 is easily inferred from the definition of parallel segments and non-strict linearity. Part 2 can be shown by invoking part 3 of the perivous theorem and part 1 of this theorem. 
\end{proof}

Given a line segment $ab$ and another point $c$, the following theorem states that all line segments parallel to $ab$ with endpoint $c$ must be colinear. The proof is simply a reinterpretation of the definition of parallel line segments. 

\begin{theorem}
	
	$T(c,d,doub(b,mid(a,c,))) \rightarrow ab \nparallel cd$
\end{theorem}

From part 5 of Theorem \ref{parallelsegthm} we can prove the following lemma. 

\begin{lemma} \label{parallelimplieslinear}
	If $ab \parallel cd$ and $ab \parallel rs$ and $pq \nparallel rs$, then 
	$\widetilde{L}(p,q,r)$.
\end{lemma}




Up to this point we have been allowing for flat parallelograms. In order to refer to traditional results about parallelograms we will now define a non-flat parallelogram. 

\begin{definition}
	$\widehat{P}(a,b,c,d) \equiv P(a,b,c,d) \wedge T(a,b,c)$
\end{definition}

Using common well-known methods, such as triangle congruence theorems, one can prove
the following. 

\begin{theorem} \label{oppositesidesanglesnonflatparallelogram}
	The opposite sides of a non-flat parallelogram are congruent and the opposite angles of a non-flat parallelogram are congruent. (Where opposite sides and opposite angles are defined traditionally.)
\end{theorem}

\subsection{Alternate Interior Angle Theorem and its Converse}

\begin{theorem}Converse of the Alternate Interior Angle Theorem
	
\hspace{.5in}	$ab \parallel cd \wedge OS(b,d, ac) \rightarrow cab \cong acd$
\end{theorem}

\begin{figure}[h!]
	\begin{picture}(216,140)
	\put(75,0){\includegraphics[scale=.9]{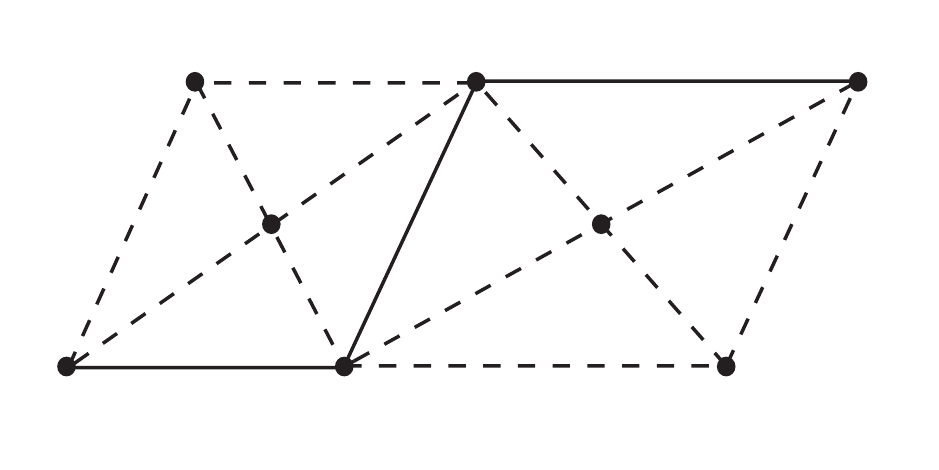}}
	
	\put(265,15){$x$}
	\put(85,15){$b$}	
	\put(160,15){$a$}	
\put(200,105){$c$}	
\put(300,105){$d$}		
\put(130,105){$y$}	


	\end{picture}
	\caption{}
	\label{converseAIAT}
\end{figure}

\begin{proof}
	See Figure \ref{converseAIAT}. Let $doub(c,mid(a,d)) = x$ and let $doub(a, mid(b,c)) = y$. By the definition of parallel segments we know that $L(b,a,x)$. Similarly we know that $L(y,c,d)$. 
	
	Note that $Int(mid(a,y),acy)$ by theorem \ref{betweensideint}. Since $P(a,b,y,c)$, $mid(a,y)=mid(b,c)$ and we can infer that $Int(mid(b,c), acy)$. Thus $Int(b, acy)$ and $SS(b,y,ac)$. Therefore $OS(y,d, ac)$. We can then infer that $B(y,c,d)$ for if not then we would contradict $OS(y,d, ac)$. By a similar argument we can conclude that $B(b,a,x)$. 
	
	We can show $\widehat{P}(b,y,c,a)$ and $\widehat{P}(c,a,x,d)$. From Theorem \ref{parallelogramtrans} we can show $\widehat{P}(b,y,x,d)$. Thus by Theorem \ref{oppositesidesanglesnonflatparallelogram} and Axiom \ref{samerayanglecongruence} $acd \cong axd \cong bxd \cong bx'd \cong bx'c \cong bac \cong cab$. 

\end{proof}

\begin{theorem}Alternate Interior Angle Theorem
	
\hspace{.5in}	$cab \cong acd \wedge OS(b,d,ac) \rightarrow ab \parallel cd$
\end{theorem}

\begin{figure}[h!]
	\begin{picture}(216,140)
	\put(75,0){\includegraphics[scale=.9]{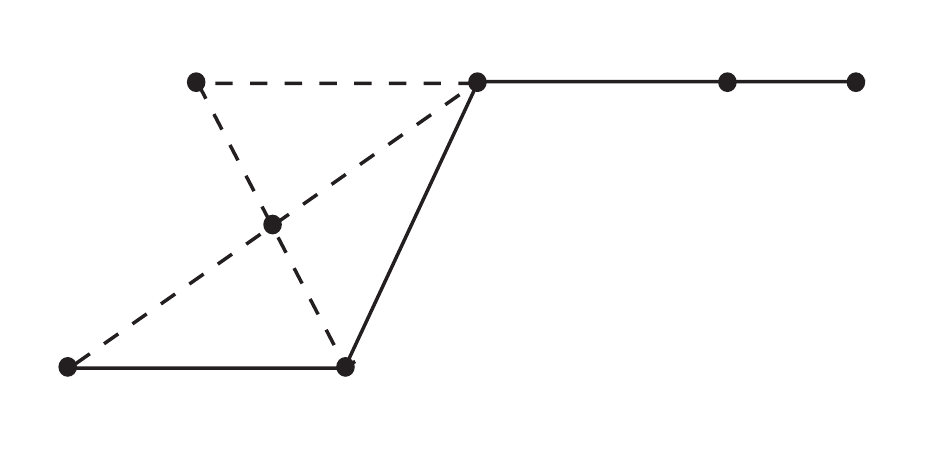}}
	
	\put(85,15){$b$}	
	\put(160,15){$a$}	
	\put(200,105){$c$}	
	\put(300,105){$d$}	
	\put(265,105){$x$}		
	\put(130,105){$y$}	
	

	\end{picture}
	\caption{}
	\label{AIAT}
\end{figure}

\begin{proof}
	See Figure \ref{AIAT}. Let $doub(a,mid(bc))=y$. By arguments similar to one used in the previous proof we can show $SS(b,y,ac)$ and thus $OS(y,d,ac)$. Let $doub(y,c)=x$. We can show $SS(x,d,ac)$. Since $L(y,c,x)$ and $\widehat{P}(a,b,y,c)$ we can infer that $ab \parallel cx$. 
	
	We now what to show that $ab \parallel cd$. We will do this by showing $\widetilde{L}(c,x,d)$ and invoking Theorem \ref{parallelsegthm2}. We have enough to show $OS(b,x,ac)$. Therefore since $ab \parallel cx$ we have $cab \cong acx$ by the previous theorem. Additionally since $cab \cong acd$, we have $acx \cong acd$. Thus $SS(x,d,ac)$ and $acx \cong acd$ and by Theorem \ref{conganglessamesidesameray} we have $SD(c,x,d)$. We can then conclude $\widetilde{L}(c,x,d)$ and we are done. 
	
\end{proof}

In standard elementary geometry the Alternate Interior Angle Theorem is provable without invoking the parallel postulate and the converse of the statement is proved using the original statement and the parallel postulate. Interestingly we needed to invoke our version of the parallel postulate to prove the Alternate Interior Angle Theorem by using its converse to prove it. This shows that there is indeed a striking difference between our definition of parallel line segments and the standard definition of parallel lines.

\subsection{Convex Quadrilaterals}

The property of convexity is an important one in planar geometry. What follows is a definition for convex quadrialterials and a few results pertaining to this definition. 

\begin{definition}
	We say $abcd$ is a convex quadrilateral if $Int(d,abc)$ and $Int(a,bcd)$.
	
\end{definition}

\begin{theorem}
	If $abcd$ is a convex quadrilateral then  $Int(b, cda)$, and $Int(c,dab)$.
\end{theorem}

\begin{proof}

See Figure \ref{convexquad}. Since $Int(d,abc)$ we have $B(a,cb(d,abc),c)$ and $SD(b, cb(d,abc),d)$ by Theorem \ref{crossbartheorem}. Additionally we know $T(d,b,a)$. Let $cb(d,abc)=x$. Since $Int(a,bcd)$ we have $B(b,cb(a,bcd),d)$ and $SD(c,cb(a,bcd),a)$ by Theorem \ref{crossbartheorem}. We also have $T(c,d,a)$. Let $cb(a,bcd)=y$. Suppose $x \not= y$. Then by Axiom \ref{uniqueline} we can show $L(a,b,d)$. This is a contradiction. Thus $x=y$. Therefore $B(b,x,d)$ and $B(a,x,c)$. Since $T(c,d,a)$ and $T(d,a,b)$, by Theorem \ref{betweensideint} we have $Int(x, cda)$ and $Int(x, dab)$. Using Axiom \ref{SOray} we can infer $Int(b, cda)$, and $Int(c,dab)$


\end{proof}

\begin{figure}[h!]
	\begin{picture}(216,100)
	\put(92,0){\includegraphics[scale=.9]{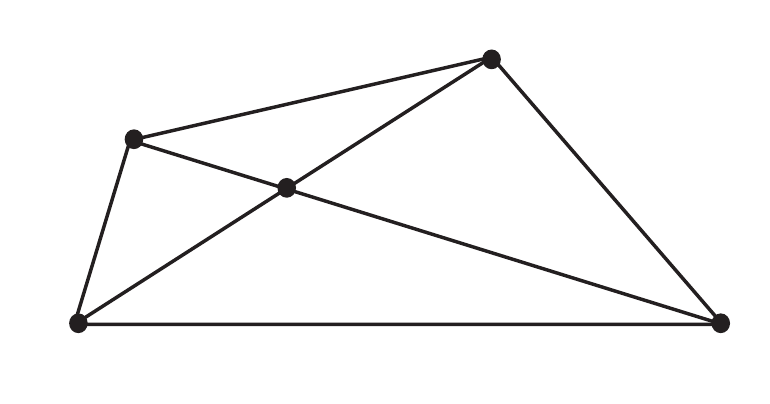}}

	\put(102,15){$b$}	

	\put(118,70){$c$}	

	\put(285,15){$a$}	

	\put(223,90){$d$}

	\put(175,54){$x=y$}

	\end{picture}
	\caption{}
	\label{convexquad}
\end{figure}

We now state two observations that will be helpful in the following section. 

\begin{observation} \label{ObservationQuad1}
	If $abcd$ is a convex quadrilateral then $SS(c,d,ab)$. 
\end{observation}

\begin{observation} \label{ObservationQuad2}
	If $\widehat{P}(a,b,c,d)$, then $abcd$ is a convex quadrilateral. 
\end{observation}

Given $B(b,x,d)$ and $B(a,x,c)$ by using triangle congruence theorems, theorems about vertical angles, and the Alternate Interior Angle Theorem and its converse one can prove the following statements. 

\begin{theorem} \label{convexquadthm} \hspace{1in}
	\begin{enumerate}
		\item If $abcd$ is a convex quadrilateral whose opposite sides are congruent, then $ab \parallel cd$ and $bc \parallel ad$. Thus $\widehat{P}(a,b,c,d)$.
		\item If $abcd$ is a convex quadrilateral whose opposite angles are congruent, then $ab \parallel cd$ and $bc \parallel ad$. Thus $\widehat{P}(a,b,c,d)$.
	\end{enumerate}
\end{theorem}

\subsection{The Feasibility of the Crossbow Construction}

In the introduction we claimed that given $a$ and $c$ on opposite sides of $bd$ the line segment from $b$ to $cb(d, abc)$ is less than $ab$ or $cb$. We will now sketch the outline of the proof of this fact. 

We have need for the following lemma.

\begin{lemma} \label{LengthSidesTriangle}
	If $T(a,b,c)$ and $abc < acb$, then $ac < ab$. 
\end{lemma}

In Euclid's \textit{The Elements} this is Proposition 19. It states that in a triangle the greater side is opposite the greater angle. The proofs of this result and the results needed to prove this result carry over to our system without difficulty. 

The following theorem is the result we need to justify our claim about the feasibility of the crossbow construction. 

\begin{theorem} \label{triangelsegmentinequality}
	If $T(a,b,c)$ and $B(a,x,c)$, then $bx < ba \lor bx < bc$.
\end{theorem}

\begin{figure}
	\begin{picture}(216,80)
	\put(37,0){\includegraphics[scale=.9]{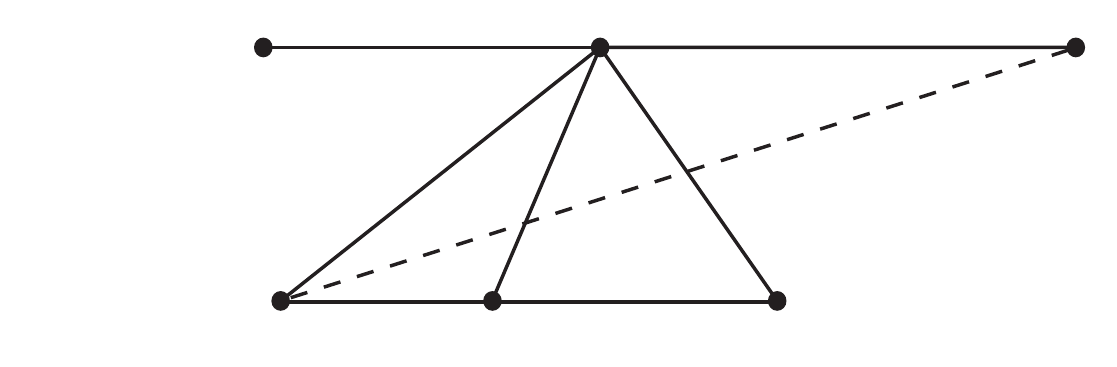}}
	
	\put(240,8){$c$}		
	\put(100,87){$e$}	
	\put(160,6){$x$}	
	\put(103,6){$a$}	
	\put(192,87){$b$}	
	\put(315,87){$d$}

	\end{picture}
	\caption{}
	\label{feasiblecrossbowfigure}
\end{figure}

\begin{proof}
	See Figure \ref{feasiblecrossbowfigure}. Let $d = doub(a, mid(b,c))$ and let $e=doub(d,b)$. First we want to show that $cbd < xbd$. Note $cabd$ is a parallelogram and thus a convex quadrilateral. There by Theorem \ref{convexquadthm} we have $Int(c, abd)$. We also have $Int(x,abc)$ by Axiom \ref{betweensideint}. By observing the angle orientations of various angles, one can show $Int(c, xbd)$. Thus $cbd < xbd$. One can then modify these methods to show that $abe < xbe$. By the converse of the Alternate Interior Angle Theorem we know  $cbd \cong bcx$, $abe \cong bax$, $axb \cong dbx$, and $cxb \cong ebx$. Thus we have $xab < cxb$ and $xcb < axb$. There are three cases to consider. If $xab \cong xcb$, then $xcb < cxb$ and by the previous lemma above $bx < bc$. If $xab < xcb$, then by the Theorem \ref{angleinequalitytransthm} above we have $xab < axb$. Thus $bx < ba$. If $xcb < xab$, then we have $xcb < cxb$. This implies $bc < bc$. 
\end{proof}

\begin{observation}
	In the previous proof it was shown that $cab \cong eba$ and $acb \cong dba$. If one so chooses they could defined the angles $eba$, $abc$, and $dbc$ as a so called linear triple and observe that the angles of the triangle are congruent to these angles in a one to one correspondence. This would be our analog to the result that the angles of a triangle sum to 180 degrees. 
\end{observation}

\subsection{The Triangle Inequality Theorem}

In Section \ref{CompassConstructions} we will need to apply a version of the Triangle Inequality Theorem. Before we state and proof such this theorem, we will need to state and prove our version of the Exterior Angle Theorem. 

\begin{theorem} \label{ExteriorAngleThm}
	The Exterior Angle Theorem:
	
	$T(a,b,c) \rightarrow ac (doub(b,c)) > abc \wedge ac (doub(b,c)) > acb$
\end{theorem}

\begin{proof}
	The proof is identical to Euclid's proof of Proposition 16 in Book I of \textit{The Elements}. The one critical issue is to verify that certain points are on the same side of a segment when claiming that a particular angle is less than another. These claims are verifiable in our system and the details are omitted. 
\end{proof}

Below we have stated our version of the Triangle Inequality Theorem. The key distinction between our statement is that we have avoided defining the sum operation on the congruence classes of line segments. Because of this, we have written out the details of the proof even though the proof is relatively standard. 

\begin{theorem} \label{TriangleInequalityThm}
	The Triangle Inequality Theorem:
	
	$T(a,b,c) \rightarrow a (ext(ab,bc)) > ac$
\end{theorem}

\begin{proof}
	
	Let $d$ be the point constructed by dropping a perpendicular from $b$ to $ac$. There are three cases to consider. 
	
	Case 1: If $B(a,c,d)$, then since $adb$ is right by Theorem \ref{ExteriorAngleThm} $abd > adb$. Thus $ab >ad$ by Lemma \ref{LengthSidesTriangle}. Therefore $a (ext(ab,bc)) > ac$. 
	
	Case 2: If $B(d,a,c)$ then a similar argument to the previous case will show $a (ext(ab,bc)) > ac$.
	
	Case 3: If $B(a,d,c)$ then by Lemma \ref{LengthSidesTriangle} $ab > ad$ and $bc > dc$. We can use the layoff construction to construct points $s$ and $t$ such that $\widetilde{L}(s,a,c)$, $ds \cong ab > da$, $\widetilde{L}(t,a,c)$, $dt \cong bc > dc$. We can then infer $B(s,a,d)$ and $B(a,d,t)$. We then have $st \cong a (ext(ab,bc))$. Therefore $a (ext(ab,bc)) > ac$.

\end{proof}

\subsection{Equidistant Segments and Extensions of Non-equidistant Segments}

In this section we will derive an analogous result to the idea that two line are parallel if and only if they are equidistant. We will also show that if two segments are not equidistant then when one extends one of the line segments in one direction the segments grow more distant while growing closer or crossing when that segment is extended the other direction. 

Traditionally in axiomatic geometry it is said that parallel lines are equidistant and that lines that are equidistant are parallel. The concept of two lines $\ell_1$ and $\ell_2$ being equidistant is formalized as follows: line $\ell_1$ is equidistant to line $\ell_2$ if for any points $a$ and $b$ on $\ell_1$ when one drops perpendiculars from $a$ and $b$ down to $\ell_2$ intersecting $\ell_2$ at $x$ and $y$ respectively $ax \cong by$.




Some modification must be made to translate this concept into our system. Importantly we can not make a statement such as 'for all points'. We can capture this idea in our system as follows. 

\begin{theorem} \label{parallelequidistance}
	Consider points $a$, $b$, $c$, and $d$ such that $T(b,c,d)$ and $T(a,c,d)$. Let $x$ be the point constructed when dropping the perpendicular from $a$ to $cd$ and let $y$ be the point constructed when dropping the perpendicular from $b$ to $cd$. Then $ab \parallel cd$ if and only if $SS(a,b,cd)$ and $ax \cong by$.
\end{theorem}

\begin{proof}
	$(\Longrightarrow)$ Suppose that $ab \parallel cd$. By Theorem \ref{parallelsegthm2} we have $ab \parallel xy$. First we will show that $SS(a,b,cd)$. Let $doub(a,mid(b,y))=t$. (See Figure \ref{equidistantfoward}.) Then $abty$ is a non-flat parallelogram and thus $ab \parallel xy$ and $ab \parallel ty$. From Theorem \ref{parallelimplieslinear} we can infer $\widetilde{L}(y,x,t)$. Since $abty$ is a non-flat parallelogram we can infer $SS(a,b,yt)$ from Observations \ref{ObservationQuad1} and \ref{ObservationQuad2}. Therefore $SS(a,b,yx)$ and $SS(a,b,cd)$ by Theorem \ref{sideslinethm2}.  
	
	Now let $doub(x,mid(a,y))=z$. Note that $axyz$ is a non-flat parallelogram. Thus $az \parallel xy$ and therefore $az \parallel cd$. Since $ab \parallel cd$ and $az \parallel cd$ by using Theorem \ref{parallelimplieslinear} we can infer $\widetilde{L}(a,z,b)$.  
	
	
	
	
	
	Suppose $b \not= z$. Since $az \parallel cd$ and $ab \parallel cd$ we know $a \not= z$ and $a \not= b$ and thus we have $L(a,z,b)$ 
	
	Using the definition of right angles and the converse of the Alternate Interior Angle Theorem we can show that since angle $yxa$ is a right angle $xyz$ is also a right angle. From above we have $SS(a,b,xy)$. Since $axzy$ is a non-flat parallelogram we have $SS(a,z,xy)$. Therefore by Theorem \ref{sideslinethm1} we have $SS(z,b,xy)$. Since we additionally know that $xyb$ and $xyz$ are both right angles and therefore congruent, we have $SD(y,z,b)$ by Theorem \ref{conganglessamesidesameray}. Since we are supposing that $b \not= z$, we have $L(y,z,b)$.  Additionally since  $L(a,z,b)$ we can infer $L(a,z,y)$ which is a contradiction. Thus $b =z$. Therefore $axyb$ is a strict parallelogram and $ax \cong by$ by Theorem \ref{convexquadthm}. 
	
	\begin{figure}[h!]
		\begin{picture}(216,110)
		\put(25,0){\includegraphics[scale=.9]{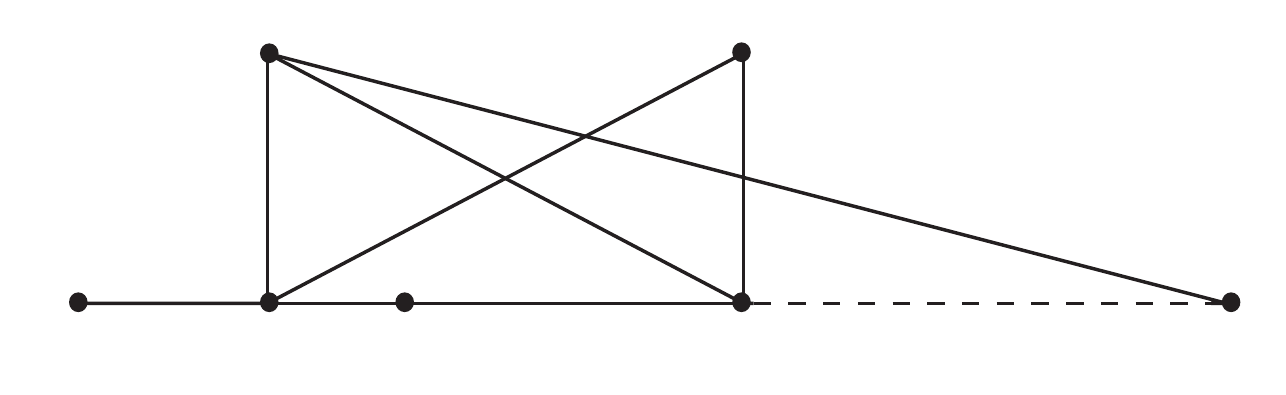}}

		\put(90,13){$x$}
		\put(130,13){$c$}	
		\put(40,11){$d$}	
		\put(217,13){$y$}	
		\put(347,13){$t$}
		\put(221,88){$b=z$}		
		\put(90,92){$a$}	
		
		\end{picture}
		\caption{}
		\label{equidistantfoward}
	\end{figure}
	
	$(\Longleftarrow)$ Assume that $SS(a,b,cd)$ and $ax \cong by$. We want to show that $ab \parallel cd$. Let $doub(x, mid(a,y))=s$. (See Figure \ref{equidistantbackwards}.) Since $axys$ is a strict parallelogram we know $SS(a,s,xy)$ and thus $SS(a,s,cd)$. Also we know that $ax \cong sy$ and thus $sy \cong by$. Furthermore, since angle $yxa$ is right angle, we have that $xys$ is a right angle. Thus by the uniqueness of angle construction (Axiom \ref{atsunique} ) we have $b=s$. Thus $axyb$ is a strict parallelogram and in particular $ab \parallel xy$. Therefore $ab \parallel cd$.
	
	 	\begin{figure}[h!]
	 	\begin{picture}(216,110)
	 	\put(85,0){\includegraphics[scale=.9]{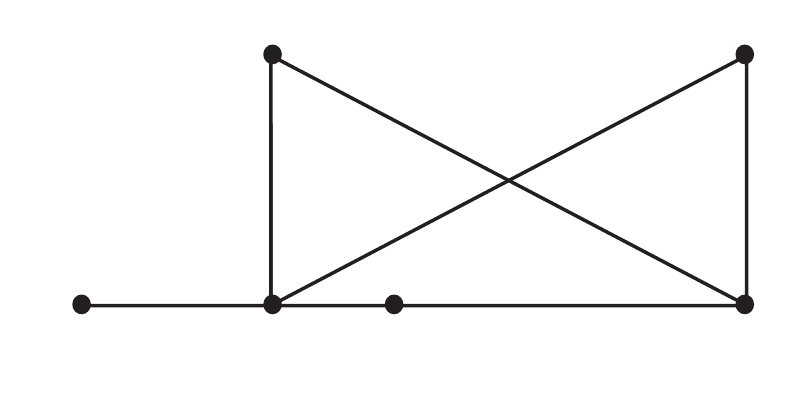}}

	 	\put(150,13){$x$}
	 	\put(187,13){$c$}	
	 	\put(100,11){$d$}	
	 	\put(277,13){$y$}	

	 	\put(282,88){$b=s$}		
	 	\put(150,92){$a$}	
	 	
	 	\end{picture}
	 	\caption{}
	 	\label{equidistantbackwards}
	 \end{figure}

\end{proof}



Let $a$, $b$, $c$, $d$, $x$, and $y$ be defined as they were in the previous theorem with $a$ and $b$ on the same side of $cd$. If $ab \nparallel cd$, then by that theorem we can infer that $ax \not \cong by$. Without loss of generality we can say that $ax < by$. We wish to show that if we extend the segment $ab$ in the direction of $a$ to some point $e$ and $z$ is the point constructed by dropping a perpendicular from $e$ to $cd$, then either $a$ and $e$ are on opposite sides of $cd$, $e$ is co-linear with $c$ and $d$, or $ez < ax$. Also we wish to show that if $ab$ is extended in the direction of $b$ to some point $g$ and $w$ is the point constructed by dropping a perpendicular from $g$ to $cd$, then $gw > by$. 

The following theorem proves these results.

\begin{theorem}
	Let $a$, $b$, $c$, $d$, $x$, and $y$ be defined as they were in Theorem \ref{parallelequidistance} where $SS(a,b, cd)$ and $ax < bx$. Let $B(e,a,b)$ and $B(a,b,g)$. Then either $OS(a,e,cd)$, $\widetilde{L}(c,d,e)$, or $SS(a,e,cd)$. If $SS(a,e,cd)$ and $z$ is the point constructed from dropping a perpendicular from $e$ to $cd$, then $ez <ax$. We can also conclude that $SS(b,g,cd)$ with $gw > by$ where $w$ is the point constructed from dropping a perpendicular from $w$ to $cd$.

\end{theorem}

	\begin{figure}[h!]
	\begin{picture}(216,110)
	\put(40,0){\includegraphics[scale=.9]{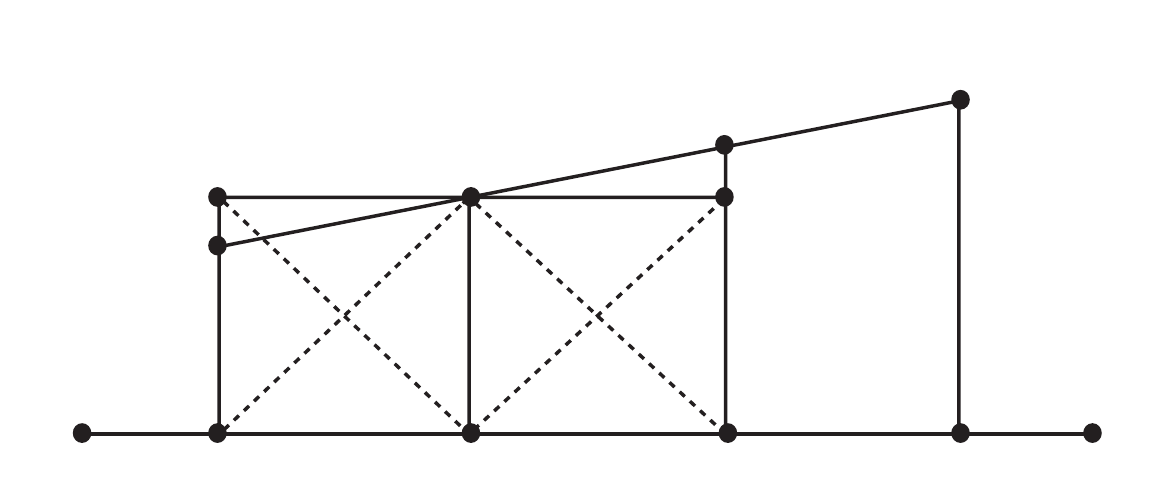}}
	
			\put(55,8){$c$}	
		\put(95,8){$z$}	
	\put(160,8){$x$}
	\put(227,8){$y$}	
	\put(287,8){$w$}		
	\put(325,8){$d$}	

	\put(158,85){$a$}	
		\put(95,85){$t$}	
	\put(232,78){$s$}	
	
	\put(230,96){$b$}	
	
	\put(294,105){$g$}	
	
	\put(87,63){$e$}

	\end{picture}
	\caption{}
	\label{nonparallelsegmentextensions}
\end{figure}

\begin{proof}
	To prove the first conclusion note that the following three relations are mutually exclusive:  $OS(a,e,cd)$, $\widetilde{L}(c,d,e)$, and $SS(a,e,cd)$. We will suppose that $a$ and $e$ are not on opposite sides of $cd$ and $c$, $d$, and $e$ are not (non-strict) colinear. We can then infer that $SS(a,e,cd)$. We now desire to show that $ez < ax$. Let $doub(x, mid(a,z))=t$. (See Figure \ref{nonparallelsegmentextensions}.) By methods used in the proof of Theorem \ref{parallelequidistance} we can infer that $SD(z,e,t)$. If $e=t$ then $ae \parallel xz$ and thus $ab \parallel cd$. This is a contradiction. Thus $e \not = t$. So either $B(z,e,t)$ or $B(z,t,e)$. Suppose $B(z,t,e)$. By axiom \ref{intersectOO} we have $OS(z,e,at)$. Let $doub(x, mid(a,y))=s$. We can show that $SD(y,s,b)$ and since $ax < by$ and $ax \cong  sy$ we know that $B(y,s,b)$. Thus $OS(y,b,sa)$ and therefore $OS(y,b,at)$. By theorem \ref{parallelogramtrans} we know $tzys$ is a non-flat parallelogram and thus $SS(y,z,at)$. We then can infer $SS(e,b,at)$. This contradicts axiom \ref{intersectOO}. Thus we have $B(z,e,t)$ and since $tz \cong ax$ we can infer that $ez < ax$. Using similar methods one can show that $gw > by$. 
\end{proof}

The following theorem shows us that as the sides of an angle are extended the distance between the sides increases. 

\begin{theorem} \label{angleopensthm}
	Given $T(b,a,d)$, $B(a,b,c)$, $bd \parallel ce$, and $L(a,d,e)$, $bd < ce$. 
\end{theorem}

\begin{figure}[h!]
	\begin{picture}(216,110)
	\put(90,0){\includegraphics[scale=.9]{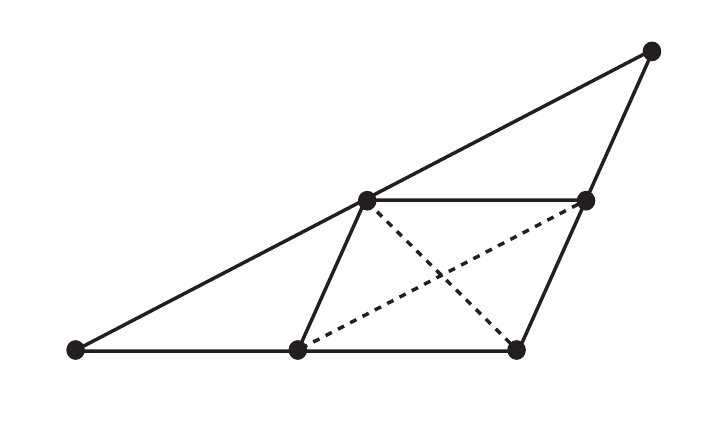}}
	
		\put(165,9){$d$}
		\put(223,9){$e$}	
		\put(100,13){$a$}	
	
		\put(177,62){$b$}	
			\put(264,96){$c$}

	\put(248,58){$x$}	

	\end{picture}
	\caption{}
	\label{angleopensfigure}
\end{figure}

\begin{proof}
	Given the assumptions, one can show $T(b,d,e)$. Thus we can construct a non-flat parallelogram $edbx$ where $doub(d,mid(b,e))=x$. (See Figure \ref{angleopensfigure}.) Note that $\widetilde{L}(e,x,c)$ since $db \parallel ec$. If $x=c$ then $bc \parallel de$ and thus $ac \parallel de$. This is a contradiction so $x \not = c$. Since $edbx$ is a non-flat parallelogram $x \not= e$. Thus $L(e,x,c)$. One can show $T(a,x,c)$. Thus by Axiom \ref{intersectOO} we have $OS(a,c,bx)$. Since $bx \parallel de$ we have $bx \parallel ae$. Therefore by Theorem \ref{parallelequidistance} we have $SS(a,e,bx)$. So $OS(e,c,bx)$. From this we can conclude that $B(e,x,c)$ and thus $ex < ec$. We know that $db \cong ex$. Thus we can conclude that $db < ec$. 
\end{proof}

\subsection{Segment Division}

In his work \textit{The Methods}, Descartes describes a method of defining addition, subtraction, multiplication, and division of magnitudes of line segments via geometry procedures. Although we are not considering interpretations of our constructions as total functions or total operations, the procedures for addition and subtraction can be carried out in our system by use of the extension and laying off constructions. The necessary property of Euclidean Geometry that proves that Descartes' procedure for multiplication is well defined relies of the Parallel Postulate. Our system is does not have the necessary constructions needed to produce the required points (In Euclidean Geometry these points are understood as the intersections of non-parallel lines.) Interestingly Decartes' procedure for division is accomplishable in our system. This leads to the potential for interesting philosophical discussion as to the geometrization of arithmetic and to potential argument that multiplication does not satisfy feasibility properties that the other three operations do. 

Here we will use Descartes' procedure for division to define a uniform construction for trisecting a line segment. This trisection construction alludes to a method for constructing points on a given line segment that divide the segment up into $n$ congruent parts. Our language does not encoded the natural numbers so a separate proof would be needed for each $n$. This construction is closely related to Suppes's trapezoid construction discussed in \cite{SuppesAxioms}. 

\begin{theorem}
	Given a line segment $ab$ there is a uniform construction for constructing points $c$ and $d$ such that $ac \cong cd \cong db$. 
\end{theorem}

\begin{figure}[h!]
	\begin{picture}(216,190)
	\put(60,0){\includegraphics[scale=.9]{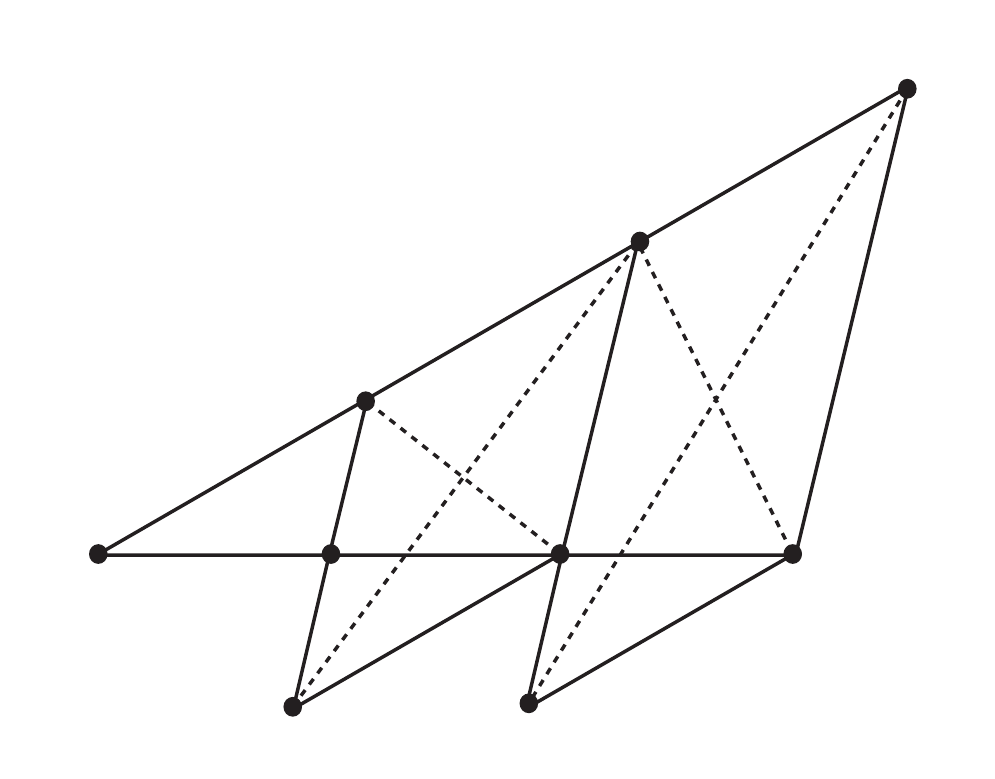}}
	
	\put(135,9){$y$}
	\put(193,9){$x$}	
	\put(75,50){$a$}

		\put(150,61){$c$}	
	\put(210,61){$d$}	
		\put(270,55){$b$}	
	
	\put(150,105){$\gamma$}	
	\put(223,148){$\gamma'$}	
	\put(299,180){$\gamma''$}

	\end{picture}
	\caption{}
	\label{trisectfigure}
\end{figure}

\begin{proof}
	Consider two distinct points $a$ and $b$. By Axiom \ref{3noncolinearpoints} we know that one of the points $\alpha$, $\beta$, or $\gamma$ must be distinct and non-colinear to $a$ and $b$. Without loss of generality let $T(a,b,\gamma)$. Let $doub(a,\gamma)=\gamma'$ and let $doub(\gamma,\gamma')=\gamma''$. (See Figure \ref{trisectfigure}.) Also let $doub(\gamma'',mid(\gamma',b))=x$. Since $b \gamma'' \gamma' x$ is a non-flat parallelogram we have $OS(\gamma'', x, \gamma'b)$. Additionally we can show $T(a,\gamma'',x)$ and $T(\gamma',x,b)$. 
	By lemma \ref{Paschlemma} either $Int(x,b\gamma'\gamma'')$ or $Int(x,a\gamma'b)$. If $Int(x,b\gamma'\gamma'')$, then we have $SS(\gamma'', x, \gamma'b)$. This is a contradiction. Thus $Int(x,a\gamma'b)$. Therefore $B(a,cb(x,a\gamma'b),b)$ and $SD(\gamma', d,x)$ (Theorem \ref{crossbartheorem}). Let $cb(x,a\gamma',b)=d$. We now wish to show that $B(\gamma',d,x)$. We will accomplish this by showing that $\gamma'd <\gamma'x$. Note that $\gamma'x \cong \gamma''b$. Thus we want that $\gamma'd < \gamma''b$. This can be shown by an application of of Theorem \ref{angleopensthm}. 
	
	Next we will construct the point $doub(\gamma', mid(\gamma,d))=y$. Using similar methods as in the previous paragraph we can construct a point $c$ such $B(\gamma,c,y)$ and $B(a,c,d)$. By the converse of the Alternate Interior Angle Theorem and Theorem \ref{convexquadthm} one can conclude $a\gamma c \cong dyc \cong \gamma \gamma' d \cong bxd$. Give how $\gamma'$ and $\gamma$ were constructed by Theorem \ref{convexquadthm} we can show $a\gamma \cong \gamma' \gamma'' \cong dy \cong \gamma' \gamma'' \cong bx$. Using the observation that vertical angles are congruent and the converse of the Alternate Interior Angle Theorem again we can show $ac\gamma \cong ycd \cong \gamma' dc \cong xdb$. Lastly, using the Angle-Angle-Side triangle congruence on triangles $ac \gamma$, $dcy$, and $bdx$ we can conclude that $ac \cong cd \cong db$. 
\end{proof}





\section{Compass Constructions and Circle-Circle Continuity} \label{CompassConstructions}


Circles and constructions involving them were central to Euclid's \textit{Methods}. One of the earliest theorems proved in Euclid's text was the construction of an equilateral triangle give a base. This construction was heavily reliant on facts about the continuity of circles which Euclid did not correctly assume. In section 11 of \cite{HartshorneBook}, Hartshorne discusses circle-circle continuity and line-circle continuity. These continuity properties determine when, how often, and under what conditions circles intersect circles and lines intersect circles. 

One version of stating circle-circle continuity is as follows: If $\Lambda$ and $\Delta$ are two circles and there are points $a$ and $b$ such that $a$ and $b$ are on $\Delta$, $a$ is inside of $\Lambda$, and $b$ is outside of $\Lambda$ then there exist two points which are on both $\Lambda$ and $\Delta$. Line-circle continuity can be stated as follows: If $\Delta$ is a circle, $\ell$ is a line, and $a$ is a point on $\ell$ which is also inside of $\Delta$, then there exit two points which are on $\ell$ and $\Delta$. 

Assuming the axioms for a Hilbert plane, Hartshorne proves line-circle continuity from circle-circle continuity. He also states that the equivalence of both over an arbitrary Hilbert plane follows from a classification theorem of Pejas of Cartesian planes over fields. Hartshorne states that he is unaware of any direct proof. We will assume an analog to circle-circle continuity and modify Hartshore's proof to derive an analog to line-circle continuity. 

It is worth pointing out for the reader that given Hartshornes definition of segment inequality and it use in defining when a point is inside of a circle, Hartshorne's statement of circle-circle continuity does not allow for $a$ to be the center of $\Lambda$. Furthermore this issue causes problems in his proof of line-circle continuity. Our definition of segment inequality does away with the issues of stating circle-circle continuity and our proof of line-circle continuity is both modified for our system and corrects the error of Hartshorne's proof. 

In Tarski's system a different continuity property was assumed. Possibly because lines are not objects in Tarski's system a segment-circle continuity axiom was assumed. Segment-circle continuity states that if $a$ is a point inside of a circle and $b$ is a point outside of the circle, then there exists a point between $a$ and $b$ which is on the circle. In \cite{Beeson} Beeson states all three continuity axioms in the language of Tarski's system using existential quantifiers. He points out that line-circle and segment-circle continuity are equivalent in Tarski's system without invoking a parallel postulate. The proof of this fact is highly non-trivial. He then shows how circle-circle continuity is derivable from line-circle continuity. Again the proof of this observation is complex. Beeson does not point out that line-circle continuity is derivable from circle-circle continuity. The reason for this might lie in the fact that Beeson uses line-circle continuity to define the construction of perpendiculars and uses perpendiculars to define a construction for midpoints. The construction of midpoints and perpendiculars are central to Hartshorne's proof of deriving line-circle continuity from circle-circle continuity. 

In this section we will introduce an undefined construction for producing a point which can be interpreted as an intersection of two circles given the conditions similar to those stated for circle-circle continuity above but modified for the language of our system. We will then provide a uniform construction of a second point of intersection. We will then derive an analog of line-circle continuity by defining a uniform construction for producing two distinct points which can be interpreted as the intersection of the line and the circle. Lastly we will derive an analog of segment-circle continuity by defining a uniform construction. The details of these results will be laid out since the use of undefined constructions, the absence of existential quantifiers, and the focus on uniform constructions differs from \cite{Beeson} and \cite{HartshorneBook} to a high enough degree. 

We introduce an undefined construction $circlecircleintersect(c_1,a,b,c_2,d)$ which will be shortened to $cci(c_1,a,b,c_2,d)$. We want $cci(c_1,a,b,c_2, d)$ to have the interpretation of a point of intersection of two circles where the first circle has center $c_1$ and radius $c_1a \cong c_1b$, the second circle has center $c_2$ with radius $c_2d$ and where $c_2a < c_2d$ ($a$ is inside the second circle) and $c_2d<c_2b$ ($b$ is outside the second circle). [See Figure \ref{cciconstructionfig}.] As one will note in the physical world, there are two points that this undefined construction could be referring to. In the language of our system there is no distinct between the two points. We will show that given one such point we can construct the other in a uniform fashion. This strategy is needed if we wish to avoid the use of existential quantifiers and have all of our constructions be uniform with no case distinctions. 

\begin{figure}[h!]
	\begin{picture}(216,155)
	\put(93,0){\includegraphics[scale=1.7]{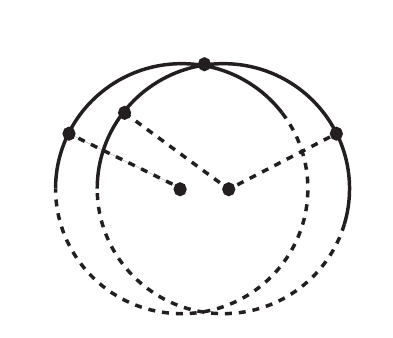}}

	\put(207,64){$c_1$}
	\put(174,64){$c_2$}

	\put(265,98){$b$}
	\put(115,97){$d$}
	\put(143,107){$a$}
	\put(193,143){$cci(c_1,a,b,c_2,d)$}

	\end{picture}
	\caption{}
	\label{cciconstructionfig}
\end{figure}

What follows is the only axiom pertaining to the $circlecircleintersect$ construction. 

\begin{axiom} \label{circircont} $c_1a \cong c_1b \wedge c_2a < c_2d \wedge c_2d < c_2b \rightarrow c_1cci(c_1,a,b,c_2,d) \cong c_1a \wedge c_2cci(c_1,a,b,c_2,d) \cong c_2d$
\end{axiom}

\begin{theorem} \label{SecondPointCircleIntersection}
	Given $c_1a \cong c_1b$, $c_2a < c_2d$, and $c_2d < c_2b$, there is a uniform construction of a point $x$ such that $c_1x \cong c_1a$, $c_2x \cong c_2d$, and $x \not= cci(c_1,a,b,c_2,d)$. 
\end{theorem}

\begin{figure}[h!]
	\begin{picture}(216,140)
	\put(76,3){\includegraphics[scale=1.7]{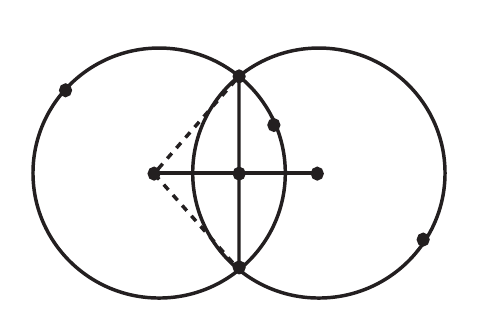}}
	
	\put(200,125){$y$}
	\put(215,100){$a$}

	\put(197,70){$p$}

	\put(145,69){$c_1$}
	\put(235,68){$c_2$}
	\put(290,40){$d$}

	\put(97,120){$b$}
	
	\put(189,21){$x$}
	\end{picture}
	\caption{}
	\label{circirintsecondpointfig}
\end{figure}

\begin{proof}
	First note that the assumptions imply $c_1 \not= c_2$. Suppose $c_1 = c_2$. Since $c_2a < c_2d$ and $c_1a = c_1b$, we have $c_2b < c_2d$. This is a contradiction. 
	
	Let $y = cci(c_1, a, b, c_2, d)$. It can also be shown $T(c_1, y, c_2)$. Suppose $\lnot T(c_1, y, c_2) \equiv \widetilde{L}(c_1, y, c_2)$, then $L(c_1, y, c_2)$. There are 3 cases to consider. 
	
	Case 1: Let $B(c_1, y, c_2)$. Note $a \not= y$ and $b \not= y$ by the uniqueness of extension construction. Let $d$ be the point uniformly constructed such that $c_1yd$ is right and $SS(a,d,c_1y)$. We claim $SS(c_1, a, yd)$. If $\widetilde{L}(a,y,d)$, then by Lemma \ref{LengthSidesTriangle} we have $c_1a > c_2y$. Note $c_1ay < c_1ya$ by the Exterior Angle Theorem (Theorem \ref{ExteriorAngleThm}). This is a contradiction. If $OS(c_1, a, yd)$, then we can use the crossbow construction to construct a point $z$ such that $B(c_1, z, a)$ and $\widetilde{L}(z,y,d)$. By similar reasoning to the previous step we have $c_1z > c_1y$. Since $c_1a > c_1z$ we have $c_1a>c_1y$. This is a contradiction. Therefore $SS(c_1, a, yd)$. Similarly we can prove $SS(c_2, a, yd)$ since $c_2 < c_2d = c_2y$. Since $B(c_1,y,c_2)$ and $T(c_1, y, d)$ we have $OS(c_1, c_2, yd)$. This is a contradiction.
	
	 \begin{figure}[h!]
	 	\begin{picture}(216,50)
	 	\put(55,0){\includegraphics[scale=.9]{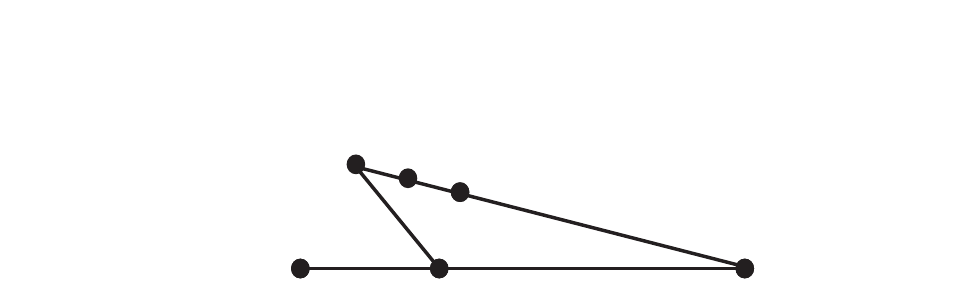}}

	 	\put(120,3){$y$}

	 	\put(177,32){$c'$}	
	 	
	 	\put(255,3){$c_2$}
	 	
	 	\put(170,0){$c_1$}
	 	
	 	\put(140,40){$b$}
	 	
	 	\put(160,35){$d'$}
	 		 	
	 	\end{picture}
	 	\caption{}
	 	\label{TriInequalityInCircProof}
	 \end{figure}
	
	Case 2: Let $B(y, c_1, c_2)$. (Reference Figure \ref{TriInequalityInCircProof}.) Note $c_1y \cong c_1b$ and $c_2y \cong c_2d$. Recall $c_2d < c_2b$. Thus we can construct a point $d'$ such that $SD(c_2, d', b)$ and $c_2d' \cong c_2d$. Note $B(c_2, d', b)$. Similarly we can construct a point $c_1'$ such that $B(c_2, c_1', d)$ and $c_2c_2 \cong c_2c_1'$. Notice $c_2y \cong c_2ext(c_2c_1, c_1b) \cong c_2d'$ where $B(c_2, d', b)$. Thus $c_2ext(c_2,c_1, c_1b) < c_2b$. This contradicts the Triangle Inequality Theorem (Theorem \ref{TriangleInequalityThm}). 
	
	Therefore we can conclude $T(c_1,y,c_2)$. 
	
	We may now construct the point $x$. Let $p$ be the point constructed by dropping a perpendicular from $y$ to $c_1c_2$ and let $x = doub(y,p)$. [See Figure \ref{circirintsecondpointfig}] We know that $ypc_1$ is a right angle. Therefore $ypc_1 \cong xpc_1$. We also know that $yp \cong xp$. Therefore by Side-Angle-Side triangle congruence theorem we have $c_1y \cong c_1x$. Since $c_1y \cong c_1a$, we have $c_1x \cong c_1a$. By similar methods we can show $c_2x \cong c_2d$. Lastly since $x = doub(y,p)$ we have $B(y,p,x)$. Thus $x \not= y$. 
\end{proof}

We now state a prove our version of line-circle continuity. 

\begin{theorem}
	Given $ca < cd$, $c \not=d$, and $a \not= b$, there is a uniform construction of two distinct points $x$ and $y$ such that $\widetilde{L}(a,b,x)$, $\widetilde{L}(a,b,y)$, and $cd \cong cy \cong cx$. 
\end{theorem}

\begin{figure}[h!]
	\begin{picture}(216,150)
	\put(60,0){\includegraphics[scale=1.7]{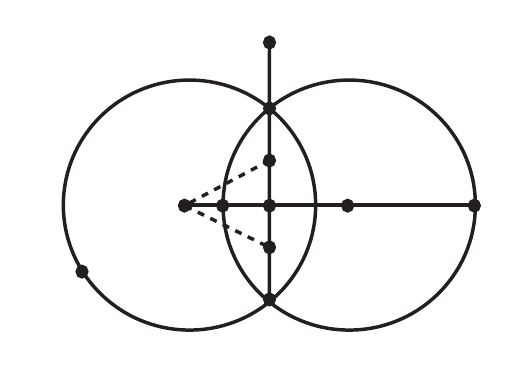}}

	\put(200,122){$x$}
	\put(197,97){$a$}
	\put(197,53){$a'$}
	\put(197,70){$p$}
	\put(172,71){$g$}
	\put(145,69){$c$}
	\put(235,68){$c'$}
	\put(297,71){$e$}
	\put(87,40){$d$}
	\put(197,155){$b$}

	\put(189,21){$y$}
	\end{picture}
	\caption{}
	\label{line-circlecontinuityfig}
\end{figure}

\begin{proof}
	We begin by defining points satisfying the conditions of Axiom \ref{circircont}. Let $p$ be the produced by dropping a perpendicular from $c$ to $ab$. Let $c' = doub(c,p)$. Let $e= ext(cc', cd)$ and $g = ext(ec',cd)$. See Figure \ref{line-circlecontinuityfig}. 
	
	We will now show that $cg <cd'$ and $ce< cd$. Thus $g$ will be `inside' and $e$ will be `outside of the `circle' with center $c$ and radius $cd$. Given properties of the extension construction we know  $B(c, c', e)$. Thus $c$ and $e$ are on opposite sides of $c'$. Given how $g$ was constructed we know $B(g,c',e)$. Thus $e$ and $g$ are on opposite sides of $c'$. Therefore by Theorem \ref{lineseparationthm} we have $SD(c, g, c')$. Let $a'=doub(a,p)$. Note $T(a,a',c)$ and $B(a,p,a')$. Thus by the Theorem \ref{triangelsegmentinequality} we have $cp < ca$ or $cp < ca'$. By Side-Angle-Side triangle congruence we have $ca \cong ca'$. Thus $cp < ca$. Since $ca < cd \cong c$ we have $cp < cd$. Since $c'p \cong cp$ we have $c'p < cd \cong c'g'$. Thus $B(c',p,g')$. Since $B(c',p,c)$ we have $SD(p, g',c)$. If $g'=c$, then $cg' < cd$ since $c \not= d$. Thus either $B(c,g',p)$ or $B(g',c,p)$. If $B(c,g',p)$, then $cg'<cp<cd$. If $B(g',c,p)$ we can infer $B(g',c,c')$ by Theorem \ref{lineseparationthm} since we have $SD(c,p,c')$. So $cg'<c'g'\cong cd$. Since $B(c,c',e)$ we have $ce>c'e \cong cd$. 
	
	Let $cci(c',e,g,c,d) = x$. By Axiom \ref{circircont} we know $cx \cong cd \cong c'x$. We wish to show $T(c,c',x)$. If $\lnot T(c,c',x)$, then $L(c,c',x)$ since all three points are distinct. Since $c'x \cong cd$, by the uniqueness of the extension construction either $x=e$ or $x=g$. In either case it can be shown that $cx \not= cd$. This is a contradiction. Thus $T(c,c',x)$. By Side-Angle-Side triangle congruence we have $xpc \cong xpc'$. Thus $xpc$ is a right angle. Therefore one can show $L(a,p,x)$ (since $a \not= p$) whether $a$ and $x$ are on the same or opposite side(s) of $cc'$. Thus $\widetilde{L}(a,b,x)$. 
	
	Let $y = doub(x,p)$. By methods used in the proof of the previous theorem we can show $x \not= y$, $\widetilde{L}(a,b,y)$, and $cy \cong cd$. 
	
\end{proof}
	
	We now prove our version of segment-circle continuity. (It should be pointed out that Tarski's system uses a non-strict between relation to define segment-circle continuity. Thus our version differs slightly.)
	
\begin{theorem}
	If $ca<cd$ and $cd<cb$, then there is a uniform construction of a point $z$ such that $B(a,z,b)$ and $cz \cong cd$. 
\end{theorem}

\begin{proof}
	By the proof of the previous theorem we know that there is a uniform construction of a point $x$ such that $cx \cong cd$. Let $z = lf(pa,px)$ where $p$ is defined as it was in the previous proof. (It may be that $z=x$, but in order to avoid case distinctions we will construct $z$ as stated.) Since $ca<cz$ we have $a \not= z$. Thus $L(p,a,z)$. If $B(p,z,a)$, we would have $T(p,a,c)$, $cz<cp$, and $cz<ca$ which would contradict Theorem \ref{triangelsegmentinequality}. Thus $B(p,a,z)$. We can then use Theorem \ref{lineseparationthm} to show $SD(a,z,b)$. If $B(a,b,z)$, we will have $cb<ca$ and $cb<cz$ which is a contradiction. Thus $B(a,z,b)$. 
\end{proof}

As we have seen, it is possible for two distinct points, $x$ and $y$, to both be equidistant to two distinct points, $c_1$ and $c_2$. The only case that has been considered up to this point is when both $x$ and $y$ are not colinear with $c_1$ and $c_2$. The following theorem states that if $x$ is equidistant and colinear with $c_1$ and $c_2$ and $y$ is some point equidistant to $c_1$ and $c_2$, then $x$ and $y$ are the same point. Note that $x= mid(c_1, c_2)$ would satisfy the conditions for $x$ in the following theorem.

\begin{theorem}
	$L(c_1, c_2, x) \wedge c_1x \cong c_1y \wedge c_2x \cong c_2y \rightarrow x=y$
\end{theorem}

\begin{proof}
	The proof is given by Hartshorne in the first half of the proof of Proposition 11.4 in \cite{HartshorneBook}. 
\end{proof}

There is a well known result in Euclidean Geometry that any three distinct non colinear points, $a$, $b$, and $c$, lie on a unique circle with a unique center at the intersection of the perpendicular bisectors of $ab$ and $bd$. In order to construct this center one must invoke a parallel postulate to find the intersection of the two perpendicular bisectors. Given that we classify these type of constructions as non feasible, we are not able to define such a construction. What we are able to prove is that three distinct points cannot be equidistant to two different points. First we state our analog to the result that if $c$ is equidistant to distinct points $a$ and $b$, then $c$ lies on the perpendicular bisector of $ab$. 

\begin{lemma}Let $x=mid(a,b)$ and $a \not= b$. If $ca \cong cb$, then $c = x \lor axc$ is right where  
\end{lemma}

\begin{proof}
	There are two cases to consider. 
	
	Case 1) If $\widetilde{L}(a,c,b)$, then $L(a,c,b)$ given that $a \not= b$. Thus $c = x$ by the properties of the midpoint construction and the uniqueness of the extension construction. 
	
	Case 2) If $T(a,c,b)$, then let y be the point constructed by dropping a perpendicular from $c$ to $ab$. First we will show $B(a,y,b)$. Suppose $\lnot B(a,y,b)$. Then either $B(y,a,b)$ or $B(a,b,y)$. Without loss of generality let $B(y,a,b)$. By Lemma \ref{LengthSidesTriangle} $ca > cy$ and by Theorem \ref{triangelsegmentinequality} $cy > ca$ or $cb > ca$. Therefore, $cb > ca$. This is a contradiction. Thus $B(a,y,b)$. Since $ca \cong cb$ we know $cab \cong cba$. We can then use Angle-Angle-Side triangle congruence (proof found in Euclid's Proposition 26 in Book I) to show $ya \cong yb$. We can then infer $x=y$. 
\end{proof}


	
	



\begin{theorem}
	Let $a \not= b$, $a \not= d$, and $b \not= d$. Then
	
	$c_1a \cong c_1b \cong c_1d \wedge c_2a \cong c_2b \cong c_2d \rightarrow c_1 = c_2$. 
\end{theorem}

\begin{proof}
	Suppose $c_1 \not= c_2$. Let $x = mid(a,b)$ and let $y=mid(b,d)$. By the previous lemma $c_1xa$, $c_1yd$, $c_2xa$, and $c_2yd$ are all right. Also note that $x \not= y$. There are 3 cases to consider. 
	
	Case 1) Let $c_1$ and $c_2$ both not equal $x$ or $y$. By the uniqueness of angle constructions we have $L(x, c_1, c_2)$ and $L(y, c_1, c_2)$. Thus $L(x,y,c_1)$ Therefore, $yxa$ and $xyd$ are both right. Thus by the Alternate Interior Angle Theorem, $ab \parallel bd$. This is a contradiction. 
	
	Case 2) Let $c_1$ equal $x$ or $y$. Without loss of generality let $c_1 =x$. Note $L(a,c_1,b)$ . Therefore $yc_1a$ and $c_1yd$ are both right. Thus by the Alternate Interior Angle Theorem, $ab \parallel bd$. This is a contradiction. 
	
	Case 3) Let $c_1$ equal $x$ or $y$. This case is analogous to the previous case. 
	
	Therefore $c_1 = c_2$.  

\end{proof}


\section{Coplanar Relation}\label{SoildGeo}

In this section we aim to develop an analog to part of Book XI of Euclid's Elements in which solid geometry is discussed. Similar to how we avoid introducing lines as objects by introducing a colinear relation, we will avoid discussing planes directly by introducing a coplanar relation. We introduce the relation $PL(a, b, c, d)$, with a condensed written form of $PL_{abc}(d)$, to have the interpretation that $d$ is coplanar with $a$, $b$, and $c$ where $a$, $b$, and $c$ are not colinear. Additionally we have the need to introduce a new construction called the orthogonal construction which is denoted $o(a,b,c)$. The point $o(a,b,c,)$ (where $a \not= b$ and $c \not= b$) is understood to be a point such that the angles $abo(a,b,c)$ and $cbo(a,b,c)$ are right angles.  To display the feasibility of this construction we have designed a tool we call the orthogonator which will produce such a point given three points $a$, $b$ and $c$ (where $a \not= b$ and $c \not= b$). In figure \ref{Orthogonator}, we have displayed a diagram of such a tool. The instrument has three segments all with a common endpoint. The segments 1 and 3 form a (fixed) right angle. Segment 2 also forms a right angle with segment 3 while being able to swing a full rotation is such a way that segments 1 and 2 act much like the hands of a clock (where segment 1 is fixed). The length of the three segments is not stipulated. Label the endpoints of the first segment $y$ and $x$. Label the endpoints of the second segment $y$ and $z$. Label the endpoints of the third segment $y$ and $o$. The tool is put into use by placing $y$ at $b$, placing $x$ at $lf(ba, yx)$, and placing $z$ at $lf(bc, yz)$. The resulting position of endpoint $o$ is the desired point. 

\begin{figure}[h!]
	\begin{picture}(216,145)
	\put(-40,0){\includegraphics[scale=1.8]{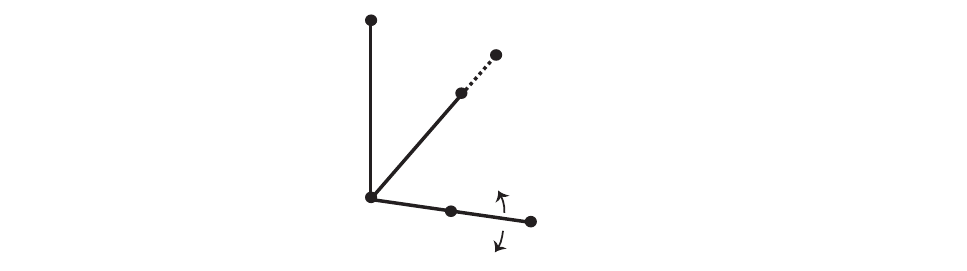}}
	
	\put(120,33){$y=b$}
	\put(140,127){$o$}
	\put(207,88){$a$}
	\put(225,110){$x=lf(ba,yx)$}
	\put(195,36){$c$}
	\put(243,24){$z=lf(bc,yz)$}
	
	\end{picture}
	\caption{}
	\label{Orthogonator}
\end{figure}

We have five axioms pertaining to coplanar relation and the orthogonal construction. 
\begin{axiom}\label{coplanarnotcolinear}
	$PL_{abc}(d) \rightarrow T(a,b,c)$
\end{axiom}

\begin{axiom}\label{bcoplanar}
	$T(a,b,c) \rightarrow PL_{abc}(b)$
\end{axiom}

\begin{axiom}\label{coplanartrans}
	$PL_{abc}(d) \wedge PL_{abc}(e) \wedge PL_{abc}(f) \wedge T(d,e,f) \wedge PL_{abc}(x)\rightarrow PL_{def}(x)$
\end{axiom}

\begin{axiom}\label{orthogonatorisright}
	$ a \not= b \wedge c \not= b \rightarrow abo(a,b,c)$ is right $\wedge$ $cbo(a,b,c)$ is right.
\end{axiom}

\begin{axiom}\label{PerptoOrthogisCoplanar}
	$T(a,b,c) \wedge dbo(a,b,c)$ is right $\rightarrow PL_{abc}(d)$
\end{axiom}
Axiom \ref{coplanarnotcolinear} states that for $d$ to be coplanar with $a$, $b$ and $c$, $a$, $b$ and $c$ must not be colinear.  Axiom \ref{bcoplanar} simply states that if $a$, $b$ and $c$ are non colinear, then $b$ is coplanar with $a$, $b$ and $c$. Axiom \ref{coplanartrans} states that if three non colinear points $d$, $e$, and $f$ are all coplanar with $a$, $b$ and $c$ and $x$ is a point which is coplanar with $a$, $b$, and $c$ then $x$ is coplanar with $d$, $e$ and $f$. Axiom \ref{orthogonatorisright} states that if $a \not= b$ and $c \not= b$ then the segment $b(o(a,b,c))$ is perpendicular to both $ab$ and $cb$. Lastly, Axiom \ref{PerptoOrthogisCoplanar} states that if $a$, $b$ and $c$ are non colinear and segment $bd$ is perpendicular to segment $b(o(a,b,c))$, then $d$ is coplanar with $a$, $b$, and $c$. 

Surprisingly there are minimal modifications that need to be made to previous axioms and theorems to fully incorporate the results from previous sections about planar geometry into the setting of solid geometry. Before we introduce those modifications, we will work out a set of results which follow without these modifications. 

Hilbert's fourth through eighth incidence axioms deal with planes. His fourth and fifth axioms state that three non colinear points determine a plane and that only one plane can contain these three points. Our use of a coplanar relation and Axiom \ref{coplanarnotcolinear} will account for analogous features. Hilbert's sixth incidence axiom states that if two points of a line are contained in a plane then the entire line is contained in the plane. We will prove an analogous result in Theorem \ref{ColinearCoplanar}. Hilbert's seventh incidence axiom states that if two (distinct) planes have a point in common, then they have another point in common. This axiom of Hilbert's is nonconstructive. In Theorem \ref{planesecondpointofintersection} we will provide a uniform construction for creating such a point. Lastly, Hilbert's eighth incidence axiom states that there exist at least four points not lying in the same plane. In Theorem \ref{4noncoplanarpoints} we will prove that for three non colinear points $a$, $b$ and $c$ (in particular $\alpha$, $\beta$ and $\gamma$), $o(a,b,c)$ is not coplanar with $a$, $b$ and $c$. 

Euclid did not clearly state all the assumptions he made about planes. Meanwhile Hilbert did not develop many results pertaining to solid geometry from his axioms. In this section we will develop numerous results analogous to propositions of Euclid's while also proving results analogous to Hilbert's axioms. 

\begin{theorem}\label{accoplanar}
	$ T(a,b,c) \rightarrow PL_{abc}(a) \wedge PL_{abc}(c)$
\end{theorem}

\begin{proof}
	An application of Axiom \ref{orthogonatorisright} followed by an application of Axiom \ref{PerptoOrthogisCoplanar}. 
\end{proof}

\begin{theorem} \label{COPlanarPermute}
	$PL_{abc}(x) \rightarrow PL_{bac}(x) \wedge PL_{cba}(x)$
\end{theorem}

\begin{proof}
	An application of Axiom \ref{coplanarnotcolinear} to obtain $T(a,b,c)$. Then apply Axiom \ref{bcoplanar} and Theorem \ref{accoplanar} followed by Axiom \ref{coplanartrans}.
\end{proof}

The previous theorem implies that we can also obtain $PL_{acb}(x)$, $PL_{bca}(x)$ and $PL_{cab}(x)$ given $PL_{abc}(x)$. Thus we obtain all permutations of $a$, $b$ and $c$ for the coplanar relation notation. 

\begin{theorem}\label{coplanartransbothways}
		$PL_{abc}(d) \wedge PL_{abc}(e) \wedge PL_{abc}(f) \wedge T(d,e,f) \rightarrow [PL_{abc}(x)\leftrightarrow PL_{def}(x)$]
\end{theorem}

\begin{proof}
	Assume $PL_{abc}(d) \wedge PL_{abc}(e) \wedge PL_{abc}(f) \wedge T(d,e,f)$.
	
	$(\Longrightarrow)$ Given $PL_{abc}(x)$, we obtain $PL_{def}(x)$ by an application of Axiom \ref{coplanartrans}.
	
	$(\Longleftarrow)$ Given the assumptions above we obtain $PL_{abc}(a)$, $PL_{abc}(b)$ and $PL_{abc}(c)$ by Axiom \ref{bcoplanar} and Theorem \ref{accoplanar}. From Axiom \ref{coplanartrans} we have  $PL_{def}(a)$, $PL_{def}(b)$ and $PL_{def}(c)$ By Axiom \ref{coplanarnotcolinear} we also have $T(a,b,c)$. Thus another application of Axiom \ref{coplanartrans} give us $PL_{def}(x) \rightarrow PL_{abc}(x)$. 
\end{proof}

The next Theorem is our analog to Hilbert's axiom that if two points of a line are contained in a plane then the entire line is contained in the plane. 

\begin{theorem}\label{ColinearCoplanar}
	$PL_{abc}(d) \wedge PL_{abc}(e) \wedge L(d,e,f) \rightarrow PL_{abc}(f)$
\end{theorem} 

\begin{proof}
	By Axiom \ref{coplanarnotcolinear} we have $T(a,b,c)$. Because of this we have $T(a,d,e)$, $T(b,d,e)$ or $T(c,d,e)$. Without loss of generality let $T(a,d,e)$. By Axiom \ref{orthogonatorisright}, angle $edo(a,d,e)$ is right. Since $L(d,e,f)$, note $f \not= d$, $fdo(a,d,e)$ is right. Thus by Axiom \ref{PerptoOrthogisCoplanar} we obtain $PL_{ade}(f)$. Lastly, since $PL_{abc}(a)$, $PL_{abc}(d)$, $PL_{abc}(e)$, $T(a,d,e)$ and $PL_{ade}(f)$ we have $PL_{abc}(f)$ by Axiom \ref{coplanartrans}.
\end{proof}


\subsection{Modifications to Previous Axioms and Theorems}

We will now discuss the modifications that are needed to lift our system from one about planar geometry to one about solid geometry. Surprisingly, there are only two modifications that need to be made before compass constructions are introduced. We modify Axiom \ref{SOtriangle} by stimulating that angles $abc$ and $def$ not only be non colinear triples (as is stated in the original version), but also that $d$, $e$, and $f$ are also coplanar with $a$, $b$, and $c$. In a similar vein, we stipulate that  $d$, $e$, and $f$ are also coplanar with $a$, $b$, and $c$ in the Definition \ref{OOdef}. 

\begin{modified axiom 6}
		$SO(abc,def) \rightarrow T(a,b,c) \wedge T(d,e,f) \wedge PL_{abc}(d) \wedge PL_{abc}(e) \wedge PL_{abc}(f)$
\end{modified axiom 6}

\begin{modified def 4}
	$OO(abc,def) \equiv \lnot SO(abc,def) \wedge T(a,b,c) \wedge T(d,e,f) \wedge PL_{abc}(d) \wedge PL_{abc}(e) \wedge PL_{abc}(f)$ 
\end{modified def 4}

From these two modification many statements in Sections 2 through 8 inherently imply that the points spoken about are coplanar with the same triple. For example if two points are said to be on the same side of a segment, then the definition of same side will imply that they are coplanar. If a point is said to be in the interior of an angle, then the definition of interior will imply that the point is coplanar with the angle. Coplanar relations can be shown in the statments of Alternate Interior Angle Theorem and its converse. One can also verify that the constructions discussed are producing points which are coplanar with the starting points. For example the $ats$ construction's definition implies this. Theorem \ref{ColinearCoplanar} is critical in justifying certain coplanar realtions. We can use Theorem \ref{ColinearCoplanar} to make the following observation. 

\begin{observation def 18}
$ab \parallel cd \rightarrow PL_{abc}(d)$
\end{observation def 18}

In Section 9 we have need to modify both the premise and conclusion of Axiom 31 to state that all points introduced and constructed are coplanar with the same triple. This fact will be used in the proof of Theroem \ref{planesecondpointofintersection}.

\begin{modified axiom 31}
	$c_1a \cong c_1b \wedge c_2a < c_2d \wedge c_2d < c_2b \wedge PL_{ac_1b}(d) \wedge PL_{ac_1b}(c_2) \rightarrow c_1cci(c_1,a,b,c_2,d) \cong c_1a \wedge c_2cci(c_1,a,b,c_2,d) \cong c_2d \wedge PL_{ac_1b}(cci(c_1,a,b,c_2,d))$
\end{modified axiom 31}

We can then add $PL_{ac_1b}(x)$ to the conclusion of Theorem 41. 

\begin{modified thm 41}
	Given $c_1a \cong c_1b$, $c_2a < c_2d$, and $c_2d < c_2b$, there is a uniform construction of a point $x$ such that $c_1x \cong c_1a$, $c_2x \cong c_2d$, and $x \not= cci(c_1,a,b,c_2,d)$. 
\end{modified thm 41}

\subsection{Results Reliant on Modifications:}

We will now be assuming the modifications stated above in our proofs. First we will state and prove a theorem that implies that there are four points for which one is not coplanar with the other three. 

\begin{theorem} \label{4noncoplanarpoints}
	$T(a,b,c) \rightarrow \lnot PL_{a, b, c}(o(a, b, c))$ [In particular $\lnot PL_{\alpha, \beta, \gamma}(o(\alpha, \beta, \gamma))$.]
\end{theorem}

\begin{proof}
Let $o = o(a,b,c)$. Suppose $\lnot PL_{abc}(o)$. Since $T(a,b,c)$ we have $a \not= b$ and $c \not= b$. Thus by Axiom \ref{orthogonatorisright} we have $abo$ is right and $cbo$ is right. Since $abo$ is right we have $T(o,b,a)$. Additionally, since $PL_{abc}(o)$ we can infer $PL_{abc}(o)$, $PL_{abc}(b)$, and $PL_{abc}(a)$. Likewise we can obtain $T(o,b,c)$, $PL_{abc}(o)$, $PL_{abc}(b)$, and $PL_{abc}(c)$. From Modified Axiom 6 and Modified Definition 4 we can conclude $SS(oba,obc)$ or $OS(oba,obc)$. Therefore $SS(a,c,ob)$ or $OS(a,c,ob)$. Since both $abo$ and $cbo$ are right, by the uniqueness of angle constructions we have $L(a,b,c)$. This is a contradiction. Thus $\lnot PL_{abc}(o)$. 
\end{proof}

The following theorem states that given a point which is coplanar with $a$, $b$ and $c$ and $d$, $e$, and $f$ there is a uniform construction for producing another point which is also coplanar with $a$, $b$ and $c$ and $d$, $e$, and $f$. This is our analog for Hilbert's seventh incidence axiom. 

\begin{theorem} \label{planesecondpointofintersection}
	Given $PL_{abc}(x)$ and $PL_{def}(x)$ there is a uniform construction of a point $p_{(abc,def,x)}$ such that $PL_{abc}(p)$, $PL_{def}(p)$ and $p \not= x$ where $p = p_{(abc,def,x)}$. 
\end{theorem}

\begin{proof}
	
	(As was stated above the constructions used will result in certain coplanar relations. We will clearly state this relations.) 
	
	We will being by uniformly constructing two points $u$ and $v$ such that $T(x,u,v)$, $PL_{abc}(u)$ and $PL_{abc}(v)$. Let $y$ be a point uniformly constructed such that $PL_{abc}(y)$, $aby$ is right and $by=ab$. Let $u=ext(ba,xy)$. Note $PL_{abc}(u)$ by Theorem \ref{ColinearCoplanar}. We will show that $x \not= u$. 
	
\begin{figure}[h!]
	\begin{picture}(216,100)
	\put(90,0){\includegraphics[scale=.9]{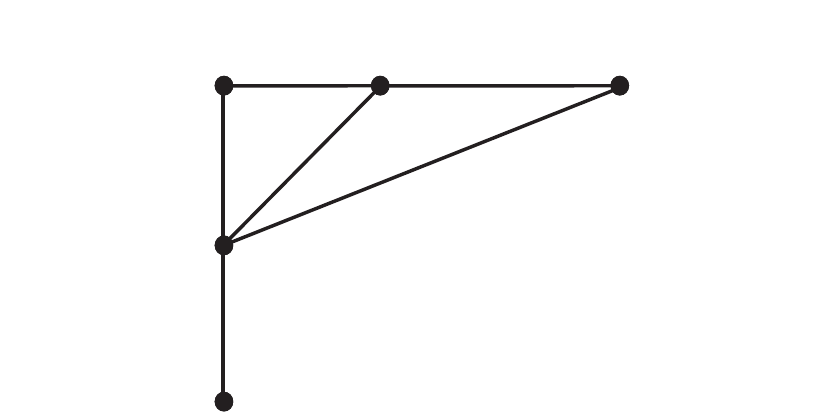}}

	\put(137,40){$y$}
	
	\put(137,1){$h$}
	
	\put(137, 80){$b$}
	
	\put(260,85){$x=u$}

	\put(190,90){$a$}

	\end{picture}
	\caption{}
	\label{x=uproof}
\end{figure}
		
	Suppose $x=u=ext(ba,xy)$.(Reference Figure \ref{x=uproof}.) Note $ax \cong yx$ and $B(b,a,x)$. Since triangle $aby$ is isosceles, $bay \cong bya$. Similarly $xay \cong xya$. Note $yab$ is supplemental to $yax$. Let $h=doub(b,t)$. Note $PL_{abc}(h)$ by Theorem \ref{ColinearCoplanar}. By Axiom \ref{intersectionoppositeorientation} we have $OS(b,h,da)$. We also have $ayb$ supplemental to $ayh$ with $ayb \cong yab$. Thus $ayx \cong ayh$. By the uniqueness of angle constructions we then have $L(t,h,x)$ and thus $L(b,y,x)$. From this we can infer that $L(b,a,y)$. This contradicts how $y$ was constructed. Therefore $x \not= u$ and with $PL_{abc}(u)$. 
	
	Let $v = cci(x, xu, u, ux)$. We can show $PL_{abc}(x)$, $PL_{abc}(u)$, $PL_{abc}(v)$ and $T(x,u,v)$. In a similar fashion we can uniformly construct points $s$ and $t$ such that $PL_{def}(x)$, $PL_{def}(s)$, $PL_{def}(t)$ and $T(x,s,t)$. 
	
	Let $p_{(abc,def,x)}=o(o(uxv), x,o(sxt))$. By Axiom \ref{orthogonatorisright},  $pxo(uxv)$ is right and $pxo(sxt)$ is right. Therefore by Axiom \ref{PerptoOrthogisCoplanar} we have $PL_{uxv}(p)$ and $PL_{sxt}(p)$. Via an application of Axiom \ref{coplanartrans} we can obtain $Pl_{abc}(p)$ and $PL_{def}(p)$. Lastly, since $pxo(uxv)$ is right we know $p \not= x$. 
	
\end{proof}	

The following theorem is our analog to the statement that distinct planes that intersect have a line as their intersection (Proposition 3 of Book XI). 

\begin{theorem} \label{PlanesInterectInLines}
	$PL_{abc}(x) \wedge PL_{def}(x) \wedge PL_{abc}(y) \wedge PL_{def}(y) \wedge \lnot PL_{abc}(d) \rightarrow \widetilde{L}(x,y,p_{(abc,def,x)})$
\end{theorem}

\begin{proof}
	Let $p = p_{(abc,def,x)}$. Suppose $\lnot \widetilde{L}(x,y,p)$. Then $T(x,y,p)$ (note $PL_{abc}(x)$, $PL_{abc}(y)$ and $PL_{abc}(p)$). Since $PL_{def}(x)$, $PL_{def}(y)$, $PL_{def}(p)$ and $PL_{def}(d)$ we have $PL_{xyp}(d)$ by Axiom \ref{coplanartrans}. Since $PL_{abc}(x)$, $PL_{abc}(y)$, $PL_{abc}(p)$ and $PL_{xyp}(d)$, we have $PL_{abc}(d)$ by Theorem \ref{coplanartransbothways}. This is a contradiction. Thus $\widetilde{L}(x,y,p)$. 
	
\end{proof}

We will need the following two lemmas in order to prove the next theorem. 

\begin{lemma} \label{verticallemma1}
	If $abc$ and $a'b'c'$ are vertical angles [$B(a,b,a')$, $B(c,b,c')$, and $T(a,b,c)$], $T(a,b,x)$ and $T(c,b,x)$, then $Int(x, abc)$, $Int(x,abc')$, $Int(x,a'bc)$, or $Int(x,a'bc')$. 
\end{lemma}

\begin{proof}
	Note $T(x,b,a)$ and $T(c,b,a)$. Thus either $SS(x,c,ba)$ or $OS(x,c,ba)$. Likewise, since $T(x,b,c)$ and $T(a,b,c)$, either $SS(x,a,bc)$ or $OS(x,a,bc)$. There are then four possibilities: 
	
	Case 1: Suppose $SS(x,c,ba)$ and $SS(x,a,bc)$. By definition we $Int(x,a,bc)$. 
	
	Case 2: Suppose $OS(x,c,ba)$ and $SS(x,a,bc)$. Note $SS(x,a,bc')$ by Theorem \ref{sideslinethm2}. Also note that $OS(c',c,ba)$ by Axiom \ref{intersectOO}. Thus, by Theorem \ref{oppositesidethm} we have $SS(x,c',ba)$. By definition we then have $Int(x, abc')$. 
	
	Case 3: Suppose $SS(x,c,ba)$ and $OS(x,a,bc)$. Using similar methods to the previous case we can show $SS(x,c,ba')$ and $SS(x,a',bc)$ and this infer that $Int(x, a'bc)$. 
	
	Case 4: Suppose $OS(x,c,ba)$ and $OS(x,a,bc)$. Then we have $SS(x,c',ba)$ and $SS(x,a',bc)$ which implies $SS(x,c',ba')$ and $SS(x,a',bc')$. Thus we obtain $Int(x,a'bc')$.

\end{proof}

\begin{lemma}\label{verticallemma2}
	If $abc$ and $a'b'c'$ are vertical angles [$B(a,b,a')$, $B(c,b,c')$, and $T(a,b,c)$], $Int(x,abc)$ and $B(x,b,x')$, then $Int(x',a'bc')$. 
\end{lemma}

\begin{proof}
	Since $Int(x, abc)$ we have $SS(x,a,bc)$ and $SS(x,c,ba)$. By methods use in the previous proof we know $SS(x,a,bc)$ implies $OS(x',a,bc)$. Which in turn implies $SS(x',a',bc)$. Which finally implies $SS(x',a',bc')$. Similarly $SS(x,c,ba)$ implies $SS(x',c',ba')$. Thus $Int(x',a'bc')$. 
\end{proof}

We now turn our attention to the interplay between the orthogonal construction and our version of parallel segments. In particular we will be focusing on results related to Propositions 4 through 8  and Proposition 13 for Book XI of \textit{The Elements}. The following theorem has a corollary which is a converse of sorts to Axiom \ref{orthogonatorisright} and is closely related to Proposition 4 from Book XI of Euclid's Elements.

\begin{theorem} \label{Prep2PerpAllPlane}
	$abd \text{ is right} \wedge cbd \text{ is right} \wedge PL_{abc}(x) \wedge x \not= b \rightarrow xbd$ is right 
\end{theorem} 

\begin{proof}
	If $\widetilde{L}(b,a,x)$ or $\widetilde{L}(b,c,x)$ then we are done. Suppose $\lnot \widetilde{L}(b,a,x)$ and $\lnot \widetilde{L}(b,c,x)$. Let $e=lf(ba,bc)$. Since $PL_{abc}(x)$ we have $T(a,b,c)$. Thus, $T(e,b,c)$. Let $e'=doub(eb)$ and $c'=doub(cb)$. Then by Lemma \ref{verticallemma1} we have $Int(x,ebc) \vee Int(x, e'bc) \vee Int(x, cbc') \vee Int(x, e'bc')$. Without loss of generality let $Int(x,ebc)$. Also let $x'=doub(xb)$. by Lemma \ref{verticallemma2} $Int(x', e'bc')$. Let $y=cb(x,ebc)$ and $y'=cb(x', e'bc')$. 
	
	To summarize we have $PL_{abc}(c)$,  $PL_{abc}(c')$,  $PL_{abc}(e)$,  $PL_{abc}(e')$,  $PL_{abc}(y)$,  $PL_{abc}(y')$, $PL_{abc}(b)$, $ebc$ is vertical to $e'bc'$, $B(y,b,y')$, $B(c,y,e)$, $B(c',y',e')$, and $ebd$, $cbd$, $e'bd$, $c'bd$ are all right. The proof can then be finished off by slightly modifying Euclid's proof of Proposition 4 from Book XI.  
\end{proof}
 From the previous theorem we obtain the following corollary. 

\begin{corollary} \label{OrthogonalPerpAllPlane}
	$PL_{abc}(x) \wedge x \not= b \rightarrow xbo(abc) \text{ is right}$
\end{corollary} 

Euclid's Proposition 5 from Book XI is captured by our Axiom \ref{PerptoOrthogisCoplanar}. The next result is analogous to Proposition 13 in Book XI.  (It can be noted that Euclid should have proved Proposition 13 before Proposition 6.)

\begin{theorem}
	$T(a,b,c) \wedge abd \text{ is right} \wedge cbd \text{ is right} \rightarrow \widetilde{L}(b,d,o(abc))$ 
\end{theorem}

\begin{proof}
	Let $o=o(abc)$ and suppose $\lnot \widetilde{L}(b,d,o) \equiv T(b,d,o)$. Let $p = p_{(abc,obd, b)}$. Since $PL_{abc}(p)$, Theorem \ref{Prep2PerpAllPlane} implies $pbd$ is right.  Likewise by Corollary \ref{OrthogonalPerpAllPlane} we have $pbo$ is right. Since $PL_{obd}(b)$, $PL_{obd}(d)$, $PL{obd}(o)$, and $PL_{obd}(p)$ with  $pbd$ is right and $pbo$ is right  and $\lnot \widetilde{L}(b,d,o)$ we have a contradiction. Thus $\widetilde{L}(b,d,o(abc))$.
\end{proof}

The following Theorem is our analog to Proposition 6 from Book XI. This proof is an analogous version of Euclid's proof of Proposition 6, but the details of the proof are provided given there is enough of a distinction between our proof and Euclid's. 

\begin{theorem} \label{OrthogonalsPara}
	$PL_{abc}(d) \wedge PL_{abc}(e) \wedge PL_{abc}(f) \wedge T(d,e,f) \rightarrow bo(abc) \parallel eo(def)$
\end{theorem}

\begin{figure}[h!]
	\begin{picture}(216,120)
	\put(60,0){\includegraphics[scale=.9]{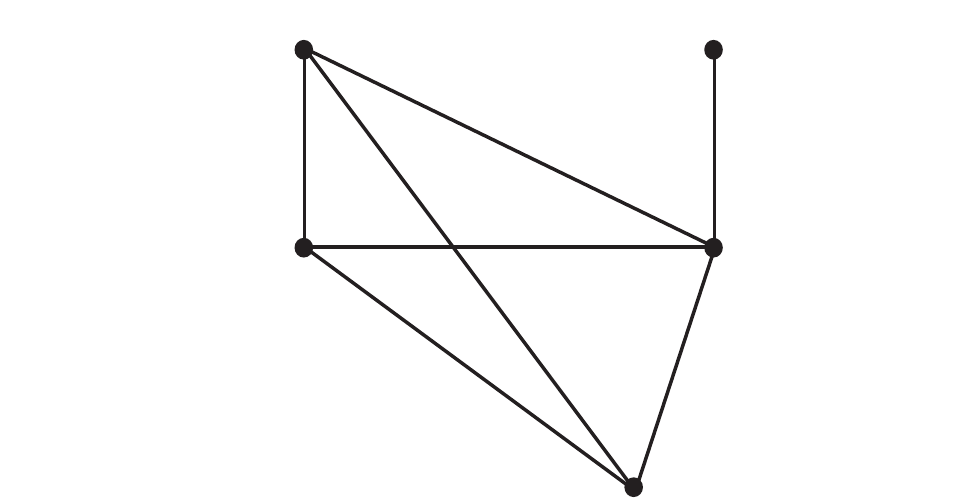}}
	
	\put(130,55){$b$}
	
	\put(130, 117){$o$}

	\put(230,1){$g$}

	\put(250,117){$o'$}

	\put(250,60){$e$}

	\end{picture}
	\caption{}
	\label{Prop6}
\end{figure}

\begin{proof}
	(Reference Figure \ref{Prop6}.) Let $g$ be the point uniformly constructed so that $beg$ is right, $eg \cong ab$, and $PL_{abc}(g)$. Let $o(abc)$ be $o$ and let $o(def)$ be $o'$. By Corollary \ref{OrthogonalPerpAllPlane} $gbo$ and $ebo$ are right. By Axiom \ref{coplanartrans} $PL_{def}(b)$ and $PL_{def}(g)$. Thus $beo'$ and $geo'$ are right. Since $bo \cong eg$, $eb \cong be$, and $obe \cong geb$, an application of Side-Angle-Side implies triangle $ebo$ is congruent to triangle $beg$. We can then infer $eo \cong bg$. Since $bg \cong eo$, $ob \cong ge$, and $go \cong og$, and application of Side-Side-Side implies triangle $obg$ is congruent to triangle $geo$. We can then infer $obg \cong geo$. Therefore $geo$ is right. By the previous theorem, we have $\widetilde{L}(e,g,o(beo'))$. Since $oeg$ is right, we have $oeo(beo')$ is right. Therefore $PL_{beo'}(o)$. By the Alternate Interior Angle Theorem [Theorem \ref{AIAT}], we have $bo \parallel bo'$. 
\end{proof}

In Book XI Euclid uses Proposition 7 to prove Proposition 8. The analog of Proposition 7 which is needed for us to prove out analog of Proposition 8 is the observation we made about Definition 18 when discussing modifications earlier in this section. This observation was the following implication: $ab \parallel cd \rightarrow PL_{abc}(d)$. The next Theorem is our analog to Proposition 8. 

\begin{theorem} \label{ParatoPlaneimpliesRight}
	$PL_{abc}(d) \wedge bo(abc) \parallel de \rightarrow [PL_{abc}(x) \rightarrow xde \text{ is right}]$
\end{theorem}

\begin{figure}[h!]
	\begin{picture}(216,120)
	\put(60,0){\includegraphics[scale=.9]{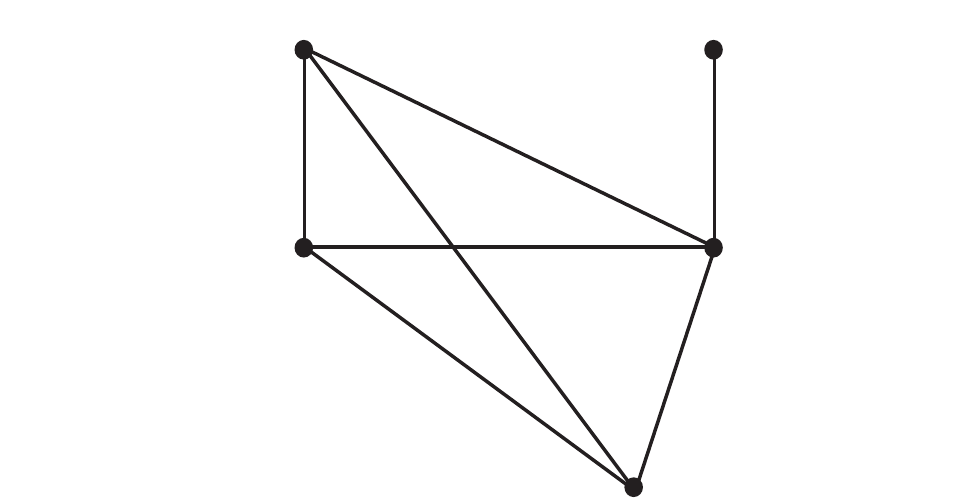}}
	
	\put(130,55){$b$}
	
	\put(130, 117){$o$}

	\put(230,1){$f$}

	\put(250,117){$e$}

	\put(250,60){$d$}

	\end{picture}
	\caption{}
	\label{Prop8}
\end{figure}

\begin{proof}
	(Reference Figure \ref{Prop8}.) Let $o = o(abc)$. Since $PL_{abc}(d)$, by Corollary \ref{OrthogonalPerpAllPlane} we have $dbo$ is right. By the observation made about Definition 18 (discussed before the statement of the theorem) and the Converse of the Alternate Interior Angle Theorem [Theorem \ref{converseAIAT}] we can infer $bde$ is right also. 
	
	Similar to the construction used in the proof of the previous theorem, construct a point $f$ such that $bdf$ is right, $df \cong bo$, and $PL_{abc}(f)$. Note $fbo$ is right. Using similar reasoning as in the proof of the previous theorem, we can use Side-Angle-Side to conclude triangle $fdb$ is congruent to triangle $obf$, thus obtaining $do \cong bf$. We can then use Side-Side-Side to conclude triangle $obf$ is congruent to triangle $fdo$, thus obtaining $fbo \cong fdo$. Therefore $fdo$ is right. By the observation made about Definition 18 we know $PL_{dbo}(e)$. Therefore by Theorem \ref{Prep2PerpAllPlane} we have $fde$ is right. 
	
	Since $PL_{abc}(b)$, $PL{abc}(f)$, $PL_{abc}(d)$ and $T(b,f,d)$ by Axiom \ref{coplanartrans} we have $PL_{bdf}(x)$. Since $bde$ and $fde$ are right, by Theorem \ref{Prep2PerpAllPlane} we have $xde$ is right. 
\end{proof}

Lastly we have a theorem which states something analogous to the idea that if two planes have parallel normals and share a point in common then they are the same plane. 

\begin{theorem}
	$bo(abc) \parallel eo(def) \wedge PL_{abc}(x) \wedge PL_{def}(x) \rightarrow [PL_{abc}(y) \rightarrow PL_{def}(y)]$
\end{theorem} 

\begin{proof}
	Since $T(a,b,c)$, then $T(a,b,x)$ or $T(a,x,c)$ or $T(x,b,c)$. Without loss of generality let $T(a,x,c)$. Likewise since $T(d,e,f)$, then $T(d,e,x)$ or $T(d,x,f)$ or $T(x,e,f)$. Without loss of generality let $T(d,x,f)$. By Theorem \ref{OrthogonalsPara} $bo(abc) \parallel xo(axc)$ and $eo(def) \parallel xo(dxf)$. Note $xo(axc) \nparallel  xo(dxf)$ since $x$ is an endpoint of both segments. Therefore by Lemma \ref{parallelimplieslinear} we have $\widetilde{L}(x, o(axc), o(dxf))$. Since $PL_{abc}(y)$ by Corollary \ref{OrthogonalPerpAllPlane} we have $ybo(abc)$ is right. By Theorem \ref{ParatoPlaneimpliesRight} we have $yxo(axc)$ is right. Since $\widetilde{L}(x, o(axc), o(dxf))$, we have $yxo(dxf)$ is right. By invoking Theorem \ref{ParatoPlaneimpliesRight} we have $yxo(def)$ is right. Therefore $PL_{def}(y)$. 
\end{proof}

From this theorem one can prove the following Corollary. 

\begin{corollary}
	$bo(abc) \parallel eo(def) \wedge PL_{abc}(x) \wedge \not PL_{def}(x) \rightarrow \lnot [PL_{abc}(y) \wedge PL_{def}(y)]$
\end{corollary}

\subsection{Sides of a Plane} 

In Theorem 7 of \cite{HilbertFoundation}, Hilbert states a separation property for space in which a plane partitions all points not on the plane into two sets called sides. Two points are said to be on the same side of the plane if the segment between them does not intersect the plane and are said to be on opposite sides if the segment between them does intersect the plane. We will develop our analog results by first defining what it will mean for two points to be on opposite sides of three non coplanar points $a$, $b$, and $c$. This approach differs from how angle orientation and sides of a line were defined in Section \ref{Pasch} where same orientation/side was defined before opposite orientation/side.

\begin{definition} \label{OSPlanDef}
	$OS(x,y,abc) \equiv T(a,b,c) \wedge T(x,b,y) \wedge \lnot PL_{abc}(x) \wedge \lnot PL_{abc}(y) \wedge OS(x,y,bp)$ (where $p=p_{(abc,xby,b)}$ )
\end{definition}

We then define what is meant by two points being on the same side of three non coplanar points in terms of the previous definition. 
\begin{definition}
	$SS(x,y,abc) \equiv T(a,b,c) \wedge \lnot PL_{abc}(x) \wedge \lnot PL_{abc}(y) \wedge \lnot OS(x,y,abc) $
\end{definition}
Note that is may be the case $\widetilde{L}(x,b,y)$. 

The following theorem shows that the coplanar relation in some sense preserves the planar opposite sides relation. 

\begin{theorem} \label{OSTrans}
	$PL_{abc}(d) \wedge PL_{abc}(e) \wedge PL_{abc}(f) \wedge T(d,e,f) \wedge OS(x,y,abc) \rightarrow OS(x,y,def)$
\end{theorem}

\begin{proof}
	
	We need to show $\lnot PL_{def}(x)$ and $\lnot PL_{def}(y)$. If these were false, then we could use Theorem \ref{coplanartrans} to show $\lnot PL_{abc}(d)$ and  $\lnot PL_{abc}(e)$ which would be a contradiction. Thus all we have left to show is $T(x,e,y)$ and $OS(x,y,ep')$ where $p'=p(def,xey,e)$. 
	
	Since $OS(x,y,abc)$, we have $OS(x,y,bp)$. Let $z= cb(p,xby)$. Then $B(x,z,y)$ and $\widetilde{L}(b,z,p)$. Since $\widetilde{L}(b,z,p)$, by Theorem \ref{ColinearCoplanar} we have $Pl_{abc}(z)$. By Axiom \ref{coplanartrans} we obtain $PL_{def}(z)$. We claim $\lnot \widetilde{L}(x,z,e)$. 
	
	If $\widetilde{L}(x,z,e)$, then $PL_{def}(x)$ by Theorem \ref{ColinearCoplanar}. This is a contradiction. Thus we have  $\lnot \widetilde{L}(x,z,e)$. This implies $T(x,e,y)$ and by Axiom \ref{intersectOO} we also have $OS(x,y,ez)$. 
	
	Note $PL_{xey}(z)$ since $L(x,z,y)$. Also recall $PL_{def}(z)$. Thus by Theorem \ref{ColinearCoplanar} we have $\widetilde{L}(e,p',z)$. Then by Theorem \ref{oppositesidethm2} we obtain $OS(x,y,ep')$. Finally by Definition \ref{OSPlanDef} we have $OS(x,y,def)$. 
\end{proof}

We now justify an analogous theorem for the planar same side relation. 

\begin{theorem}
$PL_{abc}(d) \wedge PL_{abc}(e) \wedge PL_{abc}(f) \wedge T(d,e,f) \wedge SS(x,y,abc) \rightarrow SS(x,y,def)$
\end{theorem}

\begin{proof}
	By Theorem \ref{coplanartransbothways} we have $PL_{def}(a)$, $PL_{def}(b)$ and $PL_{def}(c)$. By Axiom \ref{coplanarnotcolinear} we have $T(a,b,c)$. Suppose $OS(x,y,def)$. Then by the previous theorem we would have $OS(x,y,abc)$. This is a contradiction. Thus $\lnot OS(x,y,def)$. Therefore by Definition \ref{OSPlanDef} we have $SS(x,y,def)$. 
\end{proof}

The two previous theorem are planar analogs for Theorem \ref{sideslinethm2} and Theorem \ref{oppositesidethm2}. The following Theorem is an analog to Theorem \ref{oppositesidethm}. 

\begin{theorem} \hspace{.1in}
	\begin{enumerate}
		\item $OS(x,y,abc) \rightarrow OS(y,x,abc)$
		\item $OS(x,y,abc) \rightarrow OS(x,y,bac) \wedge OS(x,y,cba)$
		\item $OS(x,y,abc) \wedge OS(y,z,abc) \rightarrow SS(x,z,abc)$
	\end{enumerate}
\end{theorem}

\begin{proof}
	Part 1) is justified by part 1) of Theorem \ref{oppositesidethm}. Part 2) is justified with Axiom \ref{coplanarnotcolinear}, Axiom \ref{bcoplanar}, Theorem \ref{accoplanar}, Theorem \ref{COPlanarPermute} and Thereom \ref{OSTrans}. 
	
	The proof of Part 3) is a bit more work. we have $T(a,b,c)$ from $OS(x,y,abc)$. Also since $OS(x,y,abc)$ we have $\lnot PL_{abc}(x)$. Similarly since $OS(y,z,abc)$ we have $\lnot PL_{abc}(z)$. Thus all we need to prove is $\lnot OS(x,z,abc)$. Suppose $OS(x,z,abc)$. Since $OS(x,z,abc)$, we have $T(x,b,z)$ and $OS(x,z,bp_{(abc,xbz,b)})$. Likewise we have $T(x,b,y)$, $OS(x,y,bp_{(abc,xby,b)})$, $T(y,b,z)$, and $OS(y,z,bp_{(abc,ybz,b)})$. Let $u=cb(p_{(abc,xby,b)},xby)$, $v=cb(p_{(abc,ybz,b)},ybz)$, and $w=cb(p_{(abc,xbz,b)}, xbz)$. There are two cases to consider.
	
	Case 1) Let $T(x,y,z)$. Since $B(x,u,y)$, $B(y,v,z)$, and $B(x,w,z)$, we have $PL_{xyz}(u)$, $PL_{xyz}(v)$, $PL_{xyz}(w)$ by Theorem \ref{ColinearCoplanar}. One can also show $PL_{abc}(u)$, $PL_{abc}(v)$, $PL_{abc}(w)$. Thus by Theorem \ref{PlanesInterectInLines} we can infer $\widetilde{L}(u,v,w)$. If $u=v$, $u=w$ or $v=w$, then $\lnot T(x,y,z)$. This is  contradiction. If $u \not= v$, $u \not= w$, and $v \not= w$, then by Lemma \ref{orientlemma3} we have $T(u,v,w)$. This is a contradiction. Thus $\lnot OS(x,z,abc)$. 
	
	Case 2) Let $\widetilde{L}(x,y,z)$. By Theorem \ref{ColinearCoplanar} we can infer $PL_{xby}(z)$. One can then infer that $b$, $u$, $v$, and $w$ are all non-strict colinear by Theorem \ref{PlanesInterectInLines}. By using Axiom \ref{uniqueline} we can conclude $u=v=w$. Using Theorem \ref{oppositesidethm} and Theorem \ref{oppositesidethm2} we can make the following implications: $OS(x,y,bw)$, $OS(y,z,bw)$, and $SS(x,z,bw)$. But this would imply $SS(x,z,bp_{(abc,xbz,b)})$ which is a contradiction. 	
\end{proof}

The next Theorem is an analog to Theorem \ref{sideslinethm1}. 

\begin{theorem} \hspace{.1in}
	\begin{enumerate}
		\item $T(a,b,c) \wedge \lnot PL_{abc}(x) \rightarrow SS(x,x,abc)$
		\item $SS(x,y,abc) \rightarrow SS(y,x,abc)$
		\item $SS(x,y,abc) \rightarrow SS(x,y,bac) \wedge SS(x,y,cba)$
		\item $SS(x,y,abc) \wedge SS(y,z,abc) \rightarrow SS(x,z,abc)$
	\end{enumerate}
\end{theorem}

\begin{proof}
	To justify part 1) we only need to show $\lnot OS(x,x,abc)$. Note that $\lnot T(x,x,b)$. Thus $\lnot OS(x,x,abc)$. The proofs of part 2) and 3) are omitted. 
	
	Thus we turn to the proof of part 4). Since $SS(x,y,abc)$ we have $T(a,b,c)$. Also since $SS(x,y,abc)$ we have $\lnot PL_{abc}(x)$ and since $SS(y,z,abc)$ we have $\lnot PL_{abc}(z)$. Thus we only need to show $\lnot OS(x,z,abc)$. We consider four cases. 
	
	Case 1) Let $\widetilde{L}(x,y,b)$ and $\widetilde{L}(y,z,b)$. One can show $\widetilde{L}(x,b,z)$. Thus $\lnot T(x,b,z)$ and therefore $\lnot OS(x,z,abc)$. 
	
	Case 2) Let $\widetilde{L}(x,y,b)$ and $\lnot \widetilde{L}(y,z,b)$. Since $T(y,b,z)$, we can construct $p_{(abc,ybz,b)}$. Given $SS(y,z,abc)$ we can infer $SS(y,z,bp_{(abc,ybz,b)})$. Since $\widetilde{L}(x,y,b)$ we can conclude $SS(x,z,bp_{(abc,ybz,b)})$ and $PL_{yzb}(x)$. Given $PL_{yzb}(x)$ we can conclude $PL_{yzb}(p_{(abc,xzb.b)})$ and by using Theorem \ref{PlanesInterectInLines} infer $\widetilde{L}(b,p_{(abc,ybz,b)},p_{(abc,xzb.b)})$. Therefore we have $SS(x,z,bp_{(abc,xzb.b)})$ which implies $\lnot OS(x,z,abc)$. 
	
	Case 3) Let $\lnot \widetilde{L}(x,y,b)$ and $\widetilde{L}(y,z,b)$. The justification is similar to Case 2. 
	
	Case 4) Let $\lnot \widetilde{L}(x,y,b)$ and $\lnot \widetilde{L}(y,z,b)$. If $L(x,z,b)$ then $\lnot OS(x,z,abc)$ and we are done. Suppose $\lnot L(x,z,b) \equiv T(x,b,z)$. Note we can construct the following points: $p_{xy}=p{(abc,xby,b)}$, $p_{yz} = p_{(abc,ybx,b)}$, and $p_{xz}=p_{(abc,xbz,b)}$. Also observe that we have $SS(x,y,bp_{xy})$ and  $SS(y,z,bp_{yz})$. Let $u=cb(p_{xz}, xbz)$. Note that $PL_{abc}(u)$. 
	
	Suppose $OS(y,z,bu)$. Then $cb(u,ybz)$ is (non-strict) colinear with $b$ and $p_{yz}$. This would imply $OS(y,z,bp_{yz})$ which is a contradiction. Thus $\lnot OS(y,z,bu)$ and since $T(y,z,b)$ and $T(y,z,u)$ (because $PL_{abc}(u)$) we can conclude $SS(y,z,bu)$. Similarly we can conclude $SS(x,y,bu)$. We can then infer $SS(x,z,bu)$. From this we can then show $SS(x,z,bp_{xz})$. Finally we can conclude $\lnot OS(x,z,abc)$. 
\end{proof}

The following theorem is provable from parts 2) and 4) of the previous theorem. 

\begin{theorem}
	\item $SS(x,y,abc) \wedge OS(y,z,abc) \rightarrow SS(x,z,abc)$
\end{theorem}

\section{Conclusion}

In \cite{VBFinitism}, Van Bendegem ask if Suppes' quantifier-free axioms for constructive affine plane geometry could be expanded all the way into a full-fledged geometric theory. We claim this work answers that question in the affirmative. Furthermore we claim that this work is the first work to build a full geometric theory which contains only feasible construction and who's concepts are in line with the bounded experience of real-world constructions.

\end{document}